\documentclass[a4paper]{amsart}

\usepackage[utf8]{inputenc}
\usepackage[english]{babel}

\usepackage{pgf,tikz,pgfplots}
\usetikzlibrary{arrows}
\usetikzlibrary[patterns]

\usepackage{amsmath,amscd}
\usepackage{amssymb}
\usepackage{amsthm}
\usepackage{mathrsfs}

\usepackage[left=3cm,right=3cm,top=4.5cm,bottom=4.5cm]{geometry}

\usepackage{verbatim}

\usepackage[ocgcolorlinks, linkcolor=blue]{hyperref}
             
\newtheorem{Theorem}{Theorem}[section]

\newtheorem{Lemma}[Theorem]{Lemma}

\newtheorem{Proposition}[Theorem]{Proposition}

\theoremstyle{definition}

\numberwithin{equation}{section}

\newcommand{\ma}{\mathcal}
\newcommand{\la}{\lambda}

\newcommand{\N}{\mathbb{N}}   
\newcommand{\R}{\mathbb{R}}
\newcommand{\C}{\mathbb{C}}

\newcommand{\A}{\mathbf{A}} 
\newcommand{\W}{\mathbf{W}} 
\newcommand{\E}{\mathbf{E}} 
\newcommand{\ve}{\vartheta}   
\newcommand{\M}{m}  

\DeclareMathOperator{\dist}{dist}
\DeclareMathOperator{\supp}{\mathrm{supp}}
\let\Re\relax
\DeclareMathOperator{\Re}{\mathrm{Re}}
\let\Im\relax
\DeclareMathOperator{\Im}{\mathrm{Im}}

\newcommand{\abs}[1]{\lvert #1 \rvert}          
\newcommand{\norm}[1]{\lVert #1 \rVert} 
\newcommand{\br}[1]{\langle #1 \rangle} 
\newcommand{\ol}[1]{\overline{#1}}

\newcommand{\ebreve}{\,\breve{\rule{0pt}{6pt}}\,}

\newcommand{\p}{\partial}
\newcommand{\eps}{\varepsilon}

\newcommand{\bk}[1]{\left\{ #1 \right\}}

\newcommand{\doi}[1]{\href{https://doi.org/#1}{#1}} %

\newcounter{sidenote}
\begin{document}

\title[The fixed angle scattering problem]{The fixed angle scattering problem with a first order perturbation} 

\author[C. J. Meroño]{Cristóbal J. Meroño}
\address{Universidad Politécnica de Madrid, ETSI Caminos, Departmento de Matemática e Informática, Campus Ciudad Universitaria, Calle del Prof. Aranguren, 3, 28040 Madrid}
\email{cj.merono@upm.es}
\author[L. Potenciano-Machado]{Leyter Potenciano-Machado}
\address{University of Jyvaskyla, Department of Mathematics and Statistics, PO Box 35, 40014 University of Jyvaskyla, Finland}
\email{leyter.m.potenciano@jyu.fi}
\author[M. Salo]{Mikko Salo}
\address{University of Jyvaskyla, Department of Mathematics and Statistics, PO Box 35, 40014 University of Jyvaskyla, Finland}
\email{mikko.j.salo@jyu.fi}



\begin{abstract}
We study the inverse scattering problem of determining a magnetic field and electric potential from scattering measurements corresponding to finitely many plane waves. The main result shows that the coefficients are uniquely determined by $2n$ measurements up to a natural gauge. We also show that one can recover the full first order term for a related equation having no gauge invariance, and that it is possible to reduce the number of measurements if the coefficients have certain symmetries. This work extends the fixed angle scattering results of \cite{RakeshSalo1, RakeshSalo2} to Hamiltonians with first order perturbations, and it is based on wave equation methods and Carleman estimates.
\end{abstract}

\maketitle


\section{Introduction and Main Theorems}

In this work we study the inverse scattering problem of recovering a first order perturbation from fixed angle scattering measurements.
Let $\lambda > 0$,  $n\geq 2$ and let $\omega \in S^{n-1}$ be a fixed unit vector. Suppose that for $\M \in  \N$, $\mathcal V(x,D)$ is  a first order differential operator with $C^\M(\R^n)$ coefficients having  compact support  in $B = \{x\in \R^n:|x|<1\}$, the open ball of radius $1$.  We consider a Hamiltonian  $H_{\mathcal V} = -\Delta + {\mathcal V}(x,D)$ in $\R^n$ and the problem
\begin{equation} \label{eq:sta_sct}
\begin{cases}
(H_{\mathcal V} -\lambda^2) \psi_{\mathcal V} = 0    \\
\psi_{\mathcal V}(x,\lambda,\omega) = e^{i\lambda \omega \cdot x} + \psi^s_{\mathcal V}(x,\lambda,\omega),
\end{cases}
\end{equation}
where $\psi^s_{\mathcal V}(x,\lambda,\omega)$ is known as a the {\it scattering solution}.
It is well known that in order to have uniqueness for this problem one needs to put further restrictions on the function $\psi_{\mathcal V}^s$. See e.g.\ \cite{Yafaev} for the following facts. The   {\it outgoing Sommerfeld Radiation Condition} (SRC for short)
\[
\partial_r \psi_{\mathcal V}^s -i \lambda \psi_{\mathcal V}^s  = o(r^{-(n-1)/2}) \quad \text{as } r \to \infty,
\]
where $r = |x|$, selects the  solutions that heuristically behave as  Fourier transforms in time of spherical waves that propagate towards infinity.  A function  $\psi^s_{\ma V}$ satisfying \eqref{eq:sta_sct} and the SRC is called an   outgoing scattering solution.
This solution is given by the so called outgoing resolvent operator
\begin{equation} \label{eq:res}
R_{\mathcal V}(\lambda) : = (H_{ \mathcal V} - (\lambda + i0)^2)^{-1} ,
\end{equation}
so that, formally,
\[
\psi_{\mathcal V}^s = R_{\mathcal V}(\lambda)(- \mathcal V(x,D)e^{i\lambda \omega \cdot x}).
\]
Notice that, assuming that such a solution $\psi^s_{\mathcal V}$ exists---which would happen if the resolvent is well defined and  bounded in appropriate spaces---it must satisfy the Helmholtz equation
\[(-\Delta -\lambda^2) \psi^s_{\mathcal V} = 0 \quad \text{in} \; \; \R^n \setminus \ol{B}, \]
  since the coefficients of ${\mathcal V}(x,D)$ are compactly supported in ${B}$. It is well known  that a solution of Helmholtz equation satisfying the $SRC$ has always the asymptotic expansion
\begin{equation} \label{eq:asymp_psi}
\psi^s _{{\mathcal V}}(x,\lambda,\omega) = e^{i\lambda |x|} |x|^{-\frac{n-1}{2}} a_{\mathcal V}(\lambda,\theta,\omega) + o(|x|^{-\frac{n-1}{2}}), \quad \text{as} \;\; |x|\to \infty,
\end{equation}
where  $\theta = \frac{x}{|x|}$ and $ a_{\mathcal V}(\lambda,\theta,\omega)$ is called the {\it scattering amplitude} or {\it far field pattern}.  

In this setting, the main objective of an inverse scattering problem  consists in reconstructing the coefficients of $\mathcal V(x,D)$ from partial or full knowledge of $a_{\mathcal V}(\lambda,\theta,\omega)$. Depending on the data that is assumed to be known we can distinguish several types of inverse scattering problems:
\begin{itemize}
\item[1.] {\it Full data}. Recover the coefficients of $\mathcal V(x,D)$ from the knowledge of $a_{\mathcal V}(\lambda,\theta,\omega)$ for all $(\lambda, \theta, \omega)\in (0,\infty) \times S^{n-1}\times S^{n-1}$. 
\item[2.] {\it Fixed frequency} (or fixed energy). Recover the coefficients of $\mathcal V(x,D)$ from the knowledge of $a_{\mathcal V}(\lambda_0,\theta,\omega)$ for a fixed $\lambda_0>0$ and all $(\theta, \omega)\in S^{n-1}\times S^{n-1}$. 
\item[3.] {\it Backscattering}. Recover the coefficients of $\mathcal V(x,D)$ from the knowledge of $a_{\mathcal V}(\lambda,\omega,-\omega)$ for all $(\lambda, \omega)\in (0,\infty)  \times S^{n-1}$. 
\item[4.] {\it Fixed angle} (single measurement). Recover the coefficients of $\mathcal V(x,D)$ from the knowledge of $a_{\mathcal V}(\lambda,\theta,\omega_0)$ for a fixed $\omega_0\in S^{n-1}$ and all $(\lambda, \theta)\in (0,\infty)  \times S^{n-1}$.
\end{itemize}
In the case of fixed angle scattering it  is also interesting to consider analogous inverse problems in which $a_{\mathcal V}(\, \cdot \,, \, \cdot \,, \omega)$ is assumed to be known for each $\omega$ in a fixed  subset (usually finite) of $S^{n-1}$. 

Let $D = -i\nabla $. We consider the Hamiltonian
\[ H_{\mathcal V} = H_{\A,q} =(D+ \A)^2 +q= -\Delta + 2\A \cdot D+ D\cdot \A+ \A^2+ q ,\]
 where  both  the magnetic potential $\A$, and the electrostatic potential $q$  are real. Then $H_{\ma V}$ is self-adjoint, and if $\A$ and $q$ are compactly supported, the resolvent \eqref{eq:res} is bounded in appropriate spaces under very general assumptions on the regularity of $\A $ and $q$. This implies   that the problem \eqref{eq:sta_sct} has a unique solution $\psi^s_{\mathcal V}$ and hence that the scattering amplitude $a_{\mathcal V} =a_{\A,q}$ is well defined, so that the fixed angle scattering  problem can be appropriately stated.
 
 In \cite{RakeshSalo1, RakeshSalo2} it has been proved that for $H_\ma V = -\Delta +q$, knowledge of the fixed angle scattering data $a_{\mathcal V}(\, \cdot \, , \, \cdot \, , \omega)$ in two opposite directions $\omega = \pm \omega_0$ for a fixed $ \omega_0 \in  S^{n-1}$, determines uniquely  the potential $q$. The present article extends the results of \cite{RakeshSalo1, RakeshSalo2} to the case of non-vanishing first order coefficients and proves that from $2n$ measurements, or just $n+1$ measurements under symmetry conditions, one can determine both the first and zeroth order coefficients up to natural gauges. To prove these results, we follow the approach used in \cite{RakeshSalo2}, that is we show the equivalence of the fixed angle scattering problem with an appropriate inverse problem for the wave equation. This inverse scattering problem in time domain   consists in  recovering information on $\A$ and $q$ from boundary   measurements of the solution $U_{\A,q}$ of the initial value problem
\begin{equation} \label{eq:direct}
(\p_t^2 + H_{\A,q})U_{\A,q} =0 \; \; \text{in}\; \R^{n+1}, \qquad U_{\A,q}|_{\left\{t<-1\right\}} = \delta (t- x \cdot \omega).
\end{equation}
If the support of $\A$ and $q$ is contained in $B$, the boundary measurements of $U_{\A,q}$ are made in the set $ \partial B \times (-T,T) \cap \{(x,t)  : t \ge x\cdot \omega\}$,
where  $\partial B$ denotes the boundary of the ball.

We  now describe some previous results on the   inverse scattering problem of recovering a potential $q(x)$ from fixed angle measurements. As discussed above, this problem can be considered in the frequency domain, as the problem of determining $q$ from the scattering amplitude $a_q(\, \cdot \,, \, \cdot \,, \omega)$ for the Schr\"odinger operator $-\Delta + q$ with a fixed direction $\omega \in S^{n-1}$, or alternatively in the time domain as the problem of recovering $q$ from boundary or scattering measurements of the solution $U_q$ of the wave equation. The equivalence of these problems is discussed in  \cite{RakeshSalo2} (see \cite{Melrose,  Uhlmann_backscattering,  MelroseUhlmann_bookdraft} for the odd dimensional case).

The one-dimensional case is quite classical, see \cite{Marchenko, DeiftTrubowitz}.
In dimensions $n \geq 2$ uniqueness has been proved for small or generic potentials \cite{Stefanov, BarceloEtAl}, recovery of  singularities results are given in \cite{Ruiz,Meronno}, and uniqueness of the zero potential is considered in \cite{BaylissLiMorawetz}. Recently in \cite{RakeshSalo1, RakeshSalo2} it was proved that measurements for two opposite fixed angles uniquely determine a potential $q \in C^{\infty}_c(\R^n)$. The problem with one measurement remains open, but \cite{RakeshSalo1, RakeshSalo2} prove uniqueness for symmetric or horizontally controlled potentials (similar to angularly controlled potentials in backscattering \cite{RakeshUhlmann}), and Lipschitz stability estimates are given for the wave equation version of the problem. We also mention the recent work \cite{MaSalo} which studies the fixed angle problem when the Euclidean metric is replaced by a Riemannian metric, or sound speed, satisfying certain conditions, and the upcoming work \cite{KrishnanRakeshSemanapati} which studies fixed angle scattering for time-dependent coefficients also in the case of first order perturbations.

We now introduce the main results in this work. Since the metric is Euclidean, the vector potential $\A$  can equivalently be seen as a $1$-form $\A= A^j dx_j$. We denote by $d\A$ the exterior derivative of $\A$. 
Our first result shows that the magnetic field $d\A$ and the electrostatic potential $q$ are uniquely determined by the knowledge of the fixed angle scattering amplitude $a_{\A,q}(\, \cdot \,, \, \cdot \,, \omega)$ for  $n$ orthogonal directions $\omega = e_j$, $1\le j \le n$  and the  $n$ opposite ones, $\omega = -e_j$.  From now on, in this paper we will fix $\M$ to be the integer
\begin{equation} \label{eq:m}
 \M = \frac{3}{2} n +10 \; \; \; \text{if } n \text{ is even}, \quad \M =\frac{3}{2}(n+1) +10  \; \;\;\text{if } n \text{ is odd}.
\end{equation}
In general we consider $\A \in C^{m+2}(\R^n; \R^n)$ and $q \in C^m(\R^n; \R)$. This is required in order to guarantee that the solutions of \eqref{eq:direct} satisfy certain  regularity properties.
\begin{Theorem} \sl  \label{thm:main_2n_stationary}
Let $n\ge 2$ and $\lambda_0>0$, and let $e_1,\dots e_n$ be any orthonormal basis in $\R^n$. Assume  that the pairs of potentials $\A_1,\A_2 \in C^{\M+2}_c(\R^n; \R^n)$  and $q_1,q_2  \in  C^{\M}_c(\R^n; \R)$ are compactly supported in $B$. Assume also that   the following condition holds:
 \begin{equation}  \label{eq:A_condition}
\int_{-\infty}^\infty  e_n \cdot \A_k(x_1,\dots,x_{n-1},s) \, ds = 0 \quad \text{for} \;  k=1,2 \; \,  \text{and for every} \;(x_1,\dots,x_{n-1}) \in \R^{n-1}.
 \end{equation}
If for all $\theta \in S^{n-1}$ and $\lambda \ge \la_0$ we have
\[a_{ \A_1,q_1} (\lambda,\theta, \pm e_j) = a_{ \A_2,q_2}(\lambda,\theta,\pm e_j)  \quad \text{for all } j =1,\dots n, \]  
 then  $d \A_1 = d \A_2$ and $q_1 = q_2$.
\end{Theorem}
\noindent The condition \eqref{eq:A_condition} is a technical restriction necessary to decouple the information on $q$ from the information on $\A$ at some point in the proof of this uniqueness result.

In Theorem \ref{thm:main_2n_stationary} one cannot recover completely the magnetic potential $\A$ due to the phenomenon of gauge invariance. This consists simply in the observation that if $H_{\A,q} u = v$ for some functions $u$ and $v$, then $H_{\A + \nabla f,q} \, (e^{ -if} u) = e^{-if}v$, for any $f\in C^2(\R^n)$. Therefore if $f$ is compactly supported, the scattering amplitude  is not going to be affected by $f$, so that $a_{\A,q} = a_{\A + \nabla f,q}$. On the other hand if we consider Hamiltonians 
\[ \tilde{H}_{\A,V} := -\Delta -2i \A \cdot D + V,\]
 where $V$ is a fixed function, then the gauge invariance is broken and knowledge of the associated scattering amplitude $ \tilde{a}_{\A,V}(\, \cdot \, , \,\cdot  \,,\omega)$ for the $2n$   directions $\omega = \pm e_j$ determines completely $\A$.

\begin{Theorem} \sl    \label{thm:main_drif2n_stationary}
Let $n\ge 2$ and $\lambda_0>0$, and let $e_1,\dots e_n$ be any orthonormal basis in $\R^n$. Assume that $\A_1,\A_2 \in C^{\M+2}_c(\R^n; \R^n)$ and  $V \in C^{\M}_c(\R^n; \C)$ have compact  support in $B$, and that the Hamiltonians 
$\tilde{H}_{\A_1,V} $  and $\tilde{H}_{\A_2,V} $ are both self-adjoint operators.  \\ 
If for all $\theta \in S^{n-1}$ and $\lambda \ge \la_0$ we have
\[\tilde{a}_{ \A_1,V} (\lambda,\theta, \pm e_j) = \tilde{a}_{ \A_2,V}(\lambda,\theta,\pm e_j) \quad \text{for all } j =1,\dots n, \]  
 then  $\A_1 =  \A_2$.
\end{Theorem}
\noindent Notice that in this statement \eqref{eq:A_condition} is not assumed. This is related to the fact that $V$ is  fixed, so it is not necessary to decouple $V$ from $\A$ in the proof.  Similarly, if $\A$ is a fixed vector potential, it would be possible to determine $q$ from the knowledge of $a_{\A,q}(\,\cdot\,,\,\cdot\,,\pm \omega)$ for a fixed $\omega \in S^{n-1}$.

In both the previous theorems we need to measure the scattering amplitude generated by a wave incoming from  $2n$ different directions. In some cases we can avoid the need of sending a wave also from the opposite direction provided we assume there are certain symmetries in the potentials. As an example of this phenomenon we state the following result.
\begin{Theorem} \sl   \label{thm:main_n_stationary}
Let $n\ge 2$ and $\lambda_0>0$, and let $e_1,\dots e_n$ be any orthonormal basis in $\R^n$. Assume  that the pairs of potentials $\A_1,\A_2 \in C^{\M+2}_c(\R^n; \R^n)$  and $q_1,q_2  \in  C^{\M}_c(\R^n; \R)$ are compactly supported in $B$. Assume also that   
 \begin{equation} \label{eq:antysimm} 
\A_k(-x) = -\A_k(x) , \quad \text{for} \quad k=1,2. 
 \end{equation}
If for all $\theta \in S^{n-1}$ and $\lambda \ge \la_0$ we have
\[
\begin{aligned}
 a_{ \A_1,q_1} (\lambda,\theta,   e_j) &= a_{ \A_2,q_2}(\lambda,\theta, e_j)  \quad \text{for all } j =1,\dots n-1,  \text{ and}   \\  
a_{ \A_1,q_1} (\lambda,\theta, \pm e_n) &= a_{ \A_2,q_2}(\lambda,\theta,\pm e_n),
 \end{aligned}
 \]
then  $d \A_1 = d \A_2$ and $q_1 = q_2$.
\end{Theorem}
\noindent We assume that $a_{ \A_1, q_1} (\lambda,\theta, \omega) = a_{ \A_2,q_2}(\lambda,\theta,\omega)$ for $\omega = \pm e_n$ instead of just $\omega = e_n $ since we have not considered any symmetry on the potential $q$ (a result assuming symmetries on $q$ to reduce further the data could also be proved modifying slightly the arguments used to prove this theorem). We also prove in time domain  a more technical  result analogous  to Theorem \ref{thm:main_drif2n_stationary} that requires just $n$ measurements instead of $2n$, and that is compatible with less restrictive symmetry conditions  than \eqref{eq:antysimm} (see Theorem  \ref{thm:n_drift_nogauge} below).

In all the previous theorems we have considered an orthonormal basis $\{e_1, e_2 ,\dots e_n\}$ of $\R^n$ in order to simplify the notation and computations in some parts of the arguments, but we remark that our proofs  can be easily adapted to allow also  non-orthonormal directions of measurements.

As already mentioned, the previous theorems follow from corresponding results for the time domain inverse problem (Theorems \ref{thm:2n_magnetic}, \ref{thm:2n_drift_nogauge}, and \ref{thm:n_magnetic}, respectively). We now state the precise result that establishes the equivalence between the inverse scattering problem in frequency domain and  the inverse scattering problem in time domain, extending the results in \cite{RakeshSalo2} to first order perturbations.

\begin{Theorem} \sl  \label{thm:equivalence}
Let $n\geq 2$,  $\omega\in S^{n-1}$, and $\lambda_0>0$. Assume that $\A_1,\A_2 \in C^{\M+2}_c(\R^n; \R^n)$ and $q_1,q_2  \in  C^{\M}_c(\R^n; \R)$ are supported in $B$. For $k=1,2$, let $U_{\A_k, q_k}(x,t;\omega)$ be the unique distributional solution of the initial value problem
   \begin{equation*}
( \partial_t^2 + (D +\A_k )^2 +  q_k )U_{\A_k, q_k} = 0  \; \, \text{in}\; \; \R^{n+1}, \quad U_{\A_k, q_k} |_{\left\{t<-1\right\}} = \delta (t- x \cdot \omega).
\end{equation*}
 Then one has that
\[
a_{\A_1, q_1}(\lambda, \theta, \omega)= a_{\A_2, q_2}(\lambda, \theta, \omega)\; \; \text{for}\; \;\lambda \geq \lambda_0 \; \;\text{and}\; \;\theta \in S^{n-1},
\]
if and only if
\[
U_{\A_1, q_1}(x,t;\omega)= U_{\A_2, q_2}(x,t ;\omega) \; \; \text{for all}\;\; (x,t)\in (\partial B \times \R) \cap \left\{ t \ge x\cdot \omega\right\}.
\]
\end{Theorem}
\noindent  We remark that the restriction of the distribution $U_{\A_k,q_k}$ to the surface $\partial B \times \R$ is always well defined and vanishes in the open set $(\partial B \times \R) \cap \left\{ t <x\cdot \omega\right\}$. This can be seen from the explicit formula for $U_{\A_k,q_k}$ that we will compute in section \ref{sec:direct_scat_prob}.

The proof of the main results in time domain is based on a Carleman estimate method introduced in \cite{RakeshSalo1, RakeshSalo2}, which in turn adapts the method introduced in \cite{BukhgeimKlibanov} (see \cite{ImanuvilovYamamoto, Klibanov, BellassouedYamamoto} for more information and references on the Bukhgeim-Klibanov method).
 Essentially, the Carleman estimate is applied to the difference of two solutions of \eqref{eq:direct}. The general idea is to choose an appropriate Carleman weight function for the wave operator that is large on the surface $\{ t = x \cdot \omega \}$ and allows one to control a source term on the right hand side of the equation. Then one needs an additional energy estimate to absorb the error coming from the source term. This will allow one to control the difference of the potentials $\A_1-\A_2$ or $q_1-q_2$. This step is the key to get the uniqueness result. 
 Unfortunately, after doing all this, there is  a remaining  boundary term in the Carleman estimate that cannot be appropriately controlled.  However, this term can be canceled  using an equivalent Carleman estimate for solutions of the wave equation coming from the opposite direction. 
 This is why we require $2n$ measurements to recover $n$ independent functions instead of just $n$ measurements.  Assuming symmetry properties on the coefficients like in Theorem \ref{thm:main_n_stationary} is essentially an alternative  way to get around this difficulty. 

An interesting point in the proof of the time domain results is how one decouples the information concerning $\A$ from the information on $q$. The method used here consists in considering the solutions of the initial value problem
\begin{equation} \label{eq:direct_2}
(\p_t^2 + H_{\A,q}) U_{\A,q} =0 \; \; \text{in}\; \R^{n+1}, \qquad U_{\A,q}|_{\left\{t<-1\right\}} = H(t- x \cdot \omega),
\end{equation}
where $H$ stands for the Heaviside function. Since \eqref{eq:direct} is essentially the time derivative of the previous IVP, it turns out that it is equivalent to formulate the inverse scattering problem in time domain using any of these initial value problems.
The advantage is that the solutions of \eqref{eq:direct_2} contain information only about $\A$ at the surface $\{t=x \cdot \omega\}$. By using these ideas, we are able to estimate both $\A_1-\A_2$ in terms of $q_1-q_2$ and $q_1-q_2$ in terms of $\A_1-\A_2$. Using these two estimates in tandem allows us to recover both the magnetic field and electric potential under the assumption \eqref{eq:A_condition}.

This paper is structured as follows. In Section \ref{sec:direct_scat_prob} we state the time domain results, Theorem \ref{thm:2n_magnetic} and Theorem \ref{thm:2n_drift_nogauge}, from which Theorems \ref{thm:main_2n_stationary} and  \ref{thm:main_drif2n_stationary} follow by Theorem \ref{thm:equivalence}. We also analyze the structure of the solutions of the initial value problems \eqref{eq:direct} and  \eqref{eq:direct_2}   and we state several of their  properties that will  play an essential role later on. In Section \ref{sec:Carleman} we introduce the Carleman estimate and in Section \ref{sec:2n} we combine the results of the previous two sections to prove Theorems  \ref{thm:2n_magnetic}  and \ref{thm:2n_drift_nogauge}. In the last section of the paper we  state and prove Theorems \ref{thm:n_drift_nogauge} and \ref{thm:n_magnetic} in order to illustrate how  the number of measurements can be reduced in time domain by imposing symmetry assumptions on the potentials (Theorem \ref{thm:main_n_stationary} follows from the second result).  The proof of Theorem \ref{thm:equivalence} is given in  Appendix \ref{appendix:stationary},   and Appendix \ref{appendix:wave_eq} is devoted to adapting  several known results for the wave operator to our purposes.

\subsection*{Acknowledgements}

C.M.\ was supported by project MTM2017-85934-C3-3-P. L.P.\ and M.S.\ were supported by the Academy of Finland (Finnish Centre of Excellence in Inverse Modelling and Imaging, grant numbers 312121 and 309963), and M.S.\ was also supported by the European Research Council under Horizon 2020 (ERC CoG 770924).

 \section{The inverse problem in time domain} \label{sec:direct_scat_prob}
 
\subsection*{Main results in time domain}
 Let $\omega \in S^{n-1}$ be fixed. In the time domain setting we consider the initial value problem
 \begin{equation} \label{eq:wave_V}
( \partial_t^2 + H_{\mathcal V})U_{\mathcal V}=0  \; \; \text{in}\; \R^{n+1}, \qquad U_{\mathcal V}|_{\left\{t<-1\right\}} = \delta (t- x \cdot \omega),
\end{equation}
where $\delta$ represents the $1$-dimensional delta distribution and $H_{\mathcal V} = -\Delta + {\mathcal V}(x,D)$ . Formally, the problem \eqref{eq:sta_sct} is the Fourier transform in the time variable of \eqref{eq:wave_V}. As we will show later in this section, there is a unique distributional  solution of $U_{\mathcal V}$  if   the first order coefficients of $\mathcal V$ are in  $C^{\M+2}_{c}(\R^n)$ and the zero order coefficient is  $C^{\M}_{c}(\R^n)$, for $\M$ as in \eqref{eq:m}. 

The inverse   problem in the time domain   consists in determining the coefficients of $\mathcal{V}$ from certain measurements of $U_{\mathcal V}$ at the boundary $\p B \times (-T, T)\subset \R^{n+1}$ for some fixed $T>0$. To simply the notation  we define  
\[ \Sigma: = \partial B \times (-T,T)  .\]
 From now, depending on the context,  it will be useful to write the Hamiltonian $H_\ma V$ both in the forms
\begin{equation} \label{eq:hamilt_Aq}
  H_{\A,q} = (D + \A)^2 +  q   ,
\end{equation}
and  
\begin{equation} \label{eq:hamilt_WV}
  L_{\W,V} =   -\Delta + 2\W \cdot \nabla +  V   ,
\end{equation}
where $\A,\W \in C^{\M+2}_c(\R^n; \C^n)$ and   $q,V\in  C^{\M}_c(\R^n ;\C)$. 
Since the coefficients have high regularity and are complex valued, both forms are completely equivalent, but the first notation is specially convenient in the cases where there is gauge invariance. 
In fact, this inverse problem has an invariance equivalent to the gauge invariance present in the frequency domain problem. A straightforward computation shows that  if $U$ is a solution of
\[
(\p_t^2 + H_{\A,q})U=0 \; \; \text{in}\; \R^{n+1}, \qquad U|_{\left\{t<-1\right\}} = \delta (t- x \cdot \omega),
\]
then $\widetilde{U}= e^{-f} U$ is a solution of
\[
(\p_t^2 + H_{\A+\nabla f, q})\widetilde{U}=0 \; \; \text{in}\; \R^{n+1}, \qquad \widetilde{U}|_{\left\{t<-1\right\}} = \delta (t- x \cdot \omega),
\]
where $f$ is any $C^2(\R^n)$ function with compact support in $B$.   The initial condition satisfied by $\widetilde{U}$ is not affected by the exponential factor $e^{-f}$ since for $t<-1$ the distribution $\delta (t-x\cdot \omega)$ is supported in $\{ x \cdot \omega  < -1\}$, a region where $f$ vanishes. On the other hand   we also have  that $\widetilde{U}|_{\Sigma}= U|_{\Sigma}$ since the support of $f$ is contained in $B$. Hence, at best one can recover the magnetic field $d\A$ from the boundary data $U|_{\Sigma}$. 
We now state two uniqueness results for the inverse problem that we have just introduced.

\begin{Theorem} \sl  \label{thm:2n_magnetic} 
Let  $\A_1,\A_2 \in C^{\M+2}_c(\R^n; \C^n)$ and $q_1,q_2  \in  C^{\M}_c(\R^n; \C)$ with compact support in $B$ and such that
 \begin{equation*} 
\int_{-\infty}^\infty  e_n \cdot \A_k(x_1,\dots,x_{n-1},s) \, ds = 0 \quad \text{for} \;  k=1,2, \;  \text{and all} \;(x_1,\dots,x_{n-1}) \in \R^{n-1}.
 \end{equation*}
Also, let $1 \le j \le n$ and consider the   $2n$ solutions $U_{k, \pm j}(x,t)$ of
   \begin{equation}\label{eq:magnetic2n_wave}
( \partial_t^2 + (D +\A_k )^2 +  q_k )U_{k, \pm j}=0  \; \, \text{in}\; \R^{n+1}, \quad U_{k, \pm j}|_{\left\{t<-1\right\}} = \delta (t- \pm  x_j).
\end{equation}
 If for each $1 \le j \le n$ one has $U_{1,\pm j} = U_{2, \pm j}$ on the surface $\Sigma \cap \{ t \ge  \pm  x_j \}$, then $d \A_1 = d \A_2$ and $q_1 = q_2$.
\end{Theorem}
\noindent  As in the introduction, we highlight  that the restriction of the distribution $U_{1,\pm j}$ to the surface $\Sigma$  is well defined and vanishes in the open set $\Sigma \cap \{ t < \pm  x_j \}$, see the comments after Proposition \ref{prop:sol_smooth_2} for more details.
\noindent Theorem \ref{thm:main_2n_stationary} follows directly from this result and Theorem \ref{thm:equivalence}.
On the other hand, if we  fix the zero order term to be always the same, then the gauge invariance disappears and one can recover completely the first order term  of the perturbation. To state this result we use the Hamiltonian in the form \eqref{eq:hamilt_WV}.
\begin{Theorem} \sl    \label{thm:2n_drift_nogauge}
Let $\W_1,\W_2 \in C^{\M+2}_c(\R^n; \C^n)$ and  $V \in C^{\M}_c(\R^n; \C)$ with compact  support in $B$.  Let $1 \le j \le n$ and $k=1,2$, and consider the corresponding  $2n$ solutions $U_{k, \pm j}$ satisfying
  \begin{equation} \label{eq:drift2n_wave}
( \partial_t^2   - \Delta +  2\W_k \cdot \nabla + V) U_{k, \pm j} = 0, \; \, \text{in}\; \R^{n+1}, \quad U_{k, \pm j}|_{\left\{t<-1\right\}}= \delta (t- \pm x_j).
\end{equation}
If for each $1 \le j \le n$ one has $U_{1,\pm j} = U_{2, \pm j}$ on the surface $ \Sigma \cap \{ t \ge  \pm  x_j \}$, then $\W_1 = \W_2$.
\end{Theorem}
\noindent 
As in the previous case, Theorem \ref{thm:main_drif2n_stationary} follows from  Theorem \ref{thm:2n_drift_nogauge} with $\W_k = -i\A_k$ and Theorem \ref{thm:equivalence}. 
To see this, notice  that for $k=1,2$ we can write the Hamiltonians in  Theorem \ref{thm:main_drif2n_stationary} as $H_{\A_k} = (D +\A_k )^2 +  q_k $ for $q_k = V-\A_k^2 -D \cdot \A_k$. And since by assumption $\A_k$ is real and $H_{\A_k}$ is self-adjoint, $q_k$ must be a real function. This means that the conditions required to apply  Theorem  \ref{thm:equivalence} are satisfied.

 In Theorem \ref{thm:2n_drift_nogauge} one needs $2n$ measurements to obtain the unique determination of the first order coefficient $\W$. Therefore, one could expect to need  $2(n+1)$ measurements  in Theorem  \ref{thm:2n_magnetic}, since  one now also proves the unique determination of $q$. In  fact, $2n$   measurements are always enough:  the gauge invariance  essentially reduces one degree of freedom by making it  possible to choose a gauge in which the $n$th component of $\A_1 -\A_2$ vanishes (as we shall see later on,  the fact that the solutions $ U_{n, \pm j}$ coincide at $ \Sigma \cap \{ t = \pm  x_n \}$ guarantees that there are no obstructions for this   gauge transformation).


\subsection*{The direct problem}

In order to prove the previous theorems we need to study the direct problem \eqref{eq:wave_V} in more detail. Let $\omega \in S^{n-1}$. Assume $\W \in C^{\M+2}_c(\R^n; \C^n)$ and $V\in  C^{\M}_c(\R ;\C)$, and consider the   initial value problems for the wave operator
\begin{equation}\label{eq:wave_2}
(\partial_t^2 + L_{\W,V}) U_\delta =0 \; \text{in}\; \R^{n+1}, \quad U_\delta|_{\left\{t<-1\right\}}= \delta (t-x\cdot \omega),
\end{equation}
 and 
\begin{equation}\label{eq:wave_1}
(\partial_t^2 + L_{\W,V}) U_H =0 \; \text{in}\; \R^{n+1}, \quad U_H|_{\left\{t<-1\right\}}= H (t-x\cdot \omega),
\end{equation}
where $L_{\W,V}$ was defined in \eqref{eq:hamilt_WV}. 
As mentioned in the introduction, the reason we also consider the second equation is that the $\delta$-wave $U_\delta$ and $H$-wave $U_H$ contain equivalent information about $\W$ and $V$ (see Proposition \ref{prop:H_delta_equivalence}  below), but   $H$-waves  decouple the information on $\W$ from the information on   $V$.  

To study \eqref{eq:wave_2} and \eqref{eq:wave_1}, it is convenient  use the following coordinate system in $\R^n$.  We take $x = (y,z)$ where $y\in \R^{n-1}$ and $z = x \cdot \omega$. For a fixed $T> 7$, it will   be helpful to introduce the following subsets of $\mathbb{R}^{n+1}$ (see Figure \ref{fig:Q}):
\begin{equation}\label{useful:subsets}
\begin{alignedat}{2}
Q &:= B \times (-T, T) ,&   \qquad \Sigma &:= \partial B \times (-T, T), \\ 
Q_{\pm} &:= Q\cap \left\{ \pm (t-z)>0 \right\},&  \qquad     \Sigma_{\pm} &:= \Sigma \cap \left\{ \pm (t-z)>0\right\},\\
\Gamma &:= \overline{Q}\cap \left\{t=z \right\},  & \qquad  \Gamma_{\pm T}&:= \overline{Q}\cap \left\{ t=\pm T\right\}.
\end{alignedat}
\end{equation}


\definecolor{yqqqyq}{rgb}{0.5019607843137255,0,0.5019607843137255}
\definecolor{ffqqqq}{rgb}{1,0,0}
\definecolor{qqqqff}{rgb}{0,0,1}
\begin{figure}
\centering
\begin{tikzpicture}[line cap=round,line join=round,>=triangle 45,x=1cm,y=1cm, scale=0.74] 
\clip(-9.358703071672355,-5.007890305104027) rectangle (6.061023890784982,5.454473037632221);
\draw [rotate around={0:(0,3.5)},line width=0.8pt,color=qqqqff,fill=qqqqff,fill opacity=0.03] (0,3.5) ellipse (3cm and 0.5070925528371095cm);
\draw [rotate around={0:(0,-3.5)},line width=0.8pt,color=ffqqqq,fill=ffqqqq,fill opacity=0.02] (0,-3.5) ellipse (3cm and 0.5070925528371095cm);
\draw [line width=0.8pt] (-4.366666666666666,4.3)-- (3.303333333333333,4.3);
\draw [line width=0.8pt] (3.303333333333333,4.3)-- (4.37,2.7);
\draw [line width=0.8pt] (4.37,2.7)-- (-3.3,2.7);
\draw [line width=0.8pt] (-3.3,2.7)-- (-4.366666666666666,4.3);
\draw [line width=0.8pt] (-4.366666666666666,-2.7)-- (-3.3,-4.3);
\draw [line width=0.8pt] (-3.3,-4.3)-- (4.366666666666666,-4.3);
\draw [line width=0.8pt] (4.366666666666666,-4.3)-- (3.3,-2.7);
\draw [line width=0.8pt] (-4.366666666666666,-0.8)-- (-3.3,-2.4);
\draw [line width=0.8pt] (-3.3,-2.4)-- (4.37,0.78);
\draw [line width=0.8pt] (4.37,0.78)-- (3.303333333333333,2.38);
\draw [rotate around={22.519010891489422:(0.024386601355064673,-0.000580261758916719)},line width=0.4pt,dash pattern=on 1pt off 1pt,color=yqqqyq,fill=yqqqyq,pattern=north east lines,pattern color=yqqqyq] (0.024386601355064673,-0.000580261758916719) ellipse (3.2239052113396367cm and 0.6792461319238752cm);
\draw [line width=0.8pt,color=qqqqff] (-2.995294206724416,2.76)-- (-2.995294206724416,-1.16);
\draw [line width=0.8pt,color=ffqqqq] (-2.995294206724416,-2.273668263022638)-- (-2.995294206724416,-3.471608544965999);
\draw [line width=0.4pt,dash pattern=on 1pt off 1pt,color=qqqqff] (3,3.5)-- (3,2.7);
\draw [line width=0.8pt,color=qqqqff] (3,2.7)-- (3,1.2397727958346456);
\draw [line width=0.8pt,color=ffqqqq] (3,0.21199478487613982)-- (3,-3.5);
\draw [line width=0.8pt] (-4.366666666666666,-0.8)-- (-2.9952942067244166,-0.23142575976318686);
\draw [line width=0.4pt,dash pattern=on 1pt off 1pt] (-2.9952942067244166,-0.23142575976318686)-- (3,2.254237288135591);
\draw [line width=0.4pt,dash pattern=on 1pt off 1pt,color=qqqqff] (-2.997216183507296,3.5218404273387955)-- (-2.995294206724416,2.76);
\draw [line width=0.8pt] (3.3033333333333372,2.38)-- (3,2.254237288135591);
\draw [line width=0.8pt] (-4.366666666666667,-2.7)-- (-2.9952942067244153,-2.7);
\draw [line width=0.4pt,dash pattern=on 1pt off 1pt] (-2.9952942067244153,-2.7)-- (3,-2.7);
\draw [line width=0.8pt] (3.3,-2.7)-- (3,-2.7);
\draw [color=qqqqff](-4.04085324232082,4.252981103662169) node[anchor=north west] {$\mathit{\large t=T}$};
\draw [color=ffqqqq](2.194778156996587,-3.7212169189804207) node[anchor=north west] {$\mathit{\large t=-T}$};
\draw [color=yqqqyq](-3.6101023890784987,-1.4337975658717894) node[anchor=north west] {$\mathit{\large t=z}$};
\draw [color=qqqqff](-2.605017064846417,1.8384828976030583) node[anchor=north west] {$\mathbf{\Huge {Q_+}}$};
\draw [color=ffqqqq](0.9025255972696247,-1.3861429960153595) node[anchor=north west] {$\mathbf{\Huge {Q_-}}$};
\draw [color=yqqqyq](0.7589419795221843,0.8853915004744619) node[anchor=north west] {$\mathit{\mathbf{\Huge {\Gamma}}}$};
\draw [line width=0.4pt,dash pattern=on 1pt off 1pt] (0,0)-- (2.4886792452830186,0);
\draw [->,line width=0.4pt] (2.4886792452830186,0) -- (5,0);
\draw [line width=0.4pt,dash pattern=on 1pt off 1pt] (0,0)-- (-3.8377875673321737,-1.5933186490017377);
\draw [->,line width=0.4pt] (-3.8377875673321737,-1.5933186490017377) -- (-4.973404115610644,-2.064147934989421);
\draw [line width=0.4pt,dash pattern=on 1pt off 1pt] (0,0)-- (0,3.4739413680781777);
\draw [->,line width=0.4pt] (0,3.4739413680781777) -- (0,5);
\draw [line width=0.4pt,dash pattern=on 1pt off 1pt] (0.001596349391585201,0)-- (3.831428422953029,1.5878573655704546);
\draw [->,line width=0.4pt,dash pattern=on 1pt off 1pt] (3.831428422953029,1.5878573655704546) -- (5.044755893953274,2.0909057824910255);
\draw (5.045938566552902,0.20234266586563446) node[anchor=north west] {$z$};
\draw (-0.1435836177474403,5.571424203023394) node[anchor=north west] {$t$};
\draw (-5.517713310580206,-1.815034124723228) node[anchor=north west] {$y$};
\end{tikzpicture}
\caption{The regions $Q$, $Q_\pm$ and $\Gamma$}  \label{fig:Q}
\end{figure}
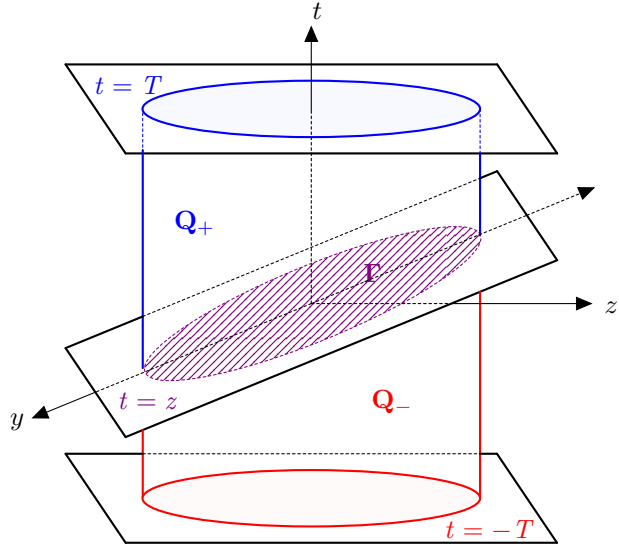


We now give a heuristic motivation of the existence of solutions $U_\delta$ and $U_H$ of \eqref{eq:wave_2} and \eqref{eq:wave_1}. For the interested reader we give a proof in  Section  \ref{subsec:existence} of the properties that we now state, by means of  the progressing wave expansion method.
We start by making the ansatz of looking for possible solutions of \eqref{eq:wave_2} and \eqref{eq:wave_1} in the family of functions satisfying
\[
U(y, z, t)=  f(y,z,t) H(t-z) + g(y,z,t) \delta(t-z),
\]
 where $f(y,z,t)$ and $g(y,z,t)$  are $C^2$ functions in $\{t \ge z\}$. A straightforward computation shows that
\begin{equation} \label{eq:u_v_computation}
\begin{aligned}
(\p_t^2+ L_{\W,V})U&= \left[(\p_t^2+ L_{\W,V})f\right]H(t-z)\\
 &+ \left[(\p_t^2+ L_{\W,V})g +2(\p_t+\p_z - \omega \cdot\W)f \right]\delta(t-z) \\
  &+ 2\left[ (\p_t+\p_z- \omega \cdot\W)g\right] \partial_t \delta(t-z).
\end{aligned} 
\end{equation}
In the case of equation \eqref{eq:wave_1} to satisfy the initial condition we need to have  $f = u $ where $u$ is a function satisfying $u(x,t) = 1$ for all $t<-1$, and $g(x,t)=0$ for all $(x,t) \in \R^{n+1}$.    Then \eqref{eq:u_v_computation} implies that $u$ must satisfy
\begin{equation} \label{eq:nu_1}
\begin{aligned}
(\partial_t^2 + L_{\W,V}) u &=0  \qquad \text{in}\; \left\{ t> z\right\}, \\ 
\partial_t u + \partial_z u - \omega \cdot\W u &= 0 \qquad \text{in } \left\{ t= z\right\}.
\end{aligned}
\end{equation}
The unique solution of the last ODE is
\[
 u(y,z,z) = e^{\int_{-\infty}^{z}  \omega \cdot\W(y,s) \,ds},
\] 
since it has to satisfy the initial condition  $u(y,z,z)=1$ for $z<-1$.  We now state this and further results about the solution of \eqref{eq:wave_1}.

\begin{Proposition} \sl \label{prop:sol_smooth_1} 
Let $\omega \in S^{n-1}$ be fixed. Let $\W \in C^{\M+2}_c(\R^n; \C^n)$ and $V \in  C^{\M}_c(\R^n; \C)$. Define
\begin{equation} \label{eq:psi1}
 \psi(x) := \int_{-\infty}^{0}  \omega \cdot\W(x + s\omega ) \,ds.
\end{equation}
There is a unique  distributional solution $U_H(x,t;\omega)$ of \eqref{eq:wave_1}, and it is supported in the region $\left\{ t\geq x \cdot \omega \right\}$. 
In particular,
\begin{equation*}
 U_H(x,t;\omega) = u(x,t)H(t-x \cdot \omega),
 \end{equation*}
  where $u$  is a ${C^2}$ function in $\left\{ t\geq x \cdot \omega \right\}$ satisfying the IVP
\begin{equation}\label{eq:smooth_1}
\begin{alignedat}{2}
&(\partial_t^2 + L_{\W,V}) u =0&  \qquad &\text{in} \; \; \left\{ t> x \cdot \omega \right\}, \\ 
&u(x, x \cdot \omega) = e^{\psi(x)}& \qquad &\text{in} \; \; \left\{ t= x \cdot \omega\right\},\\
&u(x,t)=1 &  \qquad   &\text{in} \; \; \left\{ x \cdot \omega \le t<-1 \right\}.
\end{alignedat}
\end{equation}
\end{Proposition}

\noindent Notice  that the boundary value of $u$ at $\{t=x\cdot \omega\}$ depends only on $\W$ and not on the zero order term $V$.  

We now study the solutions  of \eqref{eq:wave_2}.  In this case we have to consider \eqref{eq:u_v_computation} with $g =1$ for $t<-1$, and  $f= v$ where $v$ is a function  satisfying $v(x,t) = 0$ for $t< -1$. Then the following conditions must be satisfied:
\begin{equation} \label{eq:u_v_cond2}
\begin{alignedat}{2}
(\partial_t^2 + L_{\W,V}) v &=0 &  \quad  &\text{in} \; \;  \{t> z\}, \\ 
\partial_t v + \partial_z v - \omega \cdot\W v &= -1/2( \partial_t^2 + L_{\W,V})g &  \quad  &\text{in} \; \;  \left\{t= z\right\}, \\
\partial_t g + \partial_z g - \omega \cdot\W g &= 0 & \quad &\text{in} \; \;  \left\{ t= z\right\},\\
\partial_t (\partial_t g) + \partial_z (\partial_t g) - \omega \cdot\W (\partial_t g) &= 0 &  \quad &\text{in} \; \;  \left\{ t= z\right\}.
\end{alignedat}
\end{equation}
The last two conditions are required so that $ \left[ (\p_t+\p_z- \omega \cdot\W)g\right] \partial_t \delta(t-z) $ vanishes completely in all $\R^{n+1}$ as a distribution. As in the case of \eqref{eq:nu_1}, the third ODE and the initial condition imply that 
\[
g(y,z,z)  = e^{\psi(y,z)} = e^{\int_{-\infty}^{z}  \omega \cdot\W(y,s) \,ds}, 
\]
and in fact we are going to choose $g(y,z,t) = e^{\psi(y,z)}$ for all $(y,z,t) \in \R^{n+1}$. We can freely do this: $g_1(y,z,t) \delta (t-z)  = g_2(y,z,t) \delta (t-z) $ iff $g_1(y,z,z) = g_2(y,z,z)$. Also, the previous choice implies that $\partial_t g = 0$, so that the last condition is satisfied too.  Computing explicitly $L_{\W,V}(g)$, the second equation in \eqref{eq:u_v_cond2} becomes the ODE
\[ \partial_t v + \partial_z v - \omega \cdot\W v = - \frac{1}{2}   e^\psi  (-\Delta \psi -| \nabla \psi|^2 + 2\W \cdot \nabla \psi  +V)\quad \text{in } \left\{ t= z \right\}, \]
with initial condition $v(y,z,z)  = 0$ if $z<-1$. The unique solution is then 
\[
v(y,z,z)=  
  - \frac{1}{2} e^{  \psi(y,z)} \int_{-\infty}^{z}  \left [ -\Delta \psi - |\nabla \psi|^2 +  2\W \cdot \nabla \psi + V \right ](y,s) \, ds.
  \]
We state this rigorously in the following proposition.
\begin{Proposition} \sl  \label{prop:sol_smooth_2}
Let $\omega \in S^{n-1}$ be fixed. Let $\W \in C^{\M+2}_c(\R^n; \C^n)$ and $V \in  C^{\M}_c(\R^n; \C)$. Consider $\psi$  as in \eqref{eq:psi1}. There is a unique distributional solution $U_\delta(x,t;\omega)$ of \eqref{eq:wave_2},  and it is supported  in the region $\left\{ t\geq x \cdot \omega\right\}$.
 Moreover 
 \begin{equation} \label{eq:U_delta}
 U_\delta(x,t;\omega) =  v(x,t)  H(t-x \cdot \omega) + e^{\psi(x)}\delta(t-x \cdot \omega),
 \end{equation}
  where $v$  is a $C^2$ function  in $\left\{ t\geq x \cdot \omega \right\}$ satisfying
\begin{equation}\label{eq:smooth_2}
\begin{alignedat}{2}
(\partial_t^2 + L_{\W,V}) v &=0 & \quad &\text{in}\;  \; \left\{ t> x \cdot \omega\right\}, \\ 
v(x, x \cdot \omega) &= F(x) & \quad &\text{in} \;  \; \left\{ t= x \cdot \omega\right\},  \\
v(x,t)&=0 & \quad   &\text{in} \;  \; \left\{ x \cdot \omega \le  t<-1 \right\},
\end{alignedat}
\end{equation}
with
\begin{equation*} 
  F(x)=  
 -\frac{1}{2} e^{  \psi(x)} \int_{-\infty}^{0} \left [ -\Delta \psi - |\nabla \psi|^2 +  2\W \cdot \nabla \psi + V \right ](x +s\omega) \, ds.
\end{equation*}
\end{Proposition}
\noindent One of the consequences of the formula \eqref{eq:U_delta} is that the restriction of $U_\delta(x,t;\omega)$ to the surface $\Sigma$ is well defined. This essentially follows from the fact that the wave front set of the distribution $\delta(t-x\cdot \omega)$ is disjoint from the normal bundle of $\Sigma$.
 We emphasize  that $U_\delta$ satisfies the initial condition $ U_\delta|_{\left\{t<-1\right\}}= \delta (t-x\cdot \omega)$, even if it does not look that way at a first glance. This is due to the fact that when $t<-1$ the distribution $\delta (t-x\cdot \omega)$ is supported in $\{ x \cdot \omega  < -1\}$, a region in which $\psi$ vanishes. Therefore
\[ e^{\psi(x)}\delta (t-x\cdot \omega) =  \delta (t-x\cdot \omega) \quad \text{for} \; \, t<-1. \]
As mentioned previously, for more details about the proofs of Propositions \ref{prop:sol_smooth_1} and \ref{prop:sol_smooth_2} see Section  \ref{subsec:existence}. We remark that the condition \eqref{eq:m} on the regularity of the coefficients appears in the proofs of these propositions in order to have $C^2$ solutions $u$ and $v$ of \eqref{eq:smooth_1} and \eqref{eq:smooth_2}.

An important fact later on is that  the solutions   of \eqref{eq:wave_2} and \eqref{eq:wave_1} satisfy that $\partial_t U_H= U_\delta$. This  is consequence of the independence of $V$ and $\W$ from $t$ together with the uniqueness of solutions for both equations. Of particular relevance for the scattering problem  will be  that this equivalence also holds for the boundary data: knowledge of $U_H|_{\Sigma_+}$ gives $U_\delta|_{\Sigma_+}$ and vice versa. 

\begin{Proposition} \sl  \label{prop:H_delta_equivalence}
 Let $\W \in C^{\M+2}_c(\R^n; \C^n)$ and $V \in  C^{\M}_c(\R^n; \C)$, and let  $\omega \in S^{n-1}$ be fixed. If $U_\delta$ and $U_H$ are, respectively, the unique  distributional solutions of \eqref{eq:wave_2} and \eqref{eq:wave_1}, then one has that  $\partial_t U_H = U_\delta$ in the sense of distributions. 
In fact,  it also holds that $v(x,t) = \partial_t u(x,t)$, so that \eqref{eq:U_delta} can be written as
 \begin{equation} \label{eq:U_delta_2}
 U_\delta(x,t;\omega) =  (\partial_t u(x,t))  H(t-x \cdot \omega) + u(x,t) \delta(t-x \cdot \omega).
 \end{equation}
Specifically, if $\psi$ is given by \eqref{eq:psi1} we have that
\begin{equation} \label{eq:u_v}
u (x,t) =  \int_{x \cdot \omega}^{t} v(x,\tau ) \, d\tau + e^{\psi(x)}, \quad \text{for} \;\, t \ge x \cdot \omega.
\end{equation}
\end{Proposition}
\noindent Notice that  the previous identity holds  in particular   for every $x \in \partial B$, so  we can write that
\begin{equation} \label{eq:u_v_boundary}
u|_{\Sigma_+}(x,t) =  \int_{x \cdot \omega}^{t} v|_{\Sigma_+}(x,\tau ) \, d\tau + e^{\psi(x)}, \quad \text{for} \;\, (x,t) \in \Sigma_+.
\end{equation}

\begin{proof}[Proof of Proposition \ref{prop:H_delta_equivalence}]
Since $\W$ and $V$ are independent of $t$, we can take a time derivative of both sides of \eqref{eq:wave_1}. This implies that $\partial_t U_H$ satisfies
\begin{equation*}
(\partial_t^2 + L_{\W,V}) \partial_t U_H =0 \; \text{in}\; \R^{n+1}, \quad \partial_t U_H|_{\left\{t<-1\right\}}= \delta (t-x\cdot \omega).
\end{equation*}
Computing explicitly $\partial_t U_H$, we get
 \begin{equation*}
  \partial_t U_H (x,t;\omega) =  (\partial_t u(x,t))  H(t-x \cdot \omega) + u(x,t) \delta(t-x \cdot \omega).
 \end{equation*}
Proposition \ref{prop:sol_smooth_2} implies there is a unique distributional solution of \eqref{eq:wave_1}, and hence $\partial_t U_H = U_\delta$. Then $\partial_t u = v$. We also get that $u(x,x\cdot \omega) = e^{\psi}$, but we already knew this from Proposition \ref{prop:sol_smooth_1}. Therefore \eqref{eq:U_delta_2} holds true. Identity \eqref{eq:u_v} follows directly  by the fundamental theorem of calculus. 
\end{proof}

As an immediate  consequence of identity \eqref{eq:u_v_boundary} we get the following lemma.

\begin{Lemma} \sl 
\label{lemma:equivalence_data}
  Let $\W_1,\W_2 \in C^{\M+2}_c(\R^n; \C^n)$  and  $V_1,V_2\in  C^{\M}_c(\R^n; \C)$.  For $k=1,2$ and $\omega \in S^{n-1}$,
consider the solutions  $U_{H,k}$ and $U_{\delta,k}$   of
  \begin{equation*}
( \square  +  2\W_k \cdot \nabla + V_k) U_{H,k} = 0, \; \, \text{in}\; \R^{n+1}, \quad U_{H,k}|_{\left\{t<-1\right\}}= H (t-  x \cdot \omega),
\end{equation*}
and
  \begin{equation*}
( \square  +  2\W_k \cdot \nabla + V_k) U_{\delta,k} = 0, \; \, \text{in}\; \R^{n+1}, \quad U_{\delta,k}|_{\left\{t<-1\right\}}= \delta (t-  x \cdot \omega).
\end{equation*}
Then $U_{H,1} = U_{H,2}$ in $\Sigma \cap \{t \ge  x\cdot \omega\}$ if and only if $U_{\delta,1}  = U_{\delta,2}$ in $\Sigma \cap \{ t \ge  x\cdot \omega \}$.
\end{Lemma}


\subsection*{Energy estimates}
To finish this section we  state  three different  estimates related to the  wave operator that will be useful later on.   They are analogues of the estimates given in   \cite[Lemmas 3.3-3.5]{RakeshSalo1}, modified in order to account for the presence of a first order perturbation not considered in the mentioned paper. For completeness we have included the proofs in Section \ref{subsec:energy}.
The first two lemmas will be used to control certain boundary terms appearing in the Carleman estimate. We denote by $\nabla_\Gamma \alpha$ the component of $\nabla \alpha $ tangential to $\Gamma$.

\begin{Lemma} \sl \label{lemma:energy_T}
Let $T>1$. Let $\W  \in L^\infty(B,\C^n) $ and $V\in L^\infty(B,\C)$. Then, the estimate
\begin{multline*}
\norm{\alpha}_{L^2(\Gamma_T)} + \norm{\nabla_{x,t}\alpha}_{L^2(\Gamma_T)} \\
\lesssim   \norm{\alpha}_{H^1(\Gamma)}  + \norm{(\square + 2\W \cdot \nabla + V)\alpha}_{L^2(Q_+)} +\norm{\alpha}_{H^1(\Sigma_+)} + \norm{\p_\nu \alpha}_{L^2(\Sigma_+)} ,
\end{multline*}
holds  true \footnote{Throughout the paper we write $a \lesssim b$ or equivalently $b \gtrsim a$, when $a$ and $b$ are positive constants and there exists $C > 0$ so that $a \leq C b$. We refer to $C$ as the implicit constant in the estimate.} for every $\alpha\in C^{\infty}(\overline{Q}_+)$. The implicit constant depends  on $\norm{\W }_{L^\infty(B)}$, $\norm{V}_{L^\infty(B)}$  and $T$.
\end{Lemma}

\begin{Lemma} \sl \label{lemma:energy_estimate_gamma}
Let $T>1$. Let $\W  \in L^\infty(B,\C^n) $, $V\in L^\infty(B,\C)$ and $\phi \in C^2(\overline{Q}_+)$.  Then,  there is a constant $\sigma_0>0$ such that the following estimate
\begin{equation}\label{est:top}
\begin{aligned}
&\sigma^2 \norm{e^{\sigma\phi}\alpha}_{L^2(\Gamma)}^2 + \norm{e^{\sigma \phi}\nabla_{\Gamma}\alpha}_{L^2(\Gamma)}^2\lesssim    \sigma^3 \norm{e^{\sigma\phi}\alpha}_{L^2(Q_+)}^2 +  \sigma \norm{e^{\sigma\phi}\nabla_{x,t}\alpha}_{L^2(Q_+)}^2    \\
&\qquad \qquad \qquad  +  \norm{e^{\sigma\phi}(\square + 2\W \cdot \nabla + V)\alpha}_{L^2(Q_+)}^2 +  \sigma^2 \norm{e^{\sigma\phi}\alpha}_{L^2(\Sigma_+)}^2+  \norm{e^{\sigma\phi} \nabla_{x,t}\alpha}_{L^2(\Sigma_+)}^2 ,
\end{aligned}
\end{equation}
holds  for every $\alpha\in C^\infty(\overline{Q}_+)$ and for every $\sigma\geq \sigma_0$. The implicit constant  depends on    $\left\| \phi\right\|_{C^2(\overline{Q}_+)}$, $\norm{\W }_{L^\infty(B)}$, $\norm{V}_{L^\infty(B)}$   and $T$.
\end{Lemma}

The last lemma will be used to show that the normal derivative of a function $\alpha$ that vanishes at the set $\Sigma_{+}$ will also vanish provided that  certain   conditions are met. Here $\nu$ denotes the unit vector field normal to $\Sigma_+$.

\begin{Lemma} \sl \label{lemma:normal_estimate}
Let $T>1$. Let $\alpha (y,z,t)$ be a $C^2$ function on $\left\{ t\geq z\right\}$ satisfying
\begin{equation}\label{eq:alpha}
\begin{alignedat}{2}
\square\,  \alpha &=0 & \qquad &\text{in}\;\; \left\{ (y,z,t): |(y,z)|>1, \, t>z  \right\}, \\ 
(\partial_t +\partial_z) \alpha (y, z,z) &= 0 & \qquad &\text{in} \;\; \left\{ |(y,z)|>1, \, t=z\right\},\\
\alpha(y, z, t)&=0 & \qquad  &\text{in} \;\; \left\{ z < t<-1 \right\}.
\end{alignedat}
\end{equation}
Assume that on the region $\abs{(y,z)}\geq 1$ we have  
 \begin{equation}\label{eq:alpha_t=z}
\alpha (y,z,z)= \begin{cases}
0,  \hspace{7.5mm} \text{if} \quad \abs{y}\geq 1,  \\
0,  \hspace{7.5mm} \text{if} \quad \abs{y}\leq 1, \; \text{and}\; \; z \le - \sqrt{1-|y|^2},\\
\beta(y), \hspace{2.5mm} \text{if} \quad \abs{y}\leq 1, \; \text{and}\; \; z \ge \sqrt{1-|y|^2},
\end{cases}
\end{equation}
for some $\beta\in C^2_c(\R^{n-1})$ compactly supported on $\abs{y}\leq 1-\varepsilon$, where $\varepsilon\in (0,1)$. Then
\[
\norm{\p_\nu (\chi \alpha)}_{L^2(\Sigma_+)} \lesssim \varepsilon^{-1/2} \norm{\chi}_{C^1} \left(\norm{ \alpha}_{H^1(\Sigma_+)} + \norm{\alpha}_{H^1(\Sigma_+\cap \Gamma)}  \right),
\] 
for any function  $\chi \in C^1(\overline{Q}_+)$. The implicit constant depends  on $T$.
\end{Lemma}


\section{The Carleman estimate and its consequences} \label{sec:Carleman}

In this section we are going to apply a suitable Carleman estimate in order  to be able to control the difference of the potentials with the boundary data. For this purpose we adapt a Carleman estimate for  general second order operators stated in \cite[Theorem A.7]{RakeshSalo2}. The trick is to choose an appropriate weight function to obtain a meaningful estimate for the wave operator. 

First, take any $\vartheta \in \R^n$ such that $|\vartheta |=2$ and consider the following smooth function,
\begin{equation} \label{eq:weight_def}
 \eta(y,z,t) := |x-\ve|^2 - \frac{1}{4}(t-z)^2.
\end{equation}
The weight in the Carleman estimate is going to be the function $\phi = e^{\lambda \eta}$ for some $\lambda>0$ large enough. This choice is made in order to have several properties. On the one hand we want $\phi$ to be sufficiently pseudoconvex so that the Carleman estimate holds for the wave operator. On the other hand,  we want $\phi$ to decay very fast when $t >z$ in order to deal with certain terms appearing in the Carleman estimate (see Lemma \ref{lemma:weight} below and its application in the proof of  Lemma \ref{lemma:w+w-}). And finally, since the $z$ coordinate is going to be determined by the direction $\omega$ of the traveling wave, we require the restriction of $\phi$ to the surface $\{t=z\}$ to be independent of $\omega$ (or in other words, dependent on $x$ but independent of the choice of coordinates $x=(y,z)$).  This is of great help since we recover $\A$ combining measurements made from waves traveling in different directions.

In order to state the Carleman estimate we fix an open set $D \subset \{ (x,t) \in  \R^{n+1} : |x|<3/2\}$   such that $Q \subset D$ (recall the notation introduced in \eqref{useful:subsets}).
\begin{Proposition} \sl \label{prop:carleman_estimate} 
Let $\phi = e^{\lambda \eta}$ for $\lambda>0$ large enough.  Let $\Omega \subset D$ be any open set with Lipschitz boundary. Then there exists some $\sigma_0>0$ such that
\begin{multline*}
\sigma \int_{\Omega} e^{2\sigma \phi} (\abs{\nabla u}^2 + \sigma^2 |u|^2) \, dx + \sigma \int_{\p\Omega} \sum_{j=0}^n  \nu_j E^j(x,\sigma \Re(u), \nabla \Re(u)) \, dS\\
 + \sigma \int_{\p\Omega} \sum_{j=0}^n \nu_j E^j(x, \sigma \Im(u), \nabla \Im(u)) \, dS 
\lesssim \int_{\Omega}e^{2\sigma \phi}\abs{\square u}^2 dx
\end{multline*}
holds for all  $u\in C^{2}(\ol{\Omega})$ and all $\sigma\geq \sigma_0$. The implicit constant depends  on $n$, $\lambda$,   and $\Omega$.  Here $\nu$ is the outward pointing unit vector normal to $\p\Omega$, and for real $v$ and $0\le j \le n$, $E^j$ is given by 
\begin{multline} \label{eq:quadratic_forms}
 E^j \big (x, \sigma e^{-\sigma \phi}  v,\nabla(e^{-\sigma \phi}  v)\big ) = -\partial_j \phi (|\nabla_x v|^2-\partial_t v^2) \\
 - \sigma^2 \partial_j \phi (| \nabla_x \phi|^2 - \partial_j \phi^2)v^2 -2 \partial_j v( \nabla_x v \cdot \nabla_x \phi - \partial_t v \partial_t\phi)-g(x,t)\partial_j v v,
\end{multline}
where $g$ is some real valued and bounded function independent of $\sigma$ and $v$ (here the  index 0 corresponds to $t$, so that $\partial_0 \phi = \partial_t \phi$).
\end{Proposition} 
\begin{proof}
Since  $\square h $ is real if $h$ is a real function, the statement  follows from applying  Theorem A.7 of \cite{RakeshSalo2} to the real and imaginary parts of $u$, and then adding the resulting estimates.  
 However, to apply the mentioned result, one needs to verify that $\phi$ is {\it strongly pseudoconvex} in the domain $D$ with respect the wave operator $\square$. The reader can find the precise definition of this condition in   \cite[Appendix]{RakeshSalo2}), though it is not necessary for the discussion that follows.

Denote by $(\xi,\tau) \in \R^{n+1}$ the Fourier variables corresponding to $(x,t)$, where $\xi \in \R^n$ and $\tau\in \R$.  \\
It can be proved that a function $\phi = e^{\lambda \eta}$ will be strongly pseudoconvex for $\lambda >0$ large enough provided $\eta$ satisfies certain technical conditions.  
By Propositions A.3 and A.5 in \cite{RakeshSalo2}, these conditions are the following: one needs to verify that  the {\it level surfaces} of $\eta$ are  {\it pseudoconvex} with respect to the wave operator $\square$, and that $\abs{(\nabla_{x,t} \eta)(x,t)}>0$ for all $(x,t)\in \ol{D}$. \\
In our case, the second property  is immediate since $|\ve|=2$ implies that  $|x-\ve|^2$ has non-vanishing gradient in $\ol{D}$.
The reader can find in \cite[Definition A.1]{RakeshSalo2}  a precise definition of the first property.
For the purpose of this work it is enough to use that the level surfaces of  a function $f$ are pseudoconvex  w.r.t.\ $\square$ in a domain $D$ if for all $(x,t) \in \ol{D}$ and $(\xi,\tau) \in \R^{n+1}$ 
\begin{equation}\label{eq:c_pseudoconvex}
\tau^2 \p_t^2 f -2\tau \xi \cdot \nabla_x \p_t f + \sum_{j,k=1}^{n}  \xi_j\xi_k \, \partial^2_{jk}  f  > 0 \;  \text{ when } \tau^2 = \abs{\xi}^2 \text{ and } \tau \p_t f - \xi \cdot \nabla_x f = 0.
\end{equation}
(notice that $\tau^2 - |\xi|^2$ is the  symbol of the wave operator). 

We now consider $f$ given by
\[
f(y,z,t)= \frac{1}{2}(b|x-\ve|^2 - (t-z)^2),
\]
where $|\ve|=2$, $b>0$, and we are using coordinates $(x,t) = (y,z,t)  $, $ y\in \R^{n-1}$. Let  $(\zeta, \rho,\tau) \in \R^{n-1} \times \R\times \R$ be the  Fourier variables counterpart to $(y, z, t)$. 
A straightforward computation with $j, k =1,  \ldots, n-1$ and $\partial_j = \partial_{y_j}$, shows that the only non-vanishing  second order derivatives of $f$ are 
\[
 \partial^2_{t}f= -1,    \quad  \partial_{jk}^2   f= b\delta_{j,k} ,
  \quad \partial_z^2   f= b-1, \quad \partial^2_{tz}f=1,
\]
Then, in particular,  \eqref{eq:c_pseudoconvex} is verified if we show that
\begin{equation}\label{eq:c_pseudoconvex_2}
-\tau^2+ b|\zeta|^2+ (b-1)\rho^2- 2\tau\rho>0,\quad \text{whenever } \;  \tau^2= |\zeta|^2 + \rho^2.
\end{equation}
Note that both conditions are  homogeneous in the variables $(\zeta,\rho,\tau)$, so it is enough to study the case $\tau=1$. Thus \eqref{eq:c_pseudoconvex_2} becomes
\begin{equation*}
-1 + b + (b-2)\rho^2 - 2\rho>0  \qquad \text{ for} \;  -1 \le \rho\le 1.
\end{equation*}
It is not difficult to verify that for $b>3$ this inequality always holds. This proves that for $b=4$ the level surfaces of  $f$ are {\it strongly pseudoconvex} with respect to $\square$, and therefore, the same holds for $\eta= \frac{1}{2}f$.

To finish, we mention that precise  formula \eqref{eq:quadratic_forms} of the quadratic forms  $E_j$ is computed in detail in \cite[Section A.2]{RakeshSalo2}.
\end{proof}

In the following lemma we prove a couple of properties of the weight $\phi$ that will be important later in order to show that some   terms appearing in the Carleman estimate are suitably small in the parameter $\sigma$. 

\begin{Lemma} \sl \label{lemma:weight}
Let $\ve\in \mathbb{R}^n$ with $|\ve|=2$, and $\eta$ as in \eqref{eq:weight_def}. 
Then, the following properties are satisfied for any $T>7$:
\begin{itemize}
\item[i)]  The smallest value of $\phi$ on $\Gamma$ is strictly larger than the largest value of $\phi$ on $\Gamma_{-T} \cup \Gamma_{T}$.

\item[ii)] The function 
\[ \kappa (\sigma) = \sup_{(y,z) \in \overline{B}} \int_{-T}^T  e^{2\sigma(\phi(y,z,t)-\phi(y,z,z))}  \, dt \]
\end{itemize}
satisfies that  $\underset{\sigma \to \infty}{\lim} \kappa(\sigma) =0$.
\end{Lemma}
\begin{proof}
Let $\eta_0(y, z):=  \vert x-\ve \vert^2$. For the first assertion it is enough to prove that
\begin{equation*} 
\max_{\Gamma_{-T}\cup \Gamma_T} e^{\lambda (\eta_0 - \frac{1}{4}(t-z)^2)} <\min_{\Gamma} e^{\lambda \eta_0},
\end{equation*}
for $T$ large enough, since $\eta(y,z,z) = \eta_0(y, z)$. Observing that $|(y,z)| \le 1$ in $\Gamma$ and $\Gamma_{\pm T}$,  the previous inequality  will hold if 
\[
\max_B \eta_0- \min_B \eta_0< \frac{1}{4}(T-1)^2,
\]
and since $\max_B \eta_0- \min_B \eta_0 = 8$, this is true for any $T > 7$. This yields the first assertion.
To prove the second assertion we are going to use  the following inequality
\[
1- e^{-s} \geq \min \left\{ 1/2, s/2\right\}, \quad s\geq 0.
\]
Since $e^{\lambda \eta_0}>1$  always, we have
\[
\phi(y,z,z)-\phi(y,z,t)= e^{\lambda \eta_0} \left( 1- e^{-\frac{\lambda}{4}(t-z)^2} \right ) \geq 1- e^{-\frac{\lambda}{4}(t-z)^2}\geq \frac{1}{2}  \min \left\{ 1, \frac{\lambda}{4} (t-z)^2 \right\},
\]
and hence, since $|z| \le 1$,
\begin{align*}
0\leq \int^{T}_{-T} e^{2\sigma (\phi(y,z,t)-\phi(y,z,z))}dt \leq \int_{-T}^T e^{-\sigma  \min \left\{ 1, \frac{\lambda}{4} (t-z)^2\right\}}dt &\leq \int_{-T-z}^{T-z}e^{-\sigma  \min \left\{ 1, \frac{\lambda}{4} t^2\right\}}dt\\
 &\leq \int_{-T-1}^{T+1}e^{-\sigma  \min \left\{ 1, \frac{\lambda}{4} t^2\right\}}dt.
\end{align*}
By the dominated convergence theorem, the last integral goes to zero when $\sigma\to \infty$, and therefore $\displaystyle{\lim_{\sigma \to \infty} \kappa(\sigma) =0}$. This completes the proof.  
\end{proof}

We now adapt the Carleman estimate of Proposition \ref{prop:carleman_estimate}  to our purposes. First, define
\begin{equation*} 
\phi_0(x,\vartheta) := e^{\lambda\vert x-\ve \vert^2},
\end{equation*}
for $x\in \R^n$ and $|\vartheta| =2$, so that  $\phi(y,z,z) = \phi_0(y,z)$. From now on it is convenient to use the notation 
\begin{equation*}
 Z : =\partial_t + \partial_z, \qquad N : =\partial_t - \partial_z.
\end{equation*} 
We are interested in applying Proposition \ref{prop:carleman_estimate} for  $\Omega = Q_\pm $. The boundary of $Q_\pm$ is composed of the following regions:  $\partial Q_\pm = \Gamma \cup \Sigma_\pm \cup \Gamma_{\pm T}$. We are  going to use the precise formula \eqref{eq:quadratic_forms} only in $\Gamma$, where in fact an explicit computation yields that
\begin{multline} \label{eq:tangential}
 \sum_{j=0}^n    \nu_j  E^j \big (x, \sigma e^{-\sigma \phi}  v, \nabla(e^{-\sigma \phi}  v)\big ) = 
     4(N\phi)((Zv)^2 + \sigma^2 (Z \phi)^2 v^2) \\
  + 2(Zv)(|\nabla_y v|^2 - \sigma^2 |\nabla_y \phi|^2v^2) - 4(Zv)(\nabla_y v\cdot \nabla_y \phi) -2(Zv)gv  
 \end{multline}
  for any real function $v \in C^2(\ol{Q}_\pm )$ (for a detailed derivation of this formula see \cite[Section A.2]{RakeshSalo2}). The important thing about this identity is that it does not depend on $\nabla v$ but just on $Z v$ and $\nabla_y v$, which are  derivatives along directions tangent to $\Gamma$.
  
The following lemma is the consequence of the Carleman estimate in Proposition \ref{prop:carleman_estimate} that is relevant for our fixed angle scattering problem. It is an analogue of \cite[Proposition 3.2]{RakeshSalo2} adapted to the case of magnetic potentials.

\begin{Lemma} \sl  \label{lemma:w+w-}
 Let $T>7$ and $\omega \in S^{n-1}$, and let $\phi_0$ be as above. Assume that
\begin{equation} \label{eq:condw+w-}
\left |(\square  + 2\E_\pm \cdot \nabla + f_\pm) w_\pm \right | \lesssim |\A_\pm| + |q_\pm|  \quad \text{in} \quad Q_\pm,
\end{equation} 
 and that
\begin{equation} \label{eq:condw+w-2}
\left | (\partial_t + \partial_z - \omega \cdot \E_\pm) w_\pm \right | \gtrsim  |h_\pm| \quad \text{in} \quad \Gamma,
\end{equation}
 hold for some vector fields $\E_\pm, \A_\pm \in C(\R^n;\C^n)$, some functions $f_\pm , h_\pm, q_\pm \in C(\R^n;\C)$, and $w_\pm \in C^2(\ol{Q}_\pm)$. Then there is a constant $c>0$ such that,   for  $\sigma>0$ large enough,
 \begin{multline} \label{eq:est_0Carl}
 \sum_{\pm} \norm{ e^{\sigma  \phi_0}  h_\pm }_{L^2(B)}^2 \lesssim \sigma^3 e^{c \sigma} \norm{w_+ - w_-}_{H^1(\Gamma)}( \norm{w_+}_{H^1(\Gamma)} + \norm{ w_-}_{H^1(\Gamma)} )  
\\
 + \gamma (\sigma) \sum_{\pm}   \left(  \norm{e^{\sigma\phi_0} \A_\pm}_{L^2(B)}^2 + \norm{e^{\sigma\phi_0} q_\pm}_{L^2(B )}^2 \right) + \sigma^3 e^{c \sigma}  \sum_{\pm} \left (  \norm{w_\pm}_{H^1(\Sigma_\pm)}^2   + \norm{ \partial_\nu w_\pm}_{L^2(\Sigma_\pm)}^2 \right),
  \end{multline}
where $\gamma$ is a positive function satisfying $\gamma (\sigma) \to 0$ as $\sigma \to \infty$. The implicit constant is independent of $\sigma$ and $\ve$ and $w_\pm$. 
\end{Lemma}
 In the proof of the lemma it will be useful to introduce the   notation 
\begin{equation} \label{eq:F}
F^j(x,\sigma u, \nabla u) :=  e^{-2\sigma \phi}  E^j \big  (x,\sigma \Re(u), \nabla \Re(u)\big ) +  e^{-2\sigma \phi}E^j \big  (x, \sigma \Im(u), \nabla \Im(u) \big ) . 
\end{equation}
We remark that $F^j$ depends on the function $g$ in Proposition \ref{prop:carleman_estimate} which could in principle be different in the domains $Q_+$ and $Q_-$, but as discussed in \cite[footnote 1 in the proof of Proposition 3.2]{RakeshSalo2} we can choose $g$ to be the same both in $Q_+$ and $Q_-$.
\begin{proof}
We first apply Proposition \ref{prop:carleman_estimate} with $\Omega = Q_+ \subset D$. To simplify notation we use $w$, $\A $, $\E$, $q$, and $h$ instead of $w_+$, $\A_+$, $\E_+$, $q_+$, and $h_+$. For $\sigma >0$ large enough, Proposition \ref{prop:carleman_estimate} and \eqref{eq:F} yield the estimate
\begin{multline*} 
\sigma^3 \norm{e^{\sigma\phi}w}_{L^2(Q_+)}^2   +\sigma \norm{e^{\sigma\phi} \nabla w}_{L^2(Q_+)}^2 + \sigma \int_{\partial Q_+ } e^{2\sigma \phi} F^j (x,\sigma w, \nabla w)\nu_j  \, dS 
   \lesssim  \norm{e^{\sigma  \phi} \square w}^2_{L^2(Q_+)},
  \end{multline*}
where to shorten notation, we are temporarily adopting the convention that $F^j\nu_j$ stands for $\sum_{j=0}^n F^j\nu_j$. Since $\E$ and $f$ are bounded, a direct perturbation argument (one can absorb the extra terms in the left hand side for $\sigma$ large enough) allows to obtain from the previous estimate that
  \begin{multline*} 
\sigma^3 \norm{e^{\sigma\phi}w}_{L^2(Q_+)}^2   +\sigma \norm{e^{\sigma\phi} \nabla w}_{L^2(Q_+)}^2 + \sigma \int_{\partial Q_+ } e^{2\sigma \phi} F^j (x,\sigma w, \nabla w)\nu_j  \, dS \\   
   \lesssim  \norm{e^{\sigma  \phi} (\square + 2\E \cdot \nabla +f  )w}^2_{L^2(Q_+)},
  \end{multline*}
for $\sigma$ large enough. Then, since $\partial Q_+ = \Gamma \cup \Sigma_+ \cup \Gamma_T$ we have
  \begin{multline} \label{eq:carl_1}
\sigma^3 \norm{e^{\sigma\phi}w}_{L^2(Q_+)}^2   +\sigma \norm{e^{\sigma\phi} \nabla w}_{L^2(Q_+)}^2 + \sigma \int_{\Gamma} e^{2\sigma \phi} F^j (x,\sigma w, \nabla w)\nu_j  \, dS \\   
   \lesssim  \norm{e^{\sigma  \phi} (\square + 2\E \cdot \nabla +f  )w}^2_{L^2(Q_+)} + \sigma^3 \norm{e^{\sigma\phi}w}_{L^2(\Sigma_+ \cup \Gamma_T)}^2   + \sigma \norm{e^{\sigma\phi} \nabla w}_{L^2(\Sigma_+ \cup \Gamma_T)}^2.
  \end{multline}
 The energy estimate in Lemma \ref{lemma:energy_estimate_gamma} yields
\begin{multline*} 
\sigma^2 \norm{e^{\sigma\phi}w}_{L^2(\Gamma)}^2   + \norm{e^{\sigma\phi} \nabla_\Gamma w}_{L^2(\Gamma)}^2 
 \lesssim  \sigma^3 \norm{e^{\sigma\phi}w}_{L^2(Q_+)}^2  \\ +\sigma \norm{e^{\sigma\phi} \nabla w}_{L^2(Q_+)}^2 + \norm{e^{\sigma  \phi} (\square + 2\E  \cdot \nabla +f )w}^2_{L^2(Q_+)} + \sigma^2 \norm{e^{\sigma\phi}w}_{L^2(\Sigma_+)}^2   + \norm{e^{\sigma\phi} \nabla w}_{L^2(\Sigma_+)}^2.
\end{multline*}
Combining this estimate with \eqref{eq:carl_1} gives
\begin{multline} \label{eq:carl_3}
\sigma^2 \norm{e^{\sigma\phi}w}_{L^2(\Gamma)}^2   + \norm{e^{\sigma\phi} \nabla_\Gamma w}_{L^2(\Gamma)}^2 + \sigma \int_{ \Gamma } e^{2\sigma \phi} F^j (x,\sigma w, \nabla w)\nu_j \, dS \\   
   \lesssim  \norm{e^{\sigma  \phi} (\square + 2\E  \cdot \nabla  +f)w}^2_{L^2(Q_+)}  + \sigma^3 \norm{e^{\sigma\phi}w}_{L^2(\Sigma_+ \cup \Gamma_T)}^2   +  \sigma \norm{e^{\sigma\phi} \nabla w}_{L^2(\Sigma_+ \cup \Gamma_T)}^2.
  \end{multline}
  For the terms over $\Gamma_T$, using the energy estimate in Lemma \ref{lemma:energy_T} one has
 \begin{align} 
 \nonumber \norm{w}_{L^2(\Gamma_T)}^2   + \norm{ \nabla w}_{L^2(\Gamma_T)}^2    
   &\lesssim \norm{w}_{H^1(\Gamma)}^2  +  \norm{ (\square + 2\E  \cdot \nabla + f )w}^2_{L^2(Q_+)}  +  \norm{w}_{L^2(\Sigma_+ )}^2   + \norm{ \nabla w}_{L^2(\Sigma_+ )}^2 \\
\label{eq:carl_3B} &\lesssim \norm{w}_{H^1(\Gamma)}^2  +  \norm{\A }_{L^2(\Gamma )}^2 +  \norm{q}_{L^2(\Gamma )}^2 +  \norm{w}_{L^2(\Sigma_+ )}^2  + \norm{ \partial_\nu w}_{L^2(\Sigma_+ )}^2,
  \end{align}
where in the last line  we have used \eqref{eq:condw+w-}, which implies that 
\[\norm{ (\square + 2\E  \cdot \nabla  +f)w}^2_{L^2(Q_+)} \lesssim \norm{\A }_{L^2(B )}^2 +  \norm{q}_{L^2(B)}^2 \lesssim \norm{\A }_{L^2(\Gamma )}^2 +  \norm{q}_{L^2(\Gamma )}^2, \]
 since $\A $ and $q$ do not depend on $t$. We now multiply \eqref{eq:carl_3B} by  $e^{\sigma \sup_{\Gamma_T} \phi}$. Then by Lemma \ref{lemma:weight} one can  use in the right hand side that $\sup_{\Gamma_T} \phi \le \inf_\Gamma \phi -\delta$ for some $\delta>0$. This gives
\begin{multline*}
\sigma^3 \norm{e^{\sigma\phi}w}_{L^2(\Gamma_T)}^2    +  \sigma \norm{e^{\sigma\phi} \nabla w}_{L^2(\Gamma_T)}^2   
 \lesssim  \sigma^3 e^{-2\delta \sigma} ( \norm{e^{\sigma\phi}w}_{H^1(\Gamma)}^2  +   \norm{e^{\sigma\phi}\A }_{L^2(\Gamma )}^2 + \norm{e^{\sigma\phi}q}_{L^2(\Gamma )}^2) \\ 
 +  \sigma^3 e^{2 \sigma \sup_{\Gamma_T }\phi }(  \norm{e^{\sigma\phi}w}_{L^2(\Sigma_+)}^2   + \norm{e^{\sigma\phi} \partial_\nu w}_{L^2(\Sigma_+)}^2).
\end{multline*}
Inserting this estimate in \eqref{eq:carl_3}, and taking $\sigma$ large enough to absorb the term $\sigma^3  e^{-2\delta \sigma}  \norm{e^{\sigma\phi}w}_{H^1(\Gamma)}^2 $ on the left, yields
\begin{multline} \label{eq:carl_4}
\sigma^2 \norm{e^{\sigma\phi}w}_{L^2(\Gamma)}^2   + \norm{e^{\sigma\phi} \nabla_\Gamma w}_{L^2(\Gamma)}^2 + \sigma \int_{ \Gamma } e^{2\sigma \phi} F^j (x,\sigma w, \nabla w)\nu_j \, dS   
   \lesssim  \norm{e^{\sigma  \phi} (\square + 2\E  \cdot \nabla  +f)w}^2_{L^2(Q_+)} \\ 
   + \sigma^3 e^{-2\delta \sigma}  \left (\norm{e^{\sigma\phi}\A }_{L^2(\Gamma )}^2 +\norm{e^{\sigma\phi}q}_{L^2(\Gamma )}^2 \right ) + \sigma^3 e^{c \sigma} \left (  \norm{w}_{L^2(\Sigma_+)}^2   + \norm{ \nabla w}_{L^2(\Sigma_+)}^2 \right).
  \end{multline}
By \eqref{eq:condw+w-2} we can relate the left hand side of the previous inequality with $h$ using that
\[ \norm{ e^{\sigma  \phi} h  }_{L^2(\Gamma)}^2   \lesssim  \norm{ e^{\sigma  \phi} (Z - \omega  \cdot \E )w}_{L^2(\Gamma)}^2 , \]
and that for $\sigma$ large enough 
  \[  \norm{ e^{\sigma  \phi} (Z - \omega  \cdot \E )w}_{L^2(\Gamma)}^2   \lesssim \sigma^2 \norm{e^{\sigma\phi}w}_{L^2(\Gamma)}^2   + \norm{e^{\sigma\phi} Z w}_{L^2(\Gamma)}^2  \le \sigma^2 \norm{e^{\sigma\phi}w}_{L^2(\Gamma)}^2   + \norm{e^{\sigma\phi} \nabla_\Gamma w}_{L^2(\Gamma)}^2. \]
  Inserting these two estimates in \eqref{eq:carl_4} gives
  \begin{multline} \label{eq:carl_5}
\norm{ e^{\sigma  \phi} h  }_{L^2(\Gamma)}^2 + \sigma \int_{ \Gamma } e^{2\sigma \phi} F^j (x,\sigma w, \nabla w)\nu_j  \, dS  
   \lesssim  \norm{e^{\sigma  \phi} (\square + 2\E  \cdot \nabla  +f)w}^2_{L^2(Q_+)} \\     + \sigma^3 e^{-2\delta \sigma}  \left (\norm{e^{\sigma\phi}\A }_{L^2(\Gamma )}^2 +\norm{e^{\sigma\phi}q}_{L^2(\Gamma )}^2 \right )  + \sigma^3 e^{c \sigma} \left (   \norm{w}_{L^2(\Sigma_+)}^2   + \norm{ \nabla w}_{L^2(\Sigma_+)}^2 \right).
  \end{multline}
By \eqref{eq:condw+w-} and Lemma \ref{lemma:weight} we have
\begin{multline*}
 \norm{e^{\sigma  \phi} (\square +2 \E  \cdot \nabla  +f)w}^2_{L^2(Q_+)} \lesssim  \norm{e^{\sigma  \phi} \A }^2_{L^2(Q_+)} +\norm{e^{\sigma  \phi} q}^2_{L^2(Q_+)}  \\
 \lesssim \kappa (\sigma) \left ( \norm{e^{\sigma  \phi} \A }^2_{L^2(\Gamma)} + \norm{e^{\sigma  \phi} q}^2_{L^2(\Gamma)} \right ),
 \end{multline*}
  where $\kappa(\sigma) \to 0 $ as $\sigma \to \infty $. Using this in  \eqref{eq:carl_5} yields
    \begin{multline*} 
\norm{ e^{\sigma  \phi} h  }_{L^2(\Gamma)}^2 + \sigma \int_{ \Gamma } e^{2\sigma \phi} F^j (x,\sigma w, \nabla w)\nu_j \,  dS \\   
   \lesssim   \gamma (\sigma) \left ( \norm{e^{\sigma  \phi} \A }^2_{L^2(\Gamma)} + \norm{e^{\sigma  \phi} q}^2_{L^2(\Gamma)} \right )+ \sigma^3 e^{c \sigma} \left ( \norm{w}_{L^2(\Sigma_+)}^2   + \norm{ \nabla w}_{L^2(\Sigma_+)}^2 \right),
  \end{multline*}
  where $\gamma (\sigma ) := \kappa(\sigma) + \sigma^3 e^{-2\delta \sigma} $ also satisfies $\gamma (\sigma) \to 0 $ as $\sigma \to \infty $. We now use  that $\phi(y,z,z) = e^{\lambda |x-v|^2} = \phi_0(x)$, which means that we can  write the previous estimate changing the $L^2(\Gamma)$ norms to  $L^2(B)$ norms (this is possible since the integrands do not depend on $t$ any more). Also, \eqref{eq:tangential}
  and \eqref{eq:F} imply that 
\[F^j(x,\sigma w, \nabla w)\nu_j |_{\Gamma} = F^j(x,\sigma w, \nabla_\Gamma w)\nu_j |_{\Gamma} ,\]
that is, $F^j\nu_j$ on $\Gamma$ only depends on the part of the gradient of $w$ tangential to $\Gamma$. 
    Applying these observations and rewriting the previous estimate with $w = w_+$ and $\A  = \A_+$, yields
      \begin{multline} \label{eq:carl_6}
\norm{ e^{\sigma  \phi_0} h_+}_{L^2(B)}^2 + \sigma \int_{ \Gamma } e^{2\sigma \phi_0} F^j (x,\sigma w_+, \nabla_\Gamma w_+)\nu_j  \, dS \\   
   \lesssim   \gamma (\sigma) \left ( \norm{e^{\sigma  \phi_0}\A_+}_{L^2(B)}^2 + \norm{e^{\sigma  \phi_0}q_+}_{L^2(B)}^2 \right) + \sigma^3 e^{c \sigma} \left(   \norm{w_+}_{L^2(\Sigma_+)}^2   + \norm{ \nabla w_+}_{L^2(\Sigma_+)}^2 \right),
  \end{multline} 
  Fix $\nu$  to be the downward pointing unit normal to $\Gamma$, so $\nu$ is an exterior normal for $Q_+$. An analogous argument in $Q_-$ yields the estimate 
        \begin{multline} \label{eq:carl_7}
\norm{ e^{\sigma  \phi_0} h_-}_{L^2(B)}^2 - \sigma \int_{ \Gamma } e^{2\sigma \phi_0} F^j (x,\sigma w_-, \nabla_\Gamma w_-)\nu_j  \, dS \\   
   \lesssim   \gamma (\sigma) \left ( \norm{e^{\sigma  \phi_0}\A_-}_{L^2(B)}^2 + \norm{e^{\sigma  \phi_0}q_-}_{L^2(B)}^2 \right) + \sigma^3 e^{c \sigma} \left(   \norm{w_-}_{L^2(\Sigma_-)}^2   + \norm{ \nabla w_-}_{L^2(\Sigma_-)}^2 \right),
  \end{multline} 
  where the minus sign in the boundary term comes from the fact that the outward pointing normal at  $\Gamma$ seen as part of the boundary of $Q_-$ is the opposite to that of the case of $Q_+$.
Now, since $F^j(x,\sigma w, \nabla_\Gamma w)  \nu_j$  is quadratic in  $w$ and $\nabla_\Gamma w$ (and the coefficients are bounded functions), we have that 
\begin{multline*}
 \int_{ \Gamma } e^{2\sigma \phi_0} |F^j (x,\sigma w_+, \nabla_\Gamma w_+) \nu_j- F^j (x,\sigma w_-, \nabla_\Gamma w_-) \nu_j | \, dS \\
  \lesssim \sigma^2 e^{c \sigma} \norm{w_+ - w_-}_{H^1(\Gamma)}( \norm{w_+}_{H^1(\Gamma)} + \norm{ w_-}_{H^1(\Gamma)} )  . 
 \end{multline*}
Therefore, adding \eqref{eq:carl_6} and \eqref{eq:carl_7}  and applying the previous estimate gives the desired result.
\end{proof}

Lemma \ref{lemma:w+w-} is going to be used for two different purposes and, as a consequence, it will be convenient to restate the estimate  in a more specific way. We do this in the following couple of lemmas.

\begin{Lemma} \sl  \label{lemma:carl_1}
 Let  $T>7$. Let $\E_\pm$, $\A_\pm$, $f_\pm$, $q_\pm$, and $\gamma (\sigma)$ be as in Lemma \ref{lemma:w+w-}, and suppose that $h_\pm = q_\pm$. 
Assume that for a fixed $\omega \in S^{n-1}$ there exist $w_\pm \in H^2(Q_\pm)$ such that \eqref{eq:condw+w-} and \eqref{eq:condw+w-2} hold with  $w_+ = w_-$ on $\Gamma$.  Assume also that $w_\pm|_{\Sigma_\pm} = 0$ and $\partial_\nu w_\pm|_{\Sigma_\pm} = 0 $. Then one has that
  \begin{equation} \label{eq:carlfinal_1}
 \sum_{\pm} \norm{ e^{\sigma  \phi_0} q_\pm}_{L^2(B)}^2 \lesssim   \gamma (\sigma) \sum_{\pm}   \norm{e^{\sigma\phi_0} \A_\pm}_{L^2(B)}^2,
  \end{equation}
  for $\sigma >0$ large enough.
  \end{Lemma}
\noindent  The proof is immediate from Lemma \ref{lemma:w+w-}.
  
  \begin{Lemma} \sl  \label{lemma:carl_2}
 Let $T>7$. Also, let $\A_\pm$,  $q_\pm$, and $\gamma (\sigma)$ be as in Lemma \ref{lemma:w+w-}. 
\\
Suppose that for a fixed $\omega \in S^{n-1}$ there exist $w_\pm$,  $\E_\pm$, and $f_\pm$ satisfying the assumptions of Lemma \ref{lemma:w+w-},  and such that \eqref{eq:condw+w-} and \eqref{eq:condw+w-2} hold with  $w_+ = w_-$ on $\Gamma$ and $h_\pm = \omega \cdot \A_\pm$. Assume also  that $w_\pm|_{\Sigma_\pm} = 0$ and $\partial_\nu w_\pm|_{\Sigma_\pm} = 0 $. Then one has that
  \begin{equation} \label{eq:carlfinal_2_0}
 \sum_{\pm} \norm{ e^{\sigma  \phi_0} \omega \cdot \A_\pm }_{L^2(B)}^2 \lesssim  \gamma (\sigma) \sum_{\pm}   \left( \norm{e^{\sigma\phi_0} q_\pm}_{L^2(B)}^2 +\norm{e^{\sigma\phi_0} \A_\pm}_{L^2(B)}^2 \right) .
  \end{equation}
Moreover, let $\{e_j\}_{1\le j\le n}$ be an orthonormal basis  of $\R^n$, and let $J \subseteq \{1,\dots,n\}$ be the set of natural numbers satisfying that at least one of the components $e_j \cdot \A_+ $ and $e_j \cdot \A_- $ does not vanish completely in $\R^n$.
 Suppose that for each $j\in J$, estimate \eqref{eq:carlfinal_2_0} holds with  $\omega = e_j$. Then one has that
  \begin{equation} \label{eq:carlfinal_2}
 \sum_{\pm} \norm{ e^{\sigma  \phi_0}  \A_\pm }_{L^2(B)}^2 \lesssim  \gamma (\sigma) \sum_{\pm}   \norm{e^{\sigma\phi_0} q_\pm}_{L^2(B)}^2 ,
  \end{equation}
for $\sigma >0$ large enough.
\end{Lemma}
\begin{proof}
Estimate \eqref{eq:carlfinal_2_0} follows immediately from Lemma \ref{lemma:w+w-} under the assumptions in the statement. Let us prove \eqref{eq:carlfinal_2}.
 By assumption for each $j \in J$ we have
 \begin{equation} \label{eq:carlfinal2_primitive}
 \sum_{\pm} \norm{ e^{\sigma  \phi_0} e_j \cdot \A_\pm }_{L^2(B)}^2 \lesssim  \gamma (\sigma) \sum_{\pm}   \left(  \norm{e^{\sigma\phi_0} \A_\pm}_{L^2(B)}^2 + \norm{e^{\sigma\phi_0} q_\pm}_{L^2(B )}^2 \right).
  \end{equation}
 But notice that, by definition, $e_j \cdot \A_\pm = 0$ in $\R^n$ if $j \notin J$, so that $\A_\pm = \sum_{j\in J} (e_j \cdot \A_\pm)e_j $. This means that adding \eqref{eq:carlfinal2_primitive} for all $j\in J$ gives
  \begin{equation*} 
 \sum_{\pm} \norm{ e^{\sigma  \phi_0}  \A_\pm }_{L^2(B)}^2 \lesssim  \gamma (\sigma) \sum_{\pm}   \left(  \norm{e^{\sigma\phi_0} \A_\pm}_{L^2(B)}^2 + \norm{e^{\sigma\phi_0} q_\pm}_{L^2(B )}^2 \right).
  \end{equation*}
  Then, using that $\gamma (\sigma) \to 0$ as $\sigma \to \infty$, to finish the proof is enough to take $\sigma$ large enough in order to absorb in the left hand side the first term on the right.
\end{proof}

We have already seen in Section \ref{sec:direct_scat_prob} that the restriction to $\Gamma$ of the solutions of \eqref{eq:wave_1} and \eqref{eq:wave_2} is in a certain sense related to the coefficients of the perturbation. In the next section we use this fact to construct  appropriate functions $w_+$ and $w_-$ so that estimates \eqref{eq:carlfinal_1} and \eqref{eq:carlfinal_2} hold simultaneously with $\A_\pm$ and $q_\pm$ being quantities related to the differences of potentials $\A_1 - \A_2 $ and $q_1 - q_2$. This will yield the proofs of Theorems  \ref{thm:2n_magnetic} and \ref{thm:2n_drift_nogauge}, since the function $\gamma(\sigma)$ in the previous lemmas goes to zero when $\sigma \to \infty$.

\section{Proof of the uniqueness theorems with 2{\it n} measurements} 
\label{sec:2n}

With  Lemmas \ref{lemma:carl_1} and \ref{lemma:carl_2}  we can finally prove the uniqueness results, Theorem \ref{thm:2n_magnetic} and \ref{thm:2n_drift_nogauge}. 
Both theorems are  stated in terms of pairs of measurements, one for each direction $\pm e_j$.  Due to this fact, it is  convenient to give an explicit expression in the same coordinate system for solutions of \eqref{eq:wave_1} and \eqref{eq:wave_2} that correspond to the opposite directions $ \pm \omega_0$, where $\omega_0 \in S^{n-1}$ is fixed. 

Let $x \in \R^n$, and write $x = (y,z)$, where $y\in \R^{n-1}$ and $z= x\cdot \omega_0$.   We want to express in these coordinates the functions $U_{H,\pm}(y,z,t) := U_{H}(y,z,t; \pm\omega_0)$  given by Proposition \ref{prop:sol_smooth_1} when $\omega= \omega_0$ and when $\omega= -\omega_0$. 
By Proposition \ref{prop:sol_smooth_1}, in this coordinate system the solutions satisfy that   $U_{H,\pm}(y,z,t)  = u_\pm(y,z,t)H(t- (\pm z))$ with
\begin{equation} \label{eq:struct_1}
\begin{cases}
(\partial_t^2 + L_{\W,V}) u_\pm =0  \qquad \text{in }\; \; \left\{ t > \pm z \right\}, \\ 
u_\pm(y,z, \pm z) = e^{ \psi_{\pm}(y, z)} 
\end{cases}
\end{equation}
where  
\begin{equation*} 
 \psi_{\pm}(y, z) := \pm \int_{-\infty}^{\pm z}  \omega_0 \cdot\W(y, \pm s) \,ds.
\end{equation*}
In the case of Proposition \ref{prop:sol_smooth_2}, it is convenient to state the results in the case where the Hamiltonian is written in the form   \eqref{eq:hamilt_Aq}. This is easily obtained  making the change $\W = -i\A$ and $V= \A^2 + D\cdot \A + q$ in the previous results. Therefore, using the same coordinates as before, the solutions $U_{\delta,\pm}(y,z,t):= U_{\delta}(y,z,t;\pm \omega_0)$ of \eqref{eq:wave_2}  given by Proposition \ref{prop:sol_smooth_2}  satisfy
\[U_{\delta,\pm} (y,z,t)  = v_\pm(y,z,t)H(t- \pm z) + e^{\psi_\pm (y,z)} \delta(t- \pm z)  \quad \text{in} \; \R^n,\] 
where $ \psi_\pm (y,z) :=  \pm (-i)   \int_{-\infty}^{\pm z}  \omega_0 \cdot \A(y, \pm s) \,ds$,
and $v_\pm$ satisfies
\begin{equation} \label{eq:struct_2}
\begin{cases}
(\partial_t^2 + L_{\W,V}) v_\pm =0  \qquad \text{in}\; \left\{ t> \pm z \right\}, \\ 
v_\pm(y,z, \pm z) =  \frac{1}{2} e^{  \psi_\pm (y,z)} \int_{-\infty}^{\pm z} \left [\nabla \cdot \left( i\A + \nabla \psi_\pm \right ) + \left( i\A + \nabla \psi_\pm \right )^2 -q \right ](y, \pm  s) \, ds.
\end{cases}
\end{equation}

Having these explicit coordinate expressions, we now prove Theorems \ref{thm:2n_magnetic} and \ref{thm:2n_drift_nogauge}. We start with the second one which is the simplest, since the zeroth order term $V$ is fixed. The proof consists in the construction of two appropriate functions $w_+$ and $w_-$ using the solutions $U_{k, \pm j}$ of \eqref{eq:wave_1} in order to apply the results introduced in the previous section. For a fixed direction $\omega = e_j$ and $k=1,2$, we use $U_{k, + j}$ to construct  $w_+$ and, with a certain gauge change, we use $U_{k, - j}$ to construct  $w_-$. The gauge change is a technical requirement necessary to have $w_+ = w_-$ in $\Gamma$, as assumed in Lemmas \ref{lemma:carl_1} and \ref{lemma:carl_2}. Notice that the reason for this assumption comes from the fact that one wants to get rid of the  first term on the right hand side of the estimate \eqref{eq:est_0Carl} which is large when the parameter $\sigma$ is large.

\begin{proof}[Proof of Theorem \ref{thm:2n_drift_nogauge}]

By Lemma \ref{lemma:equivalence_data} we know that it is completely equivalent to consider that the $U_{k, \pm j}$ are solutions of the initial value problem
\begin{equation*} 
( \partial_t^2   - \Delta +  2\W_k \cdot \nabla + V) U_{k, \pm j} = 0, \; \, \text{in}\; \R^{n+1}, \quad U_{k, \pm j}|_{\left\{t<-1\right\}}= H (t- \pm x_j).
\end{equation*}
 instead of the IVP \eqref{eq:drift2n_wave}. Here it is convenient to work with $H$-waves instead of $\delta$-waves since, as mentioned previously, the former have boundary values on $\Gamma$ that are independent of $V$ (see Proposition \ref{prop:sol_smooth_1}).
 
 Fix  $1 \le j \le n$.  In this proof we use  coordinates $x= (y,z)$ in $\R^n$, where $z= x_j$ and $y \in \R^{n-1}$. Let $k=1,2$. By the previous discussion, we know that the $U_{k,\pm j}$ functions satisfy that
\begin{equation} \label{eq:ust}
 U_{k,\pm j}(y,z,t) = u_{k,\pm j}(y,z,t)H(t-\pm z) ,
 \end{equation}
where $u_{k,\pm j}$ is given by \eqref{eq:struct_1} with $\W =\W_k$,  $\omega = \pm e_j$, and $\omega_0 = e_j$. Writing this in detail we obtain that
\begin{equation}
\begin{cases}
(\square + 2\W_k \cdot \nabla + V)u_{k,\pm j} =0 \quad \text{in} \quad \{t> \pm z\}, \\
u_{k,\pm j}(y,z,\pm z) = e^{\int_{-\infty}^{\pm z} (\pm e_j) \cdot\W_k(y,\pm s) \, ds}.
\end{cases}
\label{eq:drift_2n_ivp}
\end{equation}
By assumption we  have that $U_{1,\pm j} = U_{2,\pm j}$ on $\Sigma_\pm = \Sigma \cap \{ t \ge   \pm z\}$, so that
\begin{equation} \label{eq:W_assumption}
u_{1,\pm j} = u_{2, \pm j} \quad \text{on} \quad  \Sigma_\pm,
\end{equation}
 and hence, in particular $u_{1,j} = u_{2, j}$   in $  \Sigma_+ \cap \Gamma = \Sigma \cap \{ t =    z\}$. From this  and \eqref{eq:drift_2n_ivp} we get that there is a function $\mu_j : \R^{n-1} \to \C$ such that
\begin{equation} \label{eq:mu}
\mu_j(y): =  \int_{-\infty}^{\infty} e_j \cdot \W_1(y, s) \, ds = \int_{-\infty}^{\infty} e_j \cdot \W_2(y, s) \, ds  .
\end{equation}
We now define $w_+ := u_{1,j} - u_{2,j}$ in $\ol{Q}_+$. With this  choice Proposition \ref{prop:sol_smooth_1} yields that $w_+ \in C^2(\ol{Q}_+)$, and \eqref{eq:drift_2n_ivp} that
\begin{equation*}
w_+(y,z, z) = e^{\int_{-\infty}^z e_j \cdot \W_1(y,s) \, ds} - e^{\int_{-\infty}^z e_j \cdot \W_2(y,s) \, ds}.
\end{equation*}

 To apply Lemma \ref{lemma:w+w-} and Lemma \ref{lemma:carl_2} we need also to define an appropriate function $w_-$ in $\ol{Q}_-$. To obtain a useful choice  we now consider the solutions $u_{k,-j}$ of \eqref{eq:drift_2n_ivp}
and we take 
\[w_-(y,z,t) = e^{\mu_{j}(y)}(u_{1,-j}(y,z,-t) -u_{2,-j}(y,z,-t)).\]
 Then   $w_- \in C^2(\ol{Q}_-)$ as desired. Also, if $t=z$
\begin{equation}  \label{eq:vanishing_cond_W_proof}
\begin{aligned}
 w_-(y,z,z) &= e^{\mu_{j}(y)} u_{1,-j}(y,z,-z) - e^{\mu_{j}(y)} u_{2,-j}(y,z,-z)\\
 &=  e^{   \mu_{j}(y) + \int_{-\infty}^{-z} (-e_j) \cdot \W_1(y,-s) \, ds} - e^{  \mu_{j}(y) + \int_{-\infty}^{-z} (-e_j) \cdot \W_2(y,-s) \, ds} \\ 
 &= e^{ \mu_{j}(y) -   \int_{z}^{\infty} e_j \cdot \W_1(y,s) \, ds} - e^{ \mu_{j}(y) - \int_{z}^{\infty} e_j \cdot \W_2(y,s) \, ds}  = w_+(y,z,z).
\end{aligned}
\end{equation}
Therefore $w_+$ satisfies
\begin{equation} \label{eq:w+s}
\begin{alignedat}{2}
(\square + 2\W_1 \cdot \nabla + V)w_+ &= 2(\W_2 -\W_1) \cdot \nabla u_{2,j}  & \quad &\text{in} \quad Q_+, \\
(\partial_t + \partial_z -e_j \cdot \W_1)w_+ &=  e_j \cdot (\W_1-\W_2) u_{2,j} &  \quad  &\text{in} \quad \Gamma ,
\end{alignedat}
\end{equation}
and $w_- $ satisfies
\begin{equation*}
\begin{alignedat}{2}
&(\square + 2 (\W_1 + \nabla \mu_{j}) \cdot \nabla + f_-) w_- = 2e^{\mu_{j}} (\W_2 -\W_1) \cdot \nabla u_{2,-j}(x,-t)& \quad &\text{in} \quad Q_-, \\
&(\partial_t + \partial_z -e_j \cdot \W_1) w_- =  e^{\mu_{j}} \, e_j \cdot (\W_1-\W_2) u_{2,-j}(y,z,-z) & \quad &\text{in} \quad \Gamma,
\end{alignedat}
\end{equation*}
where  $f_- = V-|\nabla \mu_{j}|^2 + \Delta \mu_{j} -2\W_1 \cdot \nabla \mu_{j}$. 

Observe  that $e^{\mu_{j}(y)}$, $u_{2,\pm j}(x,\pm t)$ and $|\nabla u_{2,\pm j}(x,\pm t) |$  are bounded functions  in $\ol{Q}_\pm$. Also, we have  that $|u_{2,\pm j}(x, \pm t)| \gtrsim 1$ in $\Gamma$ by \eqref{eq:drift_2n_ivp}. Therefore it holds that
\begin{equation*}
\begin{alignedat}{2}
|(\square + 2 \E_\pm \cdot \nabla + f_\pm )w_\pm| &\lesssim   |\W_2 - \W_1|&  \quad &\text{in} \quad Q_\pm, \\
|(\partial_t + \partial_z -e_j \cdot \E_\pm) w_\pm| &\gtrsim |e_j \cdot (\W_1-\W_2)| & \quad &\text{in} \quad \Gamma,
\end{alignedat}
\end{equation*}
with $f_+ = V$, $\E_+ = \W_1$, $\E_- = \W_1 +\nabla \mu_{j}$, and $f_-$ as before (notice that $e_j \cdot \E_\pm = e_j \cdot  \W_1$ since $e_j \cdot \nabla \mu_j =0$). Hence \eqref{eq:condw+w-} and \eqref{eq:condw+w-2} are satisfied with $\A_\pm  = \W_1-\W_2$, $q_\pm  = 0$, and $h_\pm = e_j \cdot (\W_1-\W_2)$. \\
On the other hand, \eqref{eq:W_assumption} implies that  $w_+|_{\Sigma_+} = w_-|_{\Sigma_-} =  0$. Then, applying Lemma \ref{lemma:normal_estimate} respectively  with  $\alpha = w_+$ and $\chi =1$, or $\alpha = e^{-\mu_j}w_-$ and $\chi= e^{\mu_j}$ yields that $\partial_\nu w_\pm |_{\Sigma_\pm}  =  0$.  \\
Since $w_+ = w_-$ in $\Gamma$ by \eqref{eq:vanishing_cond_W_proof}, the previous observations  imply that all the conditions to apply Lemma \ref{lemma:carl_2} are satisfied, so that \eqref{eq:carlfinal_2_0} holds for each $j\in \{1,\dots,n\}$  with $\A_\pm  = \W_1-\W_2$ and $q_\pm  = 0$. Therefore, the same lemma implies that \eqref{eq:carlfinal_2} must also hold, and this gives
  \begin{equation*} 
  \norm{ e^{\sigma  \phi_0}  (\W_1 - \W_2) }_{L^2(B)} \le 0.
  \end{equation*}
Hence  $\W_1 =\W_2$, since $e^{\sigma  \phi_0(x)}>0$ always. This finishes the proof.
 \end{proof}
 
 We now go to the remaining case, Theorem \ref{thm:2n_magnetic}. In this result the zero order term is not fixed  so that there is gauge invariance (in fact, to simplify the proof it will be convenient fix a specific gauge). This is a harder proof than the previous one since we need to decouple  information about $q$ from the information about $\A$.

\begin{proof}[Proof of Theorem \ref{thm:2n_magnetic}] Let $k=1,2$.
Since making a change of gauge $\A_k -\nabla f_k$ with $f_k$ compactly supported leaves invariant the measured values $U_{k,\pm j}|_{\Sigma_{+ ,\pm j}}$ we can freely choose a suitable $f_k$ in order to simplify the problem.
In fact, to show that $d\A_1= d\A_2$ it is enough to prove that $\A_1=\A_2$ in a specific fixed gauge.
Under the assumptions in the statement  one can always take 
\[
f_k(x) = \int_{-\infty}^{x_n} e_n \cdot \A_k (x_1,\dots,x_{n-1},s) \, ds,
\]
since \eqref{eq:A_condition} implies that $f_k$ must be compactly supported in $B$. With this choice, one has  that 
\[
e_n \cdot (\A_1 -\nabla f_1) =  e_n \cdot (\A_2 - \nabla f_2) = 0,
\]
 in $\R^n$. 
Therefore, the previous arguments show that from now on we can assume without loss of generality that we have fixed  a gauge such that  $e_n \cdot \A_1 =  e_n \cdot \A_2 = 0$ in $\R^n$. 

Fix   $1 \le j \le n-1$. We again take the coordinates in $\R^n$ given by $x= (y,z)$, where $y\in \R^{n-1}$ and $z= x_j$. Let $k=1,2$.
As in the proof of Theorem \ref{thm:2n_drift_nogauge}, by Lemma \ref{lemma:equivalence_data} we can assume that  $U_{k,\pm j}$ satisfies  the IVP
   \begin{equation}  \label{eq:nn_1}
( \partial_t^2 + (D +\A_k )^2 +  q_k )U_{k, \pm j}=0  \; \, \text{in}\; \R^{n+1}, \quad U_{k, \pm j}|_{\left\{t<-1\right\}} = H (t- \pm  x_j),
\end{equation}
instead of \eqref{eq:magnetic2n_wave}.
 By Proposition \ref{prop:sol_smooth_1} we know that the $U_{k,\pm j}$ have the structure described in \eqref{eq:ust} for $1\le j \le n-1 $ where $u_{k,\pm j}$ satisfies \eqref{eq:struct_1} with $\W = -i\A_k$,  $\omega = \pm e_j$, and $\omega_0 = e_j$. Writing this in detail we obtain that 
\begin{equation}
\begin{cases}
(\p_t^2 + (D +\A_k)^2 +q_k)u_{k,\pm j} =0 \quad \text{in} \quad \{t> \pm z\}, \\
u_{k,\pm j}(y,z,\pm z) = e^{-i\int_{-\infty}^{\pm z} (\pm e_j) \cdot \A_k(y,\pm s) \, ds}.
\end{cases}
\label{eq:nn_2}
\end{equation}

By the assumption that $U_{1,\pm j} = U_{2,\pm j}$ on $\Sigma_\pm $  and \eqref{eq:ust}, we have that   \eqref{eq:W_assumption} holds analogously in this case. Specifically, in $ \Sigma_+ \cap \Gamma$ this  implies that there is a function $\mu_j$ such that 
\begin{equation*}
\mu_j(y) = -i\int_{-\infty}^{\infty} e_j \cdot \A_1(y, s) \, ds = -i\int_{-\infty}^{\infty} e_j \cdot \A_2 (y, s) \, ds.
\end{equation*}
We now define $w_+ := u_{1,j} - u_{2,j}$, so that $w_+ \in   C^2(\ol{Q}_+)$ by Proposition  \ref{prop:sol_smooth_1}, and
\begin{equation*}
w_+(y,z, z) = e^{-i\int_{-\infty}^z e_j \cdot \A_1(y,s) \, ds} - e^{-i \int_{-\infty}^z e_j \cdot \A_2(y,s) \, ds}.
\end{equation*}
 To define $w_-$  we  consider the solutions $u_{k,-j}$ of \eqref{eq:nn_2}
and we take 
\[
 w_-(y,z,t) = e^{\mu_j(y)}u_{1,-j}(y,z,-t) - e^{\mu_j(y)} u_{2,-j}(y,z,-t).
 \]
Then  $w_- \in C^2(\ol{Q}_-)$ as desired. Also, for $t=z$
\begin{equation} \label{eq:vanishing_cond_A_proof}
\begin{aligned}
 w_-(y,z,z) &=   e^{  \mu_{j}(y) - i\int_{-\infty}^{-z} (-e_j) \cdot \A_1(y,-s) \, ds} - e^{ \mu_{j}(y) - i\int_{-\infty}^{-z} (-e_j) \cdot \A_2(y,-s) \, ds} \\ 
 &= e^{\mu_{j}(y) +   i\int_{z}^{\infty} e_j \cdot \A_1(y,s) \, ds} - e^{\mu_{j}(y) + i\int_{z}^{\infty} e_j \cdot \A_2(y,s) \, ds}  = w_+(y,z,z).
\end{aligned}
\end{equation}  
Hence,  if we define 
\begin{equation} \label{eq:Vk}
V_k := \A_k^2  +  D \cdot \A_k  +  q_k,
\end{equation}     
$w_+$ satisfies
\begin{equation*}
\begin{alignedat}{2}
(\square -i2 \A_1 \cdot \nabla + V_1) w_+ &= 2(\A_2 -\A_1) \cdot D u_{2,j} + (V_2 -V_1) u_{2,j}& \quad &\text{in} \quad Q_+, \\
(\partial_t + \partial_z +ie_j \cdot \A_1)w_+ &= - i e_j \cdot (\A_1-\A_2) u_{2,j}& \quad &\text{in} \quad \Gamma ,
\end{alignedat}
\end{equation*}
and $w_- $ satisfies
\begin{equation*}
\begin{alignedat}{2}
&e^{-\mu_{j}}(\square - 2(i\A_1 - \nabla \mu_j) \cdot \nabla + \widetilde{V}_1)w_- = \big( 2 (\A_2 -\A_1) \cdot D  + \widetilde{V}_2 -\widetilde{V}_1 \big)  u_{2,-j} (x,-t)&
 \;  &\text{in} \;\; Q_-, \\
&e^{- \mu_{j}} (\partial_t + \partial_z +i e_j \cdot \A_1) w_- = - i e_j \cdot (\A_1-\A_2) u_{2,-j}(y,z,-z)& \quad &\text{in} \;\; \Gamma,
\end{alignedat}
\end{equation*}
where $\widetilde{V}_k = V_k + |\nabla \mu_{j}|^2 + \Delta \mu_{j} + 2i \A_k \cdot \nabla \mu_{j}$. Since  $|\nabla \mu_j|$ is bounded, we have that
\begin{equation*}
|\widetilde{V}_1 -\widetilde{V}_2| \lesssim  |V_1 -V_2| +  |\A_1 -\A_2| .
\end{equation*} 
On the other hand, $u_{2,\pm j}(x, \pm t)$ and $|\nabla u_{2,\pm j}(x, \pm t) |$  are  also bounded in $\ol{Q}_\pm$, and on $\Gamma$ we have that $|u_{2,\pm j} | \gtrsim 1 $ by \eqref{eq:nn_2}. 
 Therefore, if $f_+ = V_1$,  $f_- = \widetilde{V}_1$, $\E_+ = -i\A_1$, and $\E_- = -i\A_1 + \nabla \mu_j$ (notice that $e_j \cdot \nabla \mu_j = 0$), one gets
\begin{equation*}
\begin{alignedat}{2}
|(\square + 2 \E_\pm \cdot \nabla + f_\pm )w_\pm| &\lesssim |\A_1 - \A_2| +|V_1-V_2|&  \quad &\text{in} \quad Q_\pm, \\
|(\partial_t + \partial_z  - e_j \cdot \E_\pm ) w_\pm|  &\gtrsim   |e_j \cdot (\A_1-\A_2)|&  \quad &\text{in} \quad \Gamma,
\end{alignedat}
\end{equation*}
so that \eqref{eq:condw+w-} and \eqref{eq:condw+w-2} are satisfied with $h_\pm =  e_j \cdot (\A_1-\A_2)$,
 \begin{equation} \label{eq:Aq_choice}
 \A_\pm  = \A_1-\A_2, \quad \text{and} \quad q_\pm  = V_1-V_2.
  \end{equation}

As mentioned previously, \eqref{eq:W_assumption} holds analogously in this case, so that  $w_+|_{\Sigma_+} = w_-|_{\Sigma_-} =  0$. 
Again, applying Lemma \ref{lemma:normal_estimate} respectively  with  $\alpha = w_+$ and $\chi =1$, or $\alpha = e^{-\mu_j}w_-$ and $\chi= e^{\mu_j}$ yields that $\partial_\nu w_\pm |_{\Sigma_\pm}  =  0$. \\
 Also,    \eqref{eq:vanishing_cond_A_proof}
 shows that  $w_+ = w_- $ in $\Gamma$. These assertions hold for each $j =1, \dots,n-1$, and the
  $n$-th component $e_n \cdot (\A_1-\A_2)$ vanishes in $\R^n$. This means that the assumptions  of  Lemma \ref{lemma:carl_2} 
   are satisfied  with $J= \{ 1, \dots,n-1 \}$. \\
   As a consequence  \eqref{eq:carlfinal_2_0} holds for each $j\in J$ and hence
    \eqref{eq:carlfinal_2_0} also holds with the  choices established in \eqref{eq:Aq_choice}. This gives
\begin{equation} \label{eq:magnetic_2n_final1}
\norm{ e^{\sigma  \phi_0}   (\A_1 - \A_2) }_{L^2(B)}^2 \lesssim \gamma (\sigma) \norm{e^{\sigma\phi_0} (V_1-V_2)}_{L^2(B)}^2,
  \end{equation}
  where $\gamma (\sigma) \to 0$ as $\sigma \to \infty$.

We now use the information provided by $U_{k,\pm n}$. We use the same coordinates as before, in this case with $z= x_n$. Since  $e_n \cdot \A_1= e_n \cdot \A_2 =0$ in all $\R^n$, we have that
\[
\psi_\pm (y,z) = -(\pm i) \int_{-\infty}^{\pm z}  e_n \cdot \A_k(y,\pm s)  \, ds = 0 .
\]
 Then Proposition \ref{prop:sol_smooth_2} and \eqref{eq:struct_2} yield 
\[ U_{k,\pm n}(y,z,t) =  v_{k,\pm n}(y,z,t)H(t-\pm z) + \delta(t - \pm z), \]
where, using the notation introduced in \eqref{eq:Vk} we have
\begin{equation} \label{eq:hopelast}
\begin{cases}
(\p_t^2 + (D +\A_k)^2 +q_k)v_{k,\pm n} =0 \quad \text{in} \quad \{t> \pm z\}, \\
v_{k,\pm n}(y,z,\pm z) = -\frac{1}{2}\int_{-\infty}^{\pm z}  V_k  (y,\pm s) \, ds.
\end{cases}
\end{equation}
The assumption $U_{1,\pm n}|_{\Sigma_\pm } = U_{2,\pm n}|_{\Sigma_\pm }$ in the statement implies that 
\begin{equation} \label{eq:vcond}
v_{1,\pm n}|_{\Sigma_\pm } = v_{2,\pm n}|_{\Sigma_\pm},
\end{equation} 
and since $\A_k$ and $q_k$ are compactly supported this means that
\begin{equation} \label{eq:boundary_n}
v_{1, n}|_{\Sigma_+ \cap \Gamma } = -\frac{1}{2}\int_{-\infty}^{\infty}  V_1  (y,s) \, ds = -\frac{1}{2}\int_{-\infty}^{\infty}  V_2 (y,s) \, ds =  v_{2,n}|_{\Sigma_+ \cap \Gamma }, 
\end{equation}
We define 
\[ 
w_+ (y,z,t) = 2(v_{1, n} - v_{2, n})(y,z,t)  \quad  \text{and}  \quad  w_-(y,z,t) =  -2(v_{1, -n} - v_{2, -n})(y,z,-t), 
\]
so that $w_\pm \in C^2(\ol{Q}_\pm)$ by Proposition \ref{prop:sol_smooth_2}.  The combination of the identities \eqref{eq:hopelast} and \eqref{eq:boundary_n} and a change of variables shows that
\begin{align}
\nonumber w_+ (y,z,z) - w_- (y,z,z) &= -\int_{-\infty}^{ z}  (V_1-V_2)(y, s) \, ds    - \int_{-\infty}^{-z }  (V_1-V_2)(y,- s) \, ds \\
\label{eq:A_canelation} &= -\int_{-\infty}^{ \infty}  (V_1-V_2)(y, s) \, ds = 0.
\end{align}
Also, by direct computation \eqref{eq:hopelast} yields 
\begin{equation*}
\begin{alignedat}{2}
|(\square -i2 \A_1 \cdot \nabla + V_1)w_\pm| &\lesssim |\A_1 - \A_2| +|V_1-V_2 |&  \quad &\text{in} \quad Q_\pm, \\
|(\partial_t + \partial_z + ie_n \cdot \A_1 ) w_\pm|  &=   |V_1-V_2|&  \quad &\text{in}  \quad \Gamma.
\end{alignedat}
\end{equation*}
 Hence \eqref{eq:condw+w-} and \eqref{eq:condw+w-2} are satisfied with $h_\pm = V_1-V_2$, $\E_\pm =-i\A_1$, $f_\pm =V_1$, and $\A_\pm $ and $q_\pm$ as in \eqref{eq:Aq_choice}.
Also, \eqref{eq:vcond} implies that  $w_+|_{\Sigma_+} = w_-|_{\Sigma_-} =  0$, and \eqref{eq:A_canelation}   that  $w_+|_\Gamma= w_-|_\Gamma$. The condition $\partial_\nu w_\pm|_{{\Sigma_\pm}} = 0$ follows again from Lemma \ref{lemma:normal_estimate}. 

Therefore  all the assumptions to apply  Lemma \ref{lemma:carl_1} hold with the previous choices for $\A_\pm$ and $q_\pm$, and this yields
\begin{equation} \label{eq:magnetic_last}
  \norm{ e^{\sigma  \phi_0} (V_1-V_2)}_{L^2(B)}^2 \lesssim    \gamma (\sigma)  \norm{e^{\sigma\phi_0} (\A_1-\A_2)}_{L^2(B)}^2 .
\end{equation}
Since $\gamma (\sigma) \to 0$ as $\sigma \to \infty$, combining the previous estimate with \eqref{eq:magnetic_2n_final1} immediately implies that $\A_1-\A_2 =0$  and  $V_1 -V_2 = 0$ in this gauge. In a general gauge then $d(\A_1-\A_2) =0$. And since $V_k = \A_k^2 + D \cdot \A_k + q_k$ one  obtains that $q_1 = q_2$.  This finishes the proof.
\end{proof}

\section{Reducing the number of measurements}

In this section we prove an analogous result to Theorem \ref{thm:2n_drift_nogauge}, in which the number of measurements is reduced to $n$. To compensate this, one needs assume that the potentials have certain symmetries (in fact, each component of $\W$ must satisfy some kind of antisymmetry property). The main change in the proof is  in the definition of $w_-$ in $Q_-$, which now is constructed by symmetry from $w_+$, instead of using new information coming from the solution associated to the opposite direction. For each $0\le j \le $ the symmetry of $e_j \cdot\W$ plays an essential role since it is  necessary  to show that $w_+ = w_-$ in $\Gamma$, in order to extract meaningful results from the Carleman estimate.

\begin{Theorem} \sl   \label{thm:n_drift_nogauge}
Let $\W_1,\W_2 \in C^{\M+2}_c(\R^n; \C^n)$, and  $V \in C^{\M}_c(\R^n; \C)$ with compact support in $B$.  Also, let $1 \le j \le n$ and $k=1,2$, and consider the $n$ solutions $U_{k,  j}$ of
  \begin{equation} \label{eq:drift_n_wave}
( \partial_t^2   - \Delta +  2\W_k \cdot \nabla + V) U_{k, j} = 0, \; \, \text{in}\; \R^{n+1}, \quad U_{k, j}|_{\left\{t<-1\right\}}= \delta (t-  x_j).
\end{equation}
 Assume also that for each $1 	\le j \le n$  there exists an orthogonal transformation $\mathcal O_j$ satisfying   that  $\mathcal O_j(e_j) = -e_j$ and such that
\begin{equation} \label{eq:antisymm_cond}
 e_j \cdot \W_k(x)= - e_j \cdot \W_k(\mathcal O_j(x)) \quad \text{for all} \quad k=1,2.
\end{equation}
Then, if for all $1 \le j \le n$ one has $U_{1,j} = U_{2,j}$ on the  surface $ \Sigma \cap \{ t \ge    x_j \}$, it holds that $\W_1 = \W_2$.
\end{Theorem}
\noindent The simplest example of a vector field satisfying the previous conditions  is the case of an antisymmetric vector fields $\W_k$, that is, such that $\W_k (-x) = -\W_k(x)$ (for example, the gradient of a radial function). \\
We remark that it is possible to show  that the previous theorem  also holds if one substitutes condition \eqref{eq:antisymm_cond} by
\begin{equation*} 
 e_j \cdot \W_k(x)= - e_j \cdot  \ol{\W}_k(\mathcal O_j(x)) \quad \text{for all} \quad k=1,2 \quad \text{and}\quad 1 	\le j \le n,
\end{equation*}
that is, a symmetry condition instead of an antisymmetry condition in the imaginary part of $e_j \cdot\W_k$.

\begin{proof}[Proof of Theorem \ref{thm:n_drift_nogauge}]
By Lemma \ref{lemma:equivalence_data} we know that it is completely equivalent to assume that the $U_{k, j}$ solutions solve 
\begin{equation*}  
( \partial_t^2   - \Delta +  2\W_k \cdot \nabla + V) U_{k, j} = 0, \; \, \text{in}\; \R^{n+1}, \quad U_{k,  j}|_{\left\{t<-1\right\}}= H (t-  x_j).
\end{equation*}
 instead of \eqref{eq:drift_n_wave}. 
 Take  $1 \le j \le n$.  In this proof we fix   $x= (y,z)$ in $\R^n$, where $z= x_j$ and $y \in \R^{n-1}$. Let $k=1,2$. 
 By Proposition \ref{prop:sol_smooth_1} and  \eqref{eq:struct_1} we know  that $U_{k,j}(y,z,t) = u_{k, j}(y,z,t)H(t- z)$ where $u_{k, j}$ is the same function as $u_{k, +j}$ in \eqref{eq:drift_2n_ivp} 
\\
Under the assumption in the statement we  have that $u_{1,j} = u_{2, j}$ on the surface $ \Sigma_+ $ and hence, also in $ \Sigma_+ \cap \Gamma =    \Sigma \cap \{ t  =   z\}$, which implies that there is a function $\mu_j : \R^{n-1} \to \C$ such that \eqref{eq:mu} holds.

 We now define $w_+ := u_{1,j} - u_{2,j}$ in $\ol{Q}_+$, so that 
\begin{equation} \label{eq:w+_n}
w_+(y,z, z) = e^{\int_{-\infty}^z e_j \cdot \W_1(y,s) \, ds} - e^{\int_{-\infty}^z e_j \cdot \W_2(y,s) \, ds}.
\end{equation}
In this coordinates, for each matrix $\mathcal O_j$ in the statement there is by definition  a $n-1 \times n-1$ orthogonal matrix $\mathcal T_j$ such that 
\[   \mathcal O_j (y,z) =  (\mathcal T_j(y),-z) .\]
To apply Lemma \ref{lemma:w+w-} and Lemma \ref{lemma:carl_2} we need also to define an appropriate function $w_-$ in $\ol{Q}_-$. We take 
\[
w_-(y,z,t) = e^{\mu_{j}(y)} w_+( {\mathcal T}_j (y),-z,-t) .
\]
If $t=z$, using \eqref{eq:w+_n} and the symmetry condition \eqref{eq:antisymm_cond}, we get that
\begin{equation} \label{eq:vanishing_cond_W_proof_2}
\begin{aligned} 
 w_-(y,z,z)  &=  e^{   \mu_{j}(y) + \int_{-\infty}^{-z} e_j \cdot \W_1({\mathcal T}_j(y),s) \, ds} - e^{  \mu_{j}(y) + \int_{-\infty}^{-z} e_j \cdot \W_2({\mathcal T}_j(y),s) \, ds} \\ 
&=  e^{   \mu_{j}(y) - \int_{-\infty}^{-z} e_j \cdot \W_1(y,-s) \, ds} - e^{  \mu_{j}(y) - \int_{-\infty}^{-z} e_j \cdot \W_2(y,-s) \, ds} \\ 
 &= e^{ \mu_{j}(y) -   \int_{z}^{\infty} e_j \cdot \W_1(y,s) \, ds} - e^{ \mu_{j}(y) - \int_{z}^{\infty} e_j \cdot \W_2(y,s) \, ds}  = w_+(y,z,z).
\end{aligned}
\end{equation}
Therefore $w_+$ satisfies  \eqref{eq:w+s}, and $w_- $ satisfies
\begin{multline*}
( \square + 2  \E_-\cdot \nabla  + f_-) w_- (y,z,t)  \\
\hspace{30mm} = 2 e^{\mu_{j}(y)} (\W_2 -\W_1)( {\mathcal T}_j(y),-z) \cdot \nabla \big ( u_{2,j}( {\mathcal T}_j(y),-z,-t) \big ) \; \; \text{in} \; \; Q_- ,
\end{multline*}
where
\begin{equation} \label{eq:E_f}
\begin{aligned}
\E_-(y,z) &= \W_1( {\mathcal T}_j(y),-z) + \nabla \mu_{j}(y) \\
 f_- (y,z) &= V ( {\mathcal T}_j(y),-z) - |\nabla \mu_{j}(y)|^2 + \Delta \mu_{j}(y) -2\W_1( {\mathcal T}_j(y),-z) \cdot \nabla \mu_{j}(y).
 \end{aligned}
 \end{equation}
 Also, since $e_j \cdot \W_1( {\mathcal T}_j(y),-z)  = - e_j \cdot \W_1( y,z) $,  we have that
\begin{align*}
  (\partial_t + \partial_z -e_j \cdot \W_1( y,z)) w_- ( y,z,z) &= - e^{\mu_j(y)} \big (\partial_t + \partial_z -  e_j \cdot \W_1( {\mathcal T}_j(y),-z) \big ) w_+({\mathcal T}_j(y),-z,-z)   \\
  &=  - e^{\mu_{j}(y)}
 e_j \cdot (\W_1-\W_2)({\mathcal T}_j(y),z) u_{2,j}({\mathcal T}_j(y),-z,-z),
\end{align*}
Since $ |u_2(y,z,z)| \gtrsim 1$ always, and $|\nabla u_2|$ is bounded above in $\ol{Q}_+$, the previous identities show that 
\begin{equation*}
\begin{alignedat}{2}
|(\square + 2 \E_\pm \cdot \nabla + f_\pm )w_\pm| &\lesssim |\A_\pm| & \quad &\text{in} \quad Q_\pm, \\
|(\partial_t + \partial_z -e_j \cdot \E_\pm) w_\pm| &\gtrsim   |e_j \cdot \A_\pm|&  \quad &\text{in} \quad \Gamma,
\end{alignedat}
\end{equation*}
where $\E_-$ and $f_-$ where defined in  \eqref{eq:E_f}, $\E_+ = \W_1$,  $f_+ = V$, 
\begin{equation} \label{eq:AA_n}
\A_+ = \W_1-\W_2 \quad \text{and} \quad \A_-(y,z) = (\W_1-\W_2)({\mathcal T}_j(y),-z).
\end{equation} 
From \eqref{eq:vanishing_cond_W_proof_2} we get that $w_+ = w_-$ on $\Gamma$. Also we have that $w_+|_{\Sigma_+} = w_-|_{\Sigma_-} =  0$, and  $\partial_\nu w_\pm|_{{\Sigma_\pm}} = 0$ by Lemma \ref{lemma:normal_estimate}. Applying  \eqref{eq:carlfinal_2_0}  of Lemma \ref{lemma:carl_2}  with $q_\pm = 0$ and $\omega = e_j$ yields 
 \begin{equation*}
 \sum_{\pm} \norm{ e^{\sigma  \phi_0} e_j \cdot \A_\pm }_{L^2(B)}^2 \lesssim  \gamma (\sigma) \sum_{\pm}  \norm{e^{\sigma\phi_0} \A_\pm}_{L^2(B)}^2.
  \end{equation*} 
Now, this  holds for $\phi_0(x) = \phi_0(x,\ve) = e^{\lambda\vert x-\ve \vert^2}$, where $\ve$ is an arbitrary vector such that $|\ve| = 2$. Also, the implicit constant  in the estimate is independent of $\ve$. Therefore writing $\ve = 2\theta$ for $\theta \in S^{n-1}$,  we can integrate both sides of the previous estimate in $S^{n-1}$ to get
 \begin{equation*}
\int_{S^{n-1}} \sum_{\pm} \norm{ e^{\sigma  \phi_0(\cdot,2\theta)} e_j \cdot \A_\pm }_{L^2(B)}^2  \, d S(\theta) \lesssim  \gamma (\sigma) \sum_{\pm}  \int_{S^{n-1}} \norm{e^{\sigma \phi_0(\cdot,2\theta)} \A_\pm}_{L^2(B)}^2  \, d S(\theta),
  \end{equation*}
 where $d S(\theta)$ denotes integration   against the surface measure of the unit sphere. Changing the order of integration with the $L^2(B)$ integrals gives
   \begin{equation} \label{eq:tr}
 \sum_{\pm} \norm{ r(x,\sigma) e_j \cdot \A_\pm }_{L^2(B)}^2   \lesssim  \gamma (\sigma) \sum_{\pm} \norm{ r(x,\sigma) \A_\pm}_{L^2(B)}^2 ,
  \end{equation}
 where it can be verified that
 $r(x,\sigma) : =   \left( \int_{S^{n-1}} e^{2\sigma\phi_0(x,2\theta)}  \, d S(\theta) \right)^{1/2}$ is a radial function.
 
Since $r(x)$ is a radial, and  $\A_+$ and $\A_-$ differ in an orthogonal transformation by \eqref{eq:AA_n}, a direct change of variables shows that
   \begin{equation*}
 \norm{ r(x,\sigma) e_j \cdot \A_-}_{L^2(B)}^2   =\norm{ r(x,\sigma) e_j \cdot \A_+ }_{L^2(B)}^2,  \quad \text{and} \quad \norm{ r(x,\sigma) \A_-}_{L^2(B)}^2 =    \norm{ r(x,\sigma) \A_+}_{L^2(B)}^2 ,
  \end{equation*}
  and therefore, taking into account that $\A_+ = \W_1-\W_2$ we get from \eqref{eq:tr} that
   \begin{equation*}
  \norm{ r(x,\sigma) e_j \cdot (\W_1-\W_2) }_{L^2(B)}^2   \lesssim  \gamma (\sigma)  \norm{ r(x,\sigma) (\W_1-\W_2)}_{L^2(B)}^2 .
  \end{equation*}
  This estimate can be proved for any $1 \le j \le n$.  Adding over all directions, and using that  $\gamma (\sigma) \to 0$ as $\sigma \to 0$ to absorb the resulting term on the right hand side in the left, yields
     \begin{equation*}
  \norm{ r(x,\sigma) (\W_1-\W_2) }_{L^2(B)}^2   \le  0 ,
  \end{equation*}
for $\sigma>0$ large enough. Since $r(x,\sigma)>0$ for all $x\in \R^n$, and $\sigma>0$, the previous estimate implies that  $\W_1 =\W_2$. This finishes the proof.
\end{proof}

Combining the techniques used in this proof with the techniques used in  the proof of Theorem \ref{thm:2n_magnetic}, the reader can obtain many possible results similar to the previous one, always interchanging some measurements for symmetry assumptions on $\A$ and $q$. A specially simple case is the following.
\begin{Theorem} \sl   \label{thm:n_magnetic}
Let  $\A_1, \A_2 \in C^{\M+2}_c(\R^n; \C^n)$ and $q_1,q_2  \in  C^{\M}_c(\R^n; \C)$ with compact support in $B$ and such that
 \begin{equation*} 
\A_k(-x) = -\A_k(x) , \quad \text{for} \quad k=1,2. 
 \end{equation*}
Let $1 \le j \le n-1$ and consider the   $n-1$ solutions $U_{k,   j}(x,t)$ of
   \begin{equation*}
( \partial_t^2 + (D +\A_k )^2 +  q_k )U_{k,   j}=0  \; \, \text{in}\; \R^{n+1}, \quad U_{k,   j}|_{\left\{t<-1\right\}} = \delta (t-    x_j).
\end{equation*}
and the $2$ solutions $U_{k,   \pm n}(x,t)$ of
   \begin{equation*}
( \partial_t^2 + (D +\A_k )^2 +  q_k )U_{k,   \pm n}=0  \; \, \text{in}\; \R^{n+1}, \quad U_{k,   \pm n}|_{\left\{t<-1\right\}} = \delta (t-    (\pm x_n)).
\end{equation*}
 If for each $1 \le j \le n$ one has $U_{1,  j} = U_{2,  j}$ on the surface $\Sigma \cap \{ t \ge     x_j \}$, then $d \A_1 = d \A_2$ and $q_1 = q_2$.
\end{Theorem}

\begin{proof}
 We use the notation introduced in \eqref{eq:Vk}. Here is convenient to use that $U_{k,j}$ is the same as $U_{k,+j}$ in the proof of Theorem \ref{thm:2n_magnetic}.   We only give a sketch of the main ideas in the proof. One can start as in the proof of Theorem \ref{thm:2n_magnetic} by making  a change of gauge such that $e_n \cdot \A_1 = e_n \cdot \A_2 $.  Let $0\le j \le n-1$. By Lemma \ref{lemma:equivalence_data} we can assume that the $U_{k,j} =U_{k,+j}$ satisfy \eqref{eq:nn_1} and \eqref{eq:nn_2} with $\W_k = -i\A_k$ and $V= V_k$. We define $w_+ = u_{1,j}-u_{2,j}$ and $w_-(y,z,t) = w_+(-y,-z,-t)$ (the $\mu_j$ function vanishes due to the antisymmetry condition). It follows that $w_+ = w_-$ in $\Gamma$, this can be verified as in \eqref{eq:vanishing_cond_W_proof_2}. The remaining  conditions to apply Lemma \ref{lemma:carl_2}  with $\A_+ =  \A_1 -\A_2$, $\A_- = \A_+(-x)$, $q_+  = V_1-V_2$ and $q_-(x) =q_+(-x)$ are easily verified. Then Lemma \ref{lemma:carl_2} and  a change of variables  to transform $q_-$ in $q_+$, and $\A_-$  in $\A_+$    yields the estimate
 \begin{multline}  \label{eq:last_mag_n}
\norm{ e^{\sigma  \phi_0(\cdot, \vartheta)}   (\A_1 - \A_2) }_{L^2(B)}^2 + \norm{ e^{\sigma  \phi_0(\cdot, -\vartheta)}   (\A_1 - \A_2) }_{L^2(B)}^2 \\
\lesssim \gamma (\sigma) \norm{e^{\sigma\phi_0(\cdot, \vartheta)} (V_1-V_2)}_{L^2(B)}^2 +  \gamma (\sigma) \norm{e^{\sigma\phi_0(\cdot, -\vartheta)} (V_1-V_2)}_{L^2(B)}^2.
  \end{multline}
We can now repeat exactly the same arguments used in the proof of Theorem \ref{thm:2n_magnetic} to prove \eqref{eq:magnetic_last}. In fact we have that \eqref{eq:magnetic_last} holds independently for both the weight functions $\phi_0(\cdot, \vartheta)$ and $\phi_0(\cdot,- \vartheta)$. Adding these two possible estimates yields
 \begin{multline*} 
\norm{e^{\sigma\phi_0(\cdot, \vartheta)} (V_1-V_2)}_{L^2(B)}^2  + \norm{e^{\sigma\phi_0(\cdot, -\vartheta)} (V_1-V_2)}_{L^2(B)}^2 \\
\lesssim \gamma (\sigma)\norm{ e^{\sigma  \phi_0(\cdot, \vartheta)}   (\A_1 - \A_2) }_{L^2(B)}^2 +  \gamma (\sigma) \norm{ e^{\sigma  \phi_0(\cdot, -\vartheta)}   (\A_1 - \A_2) }_{L^2(B)}^2 .
  \end{multline*}
The previous inequality and \eqref{eq:last_mag_n} imply that $q_1 = q_2$ and $\A_1 = \A_2$ in the gauge fixed at the beginning of the proof.
\end{proof}


\appendix

\section{Stationary scattering} \label{appendix:stationary}

In this section  we prove Theorem \ref{thm:equivalence}. We have adapted the proof of \cite[Theorem 5.1]{RakeshSalo2} in order to allow for the presence of a first order perturbation, but the main ideas and the exposition are similar to the work in that paper.

We  define $\C^+: = \{ \lambda \in \C : \, \Im(\lambda)>0 \}$, and we write $R_{\mathcal V}(\lambda) $ for the  resolvent operator $R_{\mathcal V}(\lambda) =(H_{\mathcal V} -\lambda^2)^{-1}$ in case it is well defined.  We also use the following nonstandard convention for the Fourier transform and its inverse for Schwartz functions on the real line:
\[ \tilde{f} (\lambda)  = \int_{-\infty}^\infty e^{i \lambda t} f(t) \, dt \qquad {\breve F} (t)  = \frac{1}{2\pi}\int_{-\infty}^\infty e^{-i \lambda t} F(\lambda) \, d\lambda ,\]
(and equally for the extension of the Fourier transform to  tempered distributions).

In order to illustrate why it is reasonable to expect an equivalence between the stationary scattering data and the time domain data as stated in Theorem \ref{thm:equivalence}, we reproduce here the following heuristic argument given in \cite{RakeshSalo2}. Let $U_{\mathcal V}(x,t;\omega)$ be the solution of
\begin{equation*}
(\partial_t^2 - \Delta + \mathcal V (x,D) ) U_{\mathcal V} =0 \; \; \text{in} \; \R^n \times \R, \qquad U_{\mathcal V}|_{\left\{t<-1\right\}}= \delta (t-x\cdot \omega).
\end{equation*}
Suppose for the moment that the Fourier transform of $U_{\mathcal V}$ in the time variable is well defined. Then for each $\lambda \in \R$ the function $\widetilde{U}_{\mathcal V}(x,\lambda;\omega) $ should solve the equation
\begin{equation} \label{eq:q_outg}
 (- \Delta + \mathcal V(x,D) -\lambda^2 ) \widetilde{U}_{\mathcal V}(x,\lambda)   = 0 \; \; \text{in} \; \R^n.
 \end{equation}
If we define the time domain scattering solution to be $u_{\mathcal V} = U_{\mathcal V } - \delta (t-x\cdot \omega)$,  one has that $\widetilde{U}_{\mathcal V}(x,\lambda) = e^{i\lambda x\cdot \omega} + \widetilde {u}_{\mathcal V}(x,\lambda)$,
 where $\widetilde{u}_{\mathcal V}(x,\lambda)$ extends holomorphically to  $\{ \Im(\lambda)>0 \}$ since $u_{\mathcal V}$ vanishes for $t<-1$. 
 Since these are the properties that characterize the the outgoing eigenfunctions of \eqref{eq:q_outg} one might expect that
 \[
\widetilde{U}_{\mathcal V}(x,\lambda; \omega) = \psi_{\mathcal V}(x,\lambda,\omega),
 \]
where $\psi_{\mathcal V}$ is the solution of \eqref{eq:sta_sct}. Now, the condition $a_{\mathcal V_1}(\lambda,\cdot,\omega) = a_{\mathcal V_2}(\lambda,\cdot,\omega)$ implies by the Rellich uniqueness theorem  that the outgoing eigenfunctions for $H_{\mathcal V_1}$ and $H_{\mathcal V_2}$ agree outside  the support of the potentials:
\begin{equation} \label{eq:out_supp}
\psi_{\mathcal V_1} (\lambda,\cdot,\omega)|_{\R^n \setminus \ol{B} }  = \psi_{\mathcal V_2}(\lambda,\cdot,\omega) |_{\R^n \setminus \ol{B} } .
\end{equation}
If the map $\lambda \to \psi_{\mathcal V}(\lambda,x,\omega)$ were smooth near $\lambda =0$, then one could have \eqref{eq:out_supp}  for all $\lambda \in \R$. Taking the inverse Fourier transform would imply that
\[ U_{\mathcal V_1}(\cdot,t;\omega)|_{\R^n \setminus \ol{B} } = U_{\mathcal V_2}(\cdot,t;\omega)|_{\R^n \setminus \ol{B} }. \]

The argument above is only formal since requires the regularity of the map $\lambda \to \psi_{\mathcal V}(\lambda,x,\omega)$ on the real line. The regularity of this map is related to the poles  of the meromorphic continuation of the resolvent $R_{\mathcal V}(\lambda)$, initially defined in the resolvent set of $H_{\mathcal V}$. Indeed, in some cases there is a pole located at $\lambda =0$ and thus the argument above does not work in general.
To get around these difficulties we start by recalling the following property of the Fourier transform. 
\begin{Lemma} \sl  \label{lemma:holomorphic_Fourier}
Suppose $F(z)$ is analytic on $\{\Im(z) >r\}$ for some $r\in \R$ and 
\[
|F(z)| \le C(1 + |z|)^N e^{R \Im(z)}, \quad \text{for} \;\; \Im(z)>r, 
\]
for some positive $R,C,N$ independent of $z$. There exist an $f \in \mathcal D(\R)$ with $\supp(f)\subset [-R,\infty) $ and $e^{-(\mu-r)t} f \in \mathcal S(\R)$ that also satisfies $(e^{-(\mu-r)t} f)^{\sim}(\cdot) =F(\cdot + i\mu)$ for every $\mu>r$.
\end{Lemma}
\noindent This is essentially a Paley-Wiener theorem that we have stated in the form given in \cite[Lemma 5.3]{RakeshSalo2}.
 In the following proposition we give the precise relation between the time domain and frequency measurements. 

\begin{Proposition} \sl  \label{prop:equivalence_data} 
Let $\omega\in S^{n-1}$ and let  ${\mathcal V}(x,D) = \W \cdot \nabla + V$ with  $\W \in C^{\M+2}_c(\R^n; \C^n)$, and  $V \in C^{\M}_c(\R^n; \C)$ compactly supported in $B$.     Let $U_{\mathcal V}$ be the  solution of
\begin{equation} \label{eq:basic_wave}
(\partial_t^2 - \Delta + \mathcal V (x,D) ) U_{\mathcal V} =0 \; \; \text{in} \; \R^n \times \R, \qquad U_{\mathcal V}|_{\left\{t<-1\right\}}= \delta (t-x\cdot \omega),
\end{equation}
 given by Proposition $\ref{prop:sol_smooth_2}$, and let  $u_{\mathcal V}(x,t;\omega) = U_{\mathcal V} (x,t;\omega) -\delta(t-x \cdot \omega)$. Assume also that there exists some $r \ge 0$ such that for $\Im(\lambda) \ge r$ 
\begin{equation} \label{eq:res_unif}
\norm{R_{\mathcal V}(\lambda ) }_{L^2\to L^2} \le C_{r} ,
\end{equation}
where $C_r>0$ is independent of $\lambda$. Then, if we define
\[
\psi^s_{\mathcal V}(\cdot, \lambda, \omega):=-  R_{\mathcal V}(\lambda)({\mathcal V}(x,D) e^{i\lambda x\cdot \omega} ), \quad \text{for } \; \Im(\lambda) > r,
\]
the following identity holds
\[
\langle u_{\mathcal V} (x,t; \omega), \varphi (x)\chi (t) \rangle_{\R^n_x \times \R_t}= \langle \psi^s_{\mathcal V}(x, \sigma + i\mu, \omega), \varphi(x) (e^{\mu t}\chi )^\vee(\sigma) \rangle_{\R^n_x \times \R_\sigma},
\]
for all $\mu > r$, and all $\varphi \in C^\infty_c(\R^n)$ and $\chi\in C^\infty_c(\R)$.
\end{Proposition}
\begin{proof}
If \eqref{eq:res_unif} holds for some $\lambda$, then $z= \lambda^2$ is, by definition, in the resolvent set $\rho(H_{\mathcal V})\subset \C$ of the operator $H_{\mathcal V}$.  It is well known that the resolvent map $z \to (H_{\mathcal V} -z)^{-1}$ forms a holomorphic family of bounded $L^2$ operators in the open set $\rho(H_{\mathcal V})$ (see for example \cite[Theorem 2.15]{teschl}). Since the map $z= \lambda^2$ is also holomorphic and \eqref{eq:res_unif} holds for $\Im (\lambda) \ge r $, then $\lambda  \to R_{\mathcal V}(\lambda) = (H_{\mathcal V} -\lambda^2)^{-1}$ is also an holomorphic map for  $\Im (\lambda) \ge r $. 

On the other hand, using \eqref{eq:res_unif} we have
\begin{equation} \label{eq:scatt_sol_est}
\left\| \psi_{\mathcal V}^s (\cdot, \lambda, \omega)\right\|_{L^2(\R^n)}\leq C_{r, A, q}(1+ |\lambda|) e^{\Im (\lambda)}, \quad \Im (\lambda) \geq r.
\end{equation}
For any fixed $\varphi \in C^\infty_x(\R^n)$, define
\[
F_\varphi(\lambda)= \int_{\R^n} \psi^s_{\mathcal V}(x, \lambda, \omega)\varphi(x)dx, \quad \Im (\lambda)\geq r.
\]
 From the previous observations it follows that $F_\varphi(\lambda)$ must be an holomorphic function in the set $\Im (\lambda)\geq r $. By  estimate \eqref{eq:scatt_sol_est}, we get
\[
|F_{\varphi}(\lambda)|\leq C_{r,  \A, q}(1+ |\lambda|) e^{\Im (\lambda)}\left\| \varphi\right\|_{L^2}, \quad \Im (\lambda) \geq r.
\]
Then,  Lemma \ref{lemma:holomorphic_Fourier} implies that there is a function $f_\varphi \in \mathcal{D}^\prime(\R)$ supported on $[-1, \infty)$ such that for all $\mu>r$:
\[
\langle e^{-(\mu-r)t}f_\varphi, \chi \rangle = \langle F_\varphi (\cdot + i\mu), \breve{\chi} \rangle, \quad \chi \in C^\infty_c(\R).
\]

Now, given $\mu >r$, define the linear map $\mathcal K: C^\infty_c(\R^n) \to \mathcal D'(\R)$ given by
\[
\mathcal K \varphi  = e^{-(\mu- r)t} f_{\varphi}.
\]
The map $\mathcal{K}$ is continuous. To see this, take a sequence $\varphi_j \to 0$ in $C^\infty_c(\R^n)$, then \eqref{eq:scatt_sol_est}  implies that $F_{\varphi_j} \to 0$ when $\Im(\lambda) \ge r$, and hence 
\[
\langle \mathcal K \varphi_j , \chi \rangle  = \langle e^{-(\mu-r)t}f_{\varphi_j}, \chi \rangle = \langle F_{\varphi_j} (\cdot + i\mu), \breve{\chi} \rangle \to 0 \quad{as} \;\; j\to \infty.
\] 
Since $\mathcal{K}$ is continuous, the Schwartz kernel theorem ensures that there is a unique $K \in  \mathcal{D}'(\R^n\times \R)$ such that 
\begin{align}
\nonumber \langle K, \varphi(x)\chi (t)\rangle &= \nonumber \langle \mathcal K \varphi, \chi  \rangle = \langle e^{-(\mu -r)t} f_\varphi , \varphi(x), \chi \rangle = \langle F_\varphi(\cdot + i\mu)  ,\breve{\chi} \rangle \\
\label{eq:K_psi} &= \langle  \psi^s_{\mathcal V}(x, \sigma +i\mu, \omega) , \varphi(x) \breve{\chi}(\sigma)\rangle_{\R^n_x \times \R_\sigma}.
\end{align}
Since $f_\varphi$ is supported in $[-1,\infty)$, it follows that $K$ is supported in $\{t\ge 1\}$. We now define the distribution 
\[
v(x,t) := e^{\mu t} K(x,t) \in \mathcal{D}^\prime(\R^n\times \R).
\]
We claim that  $v$ is a solution in $\R^{n+1}$ of 
\begin{equation} \label{eq:lst}
(\square +  {\mathcal V}(x,D))v = -{\mathcal V}(x,D) \delta(t-x\cdot\omega) .
\end{equation}
Since, by \eqref{eq:basic_wave}, this is also the equation satisfied by $u_{\mathcal V}$, then the uniqueness  of distributional solutions of the wave equation supported in $\{t\ge -1\}$ (see \cite[Theorem 9.3.2]{H76}) implies that $u_{\mathcal V} = v$, so 
\[
\langle u_{\mathcal V}, \varphi(x)  \chi(t) \rangle  = \langle  K \varphi , \chi(x) e^{\mu t} \chi(t) \rangle  = \langle \psi_{\mathcal V}^s(x, \sigma + i\mu, \omega),  \varphi(x) (e^{\mu t} \chi )\ebreve (\sigma) \rangle_{\R^n_x \times \R_\sigma},
\] 
which finishes the proof of the proposition.

To prove the previous claim we use \eqref{eq:K_psi} in the following computations. First
\begin{multline*}
\langle \partial_t^2 (e^{\mu t} K ), \varphi(x) \chi(t) \rangle 
=  \langle  K , \varphi(x) e^{\mu t} \partial_t^2 \chi(t) \rangle\\
= \langle  \psi_{\mathcal V}^s(x,\sigma + i\mu,\omega), \varphi(x) (e^{\mu t} \partial_t^2 \chi(t)) \ebreve  \rangle= -\langle \psi_{\mathcal V}^s(x,\sigma + i\mu,\omega), \varphi(x) (\sigma + i\mu)^2 (e^{\mu t}  \chi) \ebreve (\sigma) \rangle_{\R^n_x \times \R_\sigma}.
\end{multline*}
Also, if we denote by $\mathcal V^*$ the formal adjoin of $ \ma V$ (with respect to the distribution pairing $\langle \cdot,\cdot\rangle$), we have
\begin{align*}
\langle {\mathcal V}(x,D) ( e^{\mu t} K ), \varphi(x) \chi(t) \rangle 
&=  \langle  K , {\mathcal V}^*(x,D) \varphi(x) e^{\mu t} \chi(t) \rangle \\
&= \langle  {\mathcal V}(x,D) \psi_{\mathcal V}^s(x,\sigma + i\mu,\omega), \varphi(x) (e^{\mu t}  \chi) \ebreve  (\sigma )\rangle_{\R^n_x \times \R_\sigma},
\end{align*}
and similarly one gets
\[
\langle \Delta_x( e^{\mu t} K ), \varphi(x) \chi(t) \rangle 
= \langle  \Delta_x \psi_{\mathcal V}^s(x,\sigma + i\mu,\omega), \varphi(x) (e^{\mu t}  \chi) \ebreve  (\sigma )\rangle_{\R^n_x \times \R_\sigma}.
\]
Then putting this together we obtain that
\begin{align*}
\langle (\partial_t^2    &- \Delta_x + {\mathcal V}(x,D)v, \varphi(x) \chi(t) \rangle \\
&= \langle (-\Delta_x + {\mathcal V}(x,D) -(\sigma + i\mu)^2 ) \psi_{\mathcal V}^s(x,\sigma + i\mu,\omega), \varphi(x) (e^{\mu t}  \chi) \ebreve  (\sigma )\rangle_{\R^n_x \times \R_\sigma} \\
&=  \langle - {\mathcal V}(x,D) e^{i(\sigma + i\mu)x \cdot \omega}, \varphi(x) (e^{\mu t}  \chi) \ebreve  (\sigma )\rangle_{\R^n_x \times \R_\sigma} \\
&= -\langle  e^{i(\sigma + i\mu)x \cdot \omega}, {\mathcal V}^*(x,D)\varphi(x) (e^{\mu t}  \chi) \ebreve  (\sigma )\rangle_{\R^n_x \times \R_\sigma} \\
&= -\langle e^{-\mu x\cdot \omega}\delta(t-x\cdot \omega), {\mathcal V}^*(x,D)\varphi(x) e^{\mu t}  \chi(t) \rangle \\
&= \langle -{\mathcal V}(x,D) \delta(t-x\cdot \omega),  \varphi(x)  \chi(t) \rangle .
\end{align*}
Hence $v$ satisfies \eqref{eq:lst}, which proves the claim.
\end{proof}

The following proposition gives in the self-adjoint case the  properties of the resolvent that we require to apply the previous Proposition. Therefore we assume that ${\mathcal V}(x,D)$ can be written as ${\mathcal V}(x,D) = 2 \A \cdot D+ D\cdot  \A+ q$ for real $ \A$ and $q$. 
\begin{Proposition} \sl  \label{prop:holomorph _contin_res}
Let $ \A\in C^1_c(\R^n, \R^n)$ and $q\in C^1_c(\R^n, \R)$. For any $\lambda \in \C_+\setminus i(0, r_0]$, there is a bounded operator
\[
R_{ \A,q}(\lambda): L^2(\R^n) \to L^2(\R^n)
\]
such that for any $f\in L^2(\R^n)$, the function $u=R_{ \A,q}(\lambda)f$ is the unique solution in $L^2(\R^n)$ of 
\[
(H_{ \A,q}-\lambda^2)u=f.
\]
Moreover, for  $r_0 = \max_{k=1,2}(2\norm{ \A}_{L^\infty}^2 + \norm{q}_{L^\infty})^{1/2}$, if $r>r_0$ one has
\begin{equation} \label{eq:est_r_r0}
\left\| R_{ \A,q}(\lambda) \right\|_{L^2\to L^2} \leq C_{r,  \A,q}, \quad \Im (\lambda)\geq r.
\end{equation}
For any $\delta >1 /2$ and $\lambda$ in the region $ \C_+\setminus i(0, r_0]$, the family
\[
\left( \br{x}^{-\delta}  R_{r,  \A,q}(\lambda) \br{x}^{-\delta} \right)_{\lambda \in  \C_+\setminus (0, r_0]}
\]
is a holomorphic family of bounded operators on $L^2(\R^n)$ that can be extended continuously in the weak operator topology to $\overline{\C}_+\setminus i[0, r_0]$.
\end{Proposition}

\noindent  Let $\lambda \ge 0$. As mentioned in the introduction, the direct problem  \eqref{eq:sta_sct}  will have a unique scattering solution  $\psi^s_{\mathcal V}$ satisfying the $SRC$ if the {\it outgoing} resolvent operator
\begin{equation*}
(H_{\mathcal V} -(\lambda^2 + i0) )^{-1} = \lim_{\varepsilon \to 0} (H_{\mathcal V} -(\lambda^2 + i\varepsilon) )^{-1} ,
\end{equation*}
is bounded in appropriate spaces. Then one can take
\[   
\psi^s_{\mathcal V} =  (H_{\mathcal V} -(\lambda^2 + i0)^2)^{-1}(-{\mathcal V}(x,D) e^{i\lambda \omega \cdot x}),
\]
as solution of  \eqref{eq:sta_sct}. 
Under the assumptions of the previous proposition, for  real $\lambda >0$ the operator $R_{  \A,q}(\lambda)$  given by the continuous extension to $\overline{\C}_+\setminus i[0, r_0]$ of the resolvent is exactly the outgoing resolvent  operator. On the other hand, if $\lambda <0 $ the resolvent $R_{  \A,q}(\lambda)$ is the operator   known as the {\it incoming} resolvent operator. This can be seen in the following computation: for real $\lambda \neq 0$,  taking the limits in an appropriate topology, one has 
\begin{multline*}
 R_{  \A,q}(\lambda) = \lim_{\varepsilon \to 0} R_{  \A,q}(\lambda +i\varepsilon) = \lim_{\varepsilon \to 0} (H_{ \A,q} -(\lambda +i\varepsilon)^2)^{-1} \\ 
 =  \lim_{\varepsilon \to 0} (H_{ \A,q} -\lambda^2 -i2\lambda \varepsilon -\varepsilon^2 )^{-1} = (H_{ \A,q} -(\lambda^2 \pm i0))^{-1},
 \end{multline*}
 where the $\pm$ is given by the sign of $\lambda$. 
\begin{proof}[Proof of Proposition $\ref{prop:holomorph _contin_res}$]
We have that 
\[ H_{ \A,q} = (D+ \A)^2 +q = -\Delta + {\mathcal V}(x,D),\]
where ${\mathcal V}(x,D) = 2 \A \cdot D+ D\cdot  \A+ q$.
The operator $H_{ \A,q}$ is self-adjoint with domain $H^2(\R^n)$ and, as such, it has real spectrum: the resolvent $R_{ \A,q}(\lambda)$ is a bounded operator in $L^2$ if $\lambda \in \C_+$ satisfies $\lambda^2 \notin \R$. Also, since $ \A$ and  $q$ are compactly supported---${\mathcal V}(x,D)$ is a short range perturbation of $-\Delta$---it is well known that  the continuous spectrum of $H_{ \A,q}$ is $(0,\infty)$ without embedded eigenvalues (see, for example, \cite[Chapter 14]{Hormander}). 
We now show that the point spectrum of $H_{ \A,q}$ is contained in $[-r_0^2,0]$ where $r_0^2= 2\norm{ \A}_{L^\infty}^2+ \norm{q}_{L^\infty}$. Indeed, assume that $\lambda^2 \in \R$ and $\psi \in L^2$ are such that 
\[
H_{ \A,q}\psi = \lambda^2 \psi,
\]
is satisfied in the sense of distributions. Then $\psi \in H^2(\R^n)$ by elliptic regularity, so taking the $ L^2$ product with ${\psi}$ and  integrating by parts gives us
\begin{multline*} 
\lambda^2 \norm{\psi}^2= \norm{\nabla \psi}^2 + (  \A\cdot D\psi, \psi )_{L^2} + (  \A\psi, D\psi )_{L^2} + ( ( \A^2+q)\psi, \psi )_{L^2}\\
 \geq \norm{\nabla \psi}^2 - 2 \norm{ \A}_{L^\infty} \norm{\psi}\norm{\nabla \psi} - \norm{q}_{L^\infty} \norm{\psi}^2  \geq \frac{1}{2}\norm{\nabla \psi}^2 - (2\norm{ \A}_{L^\infty}^2 + \norm{q}_{L^\infty})\norm{\psi}^2,
 \end{multline*}
where we have used that $ \A^2 \ge 0$ and  Young's inequality with $\eps$. Hence
\[
(\lambda^2 +2\norm{ \A}_{L^\infty}^2 + \norm{q}_{L^\infty})\norm{\psi}^2 \geq \frac{1}{2}\norm{\nabla \psi}^2 \geq 0,
\]
and thus  we must necessarily have $\lambda^2 \ge -(2\norm{ \A}_{L^\infty}^2 + \norm{q}_{L^\infty})$. 
With this we can conclude that the full spectrum of $H_{ \A,q}$ is contained in $[-r_0^2,\infty)$, so that $R_{ \A,q}(\lambda)$ is a bounded operator in $L^2$ for all $\lambda \in \C_+ \setminus i(0, r_0]$.  Then the theory of self-adjoint operators implies two important facts. 

The first is that one has the estimate
\begin{equation} \label{eq:self_ad_est}
\norm{R_{ \A,q}(\lambda)}_{L^2 \to L^2} \le \frac{1}{\dist (\lambda^2,[-r_0^2,\infty) )} ,
\end{equation}
(see, for example, \cite[Theorem 2.15]{teschl}).  And the second is that $R_{ \A,q}(\lambda): L^2 \to H^2$ is an holomorphic family of operators for $\lambda \in \C_+$. This last statement follows from the fact that, outside the spectrum, for all $\lambda,\lambda _0 \in \C_+ \setminus i(0, r_0]$ one has the resolvent formula 
\begin{equation*} 
 R_{ \A,q}(\lambda) = R_{ \A,q}(\lambda_0) \left(\sum_{j=0}^{m} (\lambda^2-\lambda^2_0)^j R_{ \A,q}^{j}(\lambda_0) + (\lambda^2-\lambda^2_0)^{m+1} R_{ \A,q}^{m}(\lambda_0)R_{ \A,q}(\lambda)\right),
\end{equation*}
see for example \cite[p. 75]{teschl}. One can take the limit $m \to \infty$ in the $L^2 \to L^2$ operator norm to obtain an analytic expansion of the resolvent around $\lambda_0$, since the remainder of the series goes to zero if $\lambda$ is close enough to $\lambda_0$. Then,
\begin{equation} \label{eq:analytic_ex}
 R_{ \A,q}(\lambda) = R_{ \A,q}(\lambda_0) \sum_{j=0}^{\infty} (\lambda^2-\lambda^2_0)^j R_{ \A,q}^{j}(\lambda_0),
\end{equation}
for $\lambda$ close enough to $\lambda_0$.  Since $R_{ \A,q}(\lambda)$ is also bounded from $L^2$ to $H^2$, \eqref{eq:analytic_ex} implies that $   R_{ \A,q}(\lambda)  : L^2 \to H^2$  is holomorphic in $\C_+ \setminus i(0, r_0]$, and therefore so it is
\[  \br{x}^{-\delta}  R_{ \A,q}(\lambda)  \br{x}^{-\delta} : L^2 \to H^2 \]
since the weight $\br{x}^{-\delta}$ and all its derivatives are bounded in $\R^n$.

We now prove \eqref{eq:est_r_r0} using \eqref{eq:self_ad_est}.  To see this write $\lambda = \sigma + i \mu$. Then it is enough to use that $\dist (\lambda^2,[-r_0^2,\infty) ) \ge |\Im(\lambda^2)| \ge  2|\sigma \mu|$ when $\mu^2 \ge \sigma^2-r_0$, and that
 $\dist (\lambda^2,[-r_0^2,\infty) ) \ge |\Re(\lambda^2)| =  |\sigma^2-\mu^2- r_0|$ otherwise.
 
The continuity  of $ \br{x}^{-\delta}  R_{ \A,q}(\lambda)  \br{x}^{-\delta} $ for  $ \lambda \in \ol{\C}_+ \setminus i[0,r_0]$  it is the well known limiting absorption principle. See for example \cite[Proposition 1.7.1]{Yafaev} for the free resolvent, and  \cite[Chapter 14]{Hormander} for the case of short range magnetic potentials (as in this case). A more specific statement of the limiting absorption principle (also including long range magnetic potentials) can be found in \cite[Theorem 30.2.10]{Hormander}, which  implies  that $R_{ \A,q}(\lambda)$ is continuous as a function from $  \ol{\C}_+ \setminus i[0,r_0]$ to the space  of bounded operators between the Hörmander spaces $B$ and $B^*$ considered with the weak operator topology. Since $B$ is continuously embedded in $ \br{x}^{-\delta} L^2$ this implies the (weak) continuity  of $ \br{x}^{-\delta}  R_{ \A,q}(\lambda)  \br{x}^{-\delta} $ in $\ol{\C}_+ \setminus i[0,r_0]$.
\end{proof}

Putting together Propositions \ref{prop:equivalence_data} and \ref{prop:holomorph _contin_res} we can now formalize the heuristic argument given at the beginning  of this section in order to prove Theorem \ref{thm:equivalence}.

\begin{proof}[Proof of Theorem \ref{thm:equivalence}]
Let $r_0 = \max_{k=1,2}(2\norm{\A_k}_{L^\infty}^2 + \norm{q_k}_{L^\infty})^{1/2}$. 
By Proposition \ref{prop:holomorph _contin_res}, for all $\lambda \in \overline{\C}_+ \setminus i[0,r_0]$ we can define
\[
\psi_{\A_k,q_k}^s(\cdot,\lambda,\omega) =- R_{\A_k,q_k} ((\A_k^2 + D \cdot \A_k + 2\lambda \omega \cdot A + q_k )e^{i\lambda x \cdot \omega} ).
\] 
Assume first that $a_{\A_1,q_1}(\lambda,\theta,\omega) = a_{\A_2,q_2}(\lambda,\theta,\omega)$ for all $\lambda \in \R$ such that $\lambda \ge \lambda_0$ and all $\theta \in S^{n-1}$. Recall that $a_{\A_k,q_k}(\lambda,\theta,\omega)$ are defined by the asymptotic expansion
\[ 
\psi^s _{\A_k,q_k}(r\theta,\lambda,\omega) = e^{i\lambda r} r^{-\frac{n-1}{2}} a_{\A_k,q_k}(\lambda,\theta,\omega)+o(r^{-\frac{n-1}{2}}), \qquad r\to \infty \; \; k=1,2 .
\] 
Therefore, since $\A_k$ and $q_k$ are supported in ${B}$,  for any fixed $\lambda \ge \lambda_0$ the function $\psi_{\A_1,q_1} -\psi_{\A_2,q_2}$ satisfies
\begin{align*}
(-\Delta - \lambda)(\psi_{\A_1,q_1} -\psi_{\A_2,q_2}) (\cdot,\lambda,\omega) &= 0 \quad \text{in} \; \;  \R^n \setminus \overline{B} \\
(\psi_{\A_1,q_1} -\psi_{\A_2,q_2}) (x,\lambda,\omega) &= o(|x|^{-\frac{n-1}{2}}) \quad \text{as}  \; \; |x|\to \infty.
\end{align*}
The Rellich uniqueness theorem  implies that $\psi_{\A_1,q_1} -\psi_{\A_2,q_2}$ vanishes outside $\overline{B}$. In particular, for any $\varphi \in C^\infty_c(\R^n\setminus \ol{B})$, the function
\begin{equation} \label{eq:not_w}
 w_{\varphi}|_{[\lambda_0,\infty)}(\lambda) = \langle (\psi_{\A_1,q_1} -\psi_{\A_2,q_2})(\cdot,\lambda,\omega) ,\varphi\rangle_{\R^n_x} ,
\end{equation}
satisfies
\[
w_{\varphi}|_{[\lambda_0,\infty)} =0.
\]
By Proposition \ref{prop:holomorph _contin_res}, the map $\lambda \to w_{\varphi}$ is holomorphic in ${\C}_+ \setminus i(0,r_0]$ and continuous  in $\overline{\C}_+\setminus i[0,r_0]$. Since it vanishes on $[\lambda_0,\infty)$, we must have $w_\varphi =0$. In particular  
for any $\mu> r_0$ and $\sigma \in \R$ one has  
\[
\langle (\psi_{\A_1,q_1} -\psi_{\A_2,q_2})(x,\sigma + i\mu ,\omega) ,\varphi(x)\rangle = 0.
\]
Then Proposition  \ref{prop:equivalence_data} implies that
\[
\langle (u_{\A_1,q_1} -u_{\A_2,q_2})(x,t ;\omega) ,\varphi(x)\chi(t)\rangle_{\R^n_x \times\R_t} = 0.
\]
for all $\varphi \in C^\infty_c(\R^n\setminus \ol{B})$ and $\chi \in C^\infty_c(\R)$. Therefore 
\[
(u_{\A_1,q_1} -u_{\A_2,q_2})(x,t ;\omega) = 0 \quad (x,t) \in \R^n \times \overline{B}.
\]

Let us prove the converse statement. Let's take the following coordinates in $\R^n$, used in previous sections. If $x \in \R^n$ we write $x= (y,z)$, where $z = x\cdot \omega$ and $y\in \R^{n-1}$.

Assume that $u_{\A_1,q_1}(x,t; \omega) = u_{\A_2,q_2}(x,t; \omega)$ for $(x,t) \in (\partial B \times \R) \cap \{ t \ge z \}$. By Proposition \ref{prop:sol_smooth_2}, the function $\alpha = u_{\A_1,q_1} -u_{\A_2,q_2}$ solves
\begin{equation*}
\square \alpha =0 \qquad \text{in} \;\; \{(x,t): |x|>1 \; \text{and} \; t>z\}, 
\end{equation*} 
and $\alpha|_{(\partial B \times \R) \cap \{ t \ge z \}} = 0 $ and satisfies all the conditions required to apply Lemma \ref{lemma:normal_estimate}. Thus, this lemma yields  that one also has $\partial_\nu \alpha|_{(\partial B \times \R) \cap \{ t \ge z \}} = 0 $. Now, the Cauchy data of $\alpha$ vanishes on the lateral boundary of the set $\{(x,t): |x| \ge 1  \text{ and } t \ge z\}$, and Holmgren's uniqueness theorem   applied in this set shows that $\alpha$ is identically zero in the relevant domain of dependence. However, by finite speed of propagation the support of $\alpha$ is contained in the same domain of dependence. Thus $\alpha$ is identically zero in $\{(x,t): |x| \ge 1  \text{ and } t \ge z\}$, which implies that 
\[u_{\A_1,q_1}(x,t; \omega) = u_{\A_2,q_2}(x,t; \omega), \qquad (x,t) \in (\R^n \setminus \ol B) \times \R.\]
The relation in Proposition \ref{prop:equivalence_data} gives that for any $\mu > r_0$ and for any $\varphi \in C^\infty_c(\R^n \setminus \ol B)$
\[
 w_\varphi(\sigma +i\mu) =  \langle (\psi_{\A_1,q_1} -\psi_{\A_2,q_2})(x,\sigma +i\mu,\omega) ,\varphi(x)\rangle_{\R^n_x} = 0  \qquad \sigma \in \R,
 \]
using the notation in \eqref{eq:not_w}. A mentioned previously, $w_\varphi$ is holomorphic in ${\C}_+ \setminus i(0,r_0]$ and has a continuous extension to $\ol{\C}_+ \setminus i[0,r_0]$ so, in particular, it follows that 
\[ w_\varphi(\lambda) = \langle (\psi_{\A_1,q_1} -\psi_{\A_2,q_2})(\cdot,\lambda,\omega) ,\varphi\rangle_{\R^n_x} = 0
\]
 for all $\lambda >0$. Thus $(\psi_{\A_1,q_1} -\psi_{\A_2,q_2})(\cdot,\lambda,\omega)$ vanishes outside $\ol{B}$ for any $\lambda >0$. By the asymptotics  given in \eqref{eq:asymp_psi}, this implies that $a_{\A_1,q_1}(\lambda,\theta,\omega)=a_{\A_2,q_2}(\lambda,\theta,\omega)$ for all $\lambda >0 $ and $\theta \in S^{n-1}$. 
\end{proof}

\section{ Some results concerning the wave operator} 
 \label{appendix:wave_eq}

The first part of this section is devoted to the proof of Propositions \ref{prop:sol_smooth_1} and \ref{prop:sol_smooth_2}. In the second part we prove Lemmas \ref{lemma:energy_T}-\ref{lemma:normal_estimate}.

       \subsection{Existence and uniqueness of solutions} \label{subsec:existence}
Here we complete the proof of Proposition \ref{prop:sol_smooth_1} and Proposition \ref{prop:sol_smooth_2}. Let $\omega\in S^{n-1}$. In what follows, we write $\p_z=\omega\cdot \nabla$. We first prove the existence and uniqueness of distributional solutions to
\begin{equation}\label{eq:delta_ex_un}
(\partial_t^2 + L_{\W,V}) U_\delta =0 \; \text{in}\; \R^{n+1}, \quad U_\delta|_{\left\{t<-1\right\}}= \delta (t-x\cdot \omega)
\end{equation}
and
\begin{equation}\label{eq:Heaviside_ex_un}
(\partial_t^2 + L_{\W,V}) U_H =0 \; \text{in}\; \R^{n+1}, \quad U_H|_{\left\{t<-1\right\}}= H (t-x\cdot \omega).
\end{equation}
We start by proving the uniqueness. In both cases, it is reduced to prove that zero is the unique distributional solution to the following homogenous equation
\begin{equation*} 
(\partial_t^2 + L_{\W,V})\, \mathcal{U} =0 \; \text{in}\; \R^{n+1}, \quad \mathcal{U} \, |_{\left\{t<-1\right\}}=0,
\end{equation*}
which is true by \cite[Theorem 23.2.7]{Hormander}. Let us now prove the existence of solutions. The method we shall use is the so-called progressing wave expansion method, see for example \cite[Lemma 1]{shiota} and \cite[Theorem 1]{RakeshUhlmann}. For any $j\geq 0$, define
\[
s_+^j = \left\{\begin{matrix}
 s^j , &s\geq 0, \\ 
 0 , &s<0.
\end{matrix}\right.
\]
Note that $s^0_+=H(s)$ is the unidimensional Heaviside function at $s\in \R$. Let $N\in\mathbb{N}$ and suppose that the solutions to  $(\partial_t^2 + L_{\W,V}) U =0$ have the following expansion
\begin{equation}\label{pr:wave_exp}
U (x,t)= a_{-1}(x) \delta (t-x\cdot \omega) + \sum_{j=0}^N a_j (x)(t- x\cdot \omega)_+^j + R_N(x,t).
\end{equation}
In the case of \eqref{eq:delta_ex_un} the  coefficients $(a_j)_{j=-1}^N$  and the remainder term $R_N$ must satisfy the following initial value conditions  
\[
 R_N|_{t<-1}=0, \quad a_{-1}|_{x\cdot \omega<-1}=1, \quad a_{j}|_{x\cdot \omega<-1}=0, \quad j=0, 1, \dots, N
\]
and in the case of \eqref{eq:Heaviside_ex_un}, the conditions
\[
 R_N|_{t<-1}=0, \quad a_{0}|_{x\cdot \omega<-1}=1, \quad a_{j}|_{x\cdot \omega<-1}=0, \quad j=-1, 1, \dots, N.
\]
A straightforward computation shows that the remainder term $R_N$ must satisfy 
\begin{equation}\label{id_recursive_formulae}
\begin{aligned}
(\partial_t^2+ L_{\W,V})R_N(x,t)&= -2\left(( \partial_z- \omega\cdot\W)a_{-1}\right)\partial_t \delta (t-x\cdot \omega)\\
& \quad  \quad- \left(2(\p_z-\omega\cdot\W)a_0+L_{\W,V}a_{-1}\right)\delta(t-x\cdot \omega)\\
&\quad  \quad - \sum_{j=0}^{N-1}\left(2(j+1)( \partial_z- \omega\cdot\W)a_{j+1}+  L_{\W,V} \, a_j \right)(t-x\cdot \omega)^j_+\\
& \quad \quad  -\left( L_{\W,V}\, a_N\right)(t-x\cdot \omega)^N_+.
\end{aligned}
\end{equation}
The task now is to prove the existence of the coefficients $(a_j)_{j=-1}^N$ and $R_N$, satisfying the recursive identity \eqref{id_recursive_formulae} with the corresponding initial value conditions. One expects getting smoother remainder terms $R_N$ as $N$ grows, or at least with better regularity than the Delta distribution and Heaviside function. This can be achieved by killing most the non-smooth terms on the right-hand side of \eqref{id_recursive_formulae}. We now split the proof into two cases depending on the nature of the initial value conditions.\\

\noindent {\it{First case. Existence of solutions of \eqref{eq:delta_ex_un}}}. Above discussion motivates choosing the recursive formulae 
\begin{alignat}{3}
(\partial_z- \omega\cdot\W)a_{-1}&=0,    & \quad &\text{in}\; \R^n,& \quad  a_{-1}|_{x\cdot \omega<-1}&=1,\label{eq:remainder-1}\\
(\partial_z- \omega\cdot\W)a_{0}&=-\frac{1}{2} L_{\W,V}a_{-1} & \quad &\text{in}\; \R^n, & \quad   a_{0}|_{x\cdot \omega<-1}&=0,\nonumber  \\
( \partial_z- \omega\cdot\W)a_{k+1}&=-\frac{1}{2(k+1)} L_{\W,V}\, a_k, &\quad &\text{in}\; \R^n,& \quad  a_{k+1}|_{x\cdot \omega<-1}&=0,\label{eq:remainder0}\\
(\partial_t^2+ L_{\W,V})R_N(x,t)&= -\left( L_{\W,V}\, a_N\right)(t-x\cdot \omega)^N_+, & \quad &\text{in}\; \R^{n+1}, & \quad R_N|_{t<-1}&=0,\label{eq:remainder}
\end{alignat}
where $k=0, 1, \dots, N-1$. By standard ODEs techniques, one can prove that if 
\[
\psi(x)=\int_{-\infty}^0 \omega \cdot\W(x+s\omega)\, ds,
\]
then we have in $\R^n$ for $k=0, 1, \dots,  N-1$: 
\begin{equation}\label{id:coefficients}
\begin{aligned}
    a_{-1}(x)&= e^{\psi(x)}, \\
    a_0(x)&=-\frac{1}{2}e^{\psi(x)}\int_{-\infty}^0 e^{-\psi(x+s\omega)} (L_{\W,V}a_{-1})(x+s\omega)\,ds\\
    &  = -\frac{1}{2} e^{  \psi(x)} \int_{-\infty}^{0} \left [ -\Delta \psi - |\nabla \psi|^2 +  2 \W \cdot \nabla \psi + V \right ](x +s\omega) \, ds,\\
     a_{k+1}(x)&= -\frac{1}{2(k+1)}e^{\psi(x)}\int_{-\infty}^0 e^{-\psi(x+s\omega)} (L_{\W,V}\,a_k)(x+s\omega)\,ds.
\end{aligned}
\end{equation}
It remains to prove the existence of solutions to \eqref{eq:remainder}. To do that, assume that
\[
\W\in C^{\M+2}_c(\R^n) \quad \text{and} \quad  V\in C^{\M}_c(\R^n)
\]
for some $\M\in \mathbb{N}$ large enough which will be fixed later. Note that
\[
\norm{a_{-1}}_{L^\infty}\lesssim \norm{\W}_{C^{\M+2}},
\]
and for $j=0, 1, \ldots, N$
\[
\norm{a_j}_{L^\infty}\lesssim \norm{\W}_{C^{\M-2j}} + \norm{V}_{C^{\M-2j}}.
\]
Setting 
\[
\beta= \min\bk{\M-2N-2, N-1}
\]
we deduce that $\left( L_{\W,V}\, a_N\right)(t-x\cdot \omega)^N_+$ belongs to $C^{\beta}( \mathbb{R}^{n+1})\subset H^{\beta_1}_{loc}(\mathbb{R}; H^{\beta_2}(\mathbb{R}^n))$ with $\beta_j\geq 0$ (both will be fixed later) and $\beta_1+\beta_2=\beta$. By \cite[Theorems 9.3.1 and 9.3.2]{H76}, there exists a unique solution $R_N\in H^{\beta_1+1}_{loc}(\mathbb{R}; H^{\beta_2}(\mathbb{R}^n))$ to  \eqref{eq:remainder} such that for any given $T>-1$ we have
\[
\left\|R_N\right\|_{H^{\beta_1+1}((-1, T]; H^{\beta_2}(\mathbb{R}^n))}\lesssim \left\|\left( L_{\W,V}\, a_N\right)(t-x\cdot \omega)^N_+\right\|_{H^{\beta_1}((-1, T]; H^{\beta_2}(\mathbb{R}^n))}.
\]
We claim that $R_N\in C^2(\R^{n+1})$ by suitably choosing the parameters  $\M$, $N$, $\beta_1$, and $\beta_2$. This regularity is needed to apply, for instance,  the Carleman estimate with boundary terms in Proposition \ref{prop:carleman_estimate} and the estimate in Lemma \ref{lemma:w+w-}. By \cite[Theorem B.2.8/Vol III]{Hormander}, this follows if for instance
\[
\frac{n+3}{2}< \beta_1 + \beta_2=\beta \quad \text{and} \quad \frac{3}{2}<\beta_1,
\]
and furthermore
\begin{equation}\label{sob_emb} 
\begin{aligned}
\norm{R_N}_{C^2((-1,T]\times \R^n)} & \lesssim \left\|R_N\right\|_{H^{\beta_1+1}((-1, T]; H^{\beta_2}(\mathbb{R}^n))}\\
& \lesssim \left\|\left( L_{\W,V}\, a_N\right)(t-x\cdot \omega)^N_+\right\|_{H^{\beta_1}((-1, T]; H^{\beta_2}(\mathbb{R}^n))}.
\end{aligned}
\end{equation}

Equating the two parameters involved in the definition of $\beta$, that is, $\M-2N-2=N-1$; we choose $\M=3N+1$, and hence $\beta=N-1$. We distinguish two cases:
\begin{itemize}
\item When $n$ is even  we consider 
\[
N =\frac{n+6}{2}, \quad \M=\frac{3}{2}n +10, \quad  \beta_1=2, \quad \beta_2=\frac{n}{2}. 
\]
\item When $n$ is odd we consider 
\[
N =\frac{n+7}{2}, \quad \M=\frac{3}{2}(n+1) +10, \quad  \beta_1=2, \quad \beta_2=\frac{n+1}{2}. 
\]
\end{itemize}
The desired claim is proved by combining above choices with \eqref{sob_emb}. On the other hand, by \eqref{pr:wave_exp} and \eqref{id:coefficients}, we deduce that $U_\delta$ can be written as follows
\[
U_\delta(x,t)= e^{\psi(x)} \delta(t-x\cdot \omega) + v(x,t)H(t-x\cdot \omega),
\]
where
\[
v(x,t)=  \sum_{j=0}^N a_j (x)(t- x\cdot \omega)_+^j + R_N(x,t) 
\]
is of class $C^2$  in the region $\bk{t\geq x\cdot \omega}$.  This   shows that $v$ satisfies all the properties stated in Proposition \ref{prop:sol_smooth_2}, and hence the proof of existence and uniqueness of solutions to \eqref{eq:delta_ex_un} is completed. \\

\noindent {\it{Second case. Existence of solutions of \eqref{eq:Heaviside_ex_un}}}. This case is quite similar to the previous one, and hence we only give a brief explanation of the proof. Here we have $a_{-1}=0$ and according to identity \eqref{id_recursive_formulae}, the function $a_0$ must satisfy \eqref{eq:remainder-1} in place of $a_{-1}$. The remaining coefficients $a_k$ satisfy \eqref{eq:remainder0}. The remainder term $R_N$ satisfies the same equation as in \eqref{eq:remainder}, and thus its existence is guaranteed by  \cite[Theorems 9.3.1 and 9.3.2]{H76}. In this case, the solution $U_H$ to \eqref{eq:Heaviside_ex_un} has the form
\[
U_H(x, t)=u(x,t) H(t- x\cdot \omega), \quad u(x,t)=  \sum_{j=0}^N a_j (x)(t- x\cdot \omega)_+^j + R_N(x,t),
\]
where $u$ is of class $C^2$ in the region $\bk{t\geq x\cdot \omega}$. Moreover, it satisfies all the desired properties stated in Proposition \ref{prop:sol_smooth_1}.

\subsection{Proof of the energy lemmas} \label{subsec:energy}

The proofs of the following lemmas are mainly based on standard multiplier techniques. Here $\text{div}_{\mathfrak{m}}$ and $\nabla_{\mathfrak{m}}$ stand respectively  for the divergence and gradient operators with respect to a set of variables $\mathfrak{m}$; while $\nabla$ is reserved for the gradient operator with respect to the full spatial variables, that is $\nabla = \nabla_x$. Let $r=|x|$. We define the radial and angular derivatives as 
\[
\p_r = \frac{x}{|x|}\cdot \nabla, \quad \quad \Omega_{j k}=x_j \partial_k - x_k \p_j, \quad j, k=1, 2, \ldots, n. 
\]
Note that
\begin{equation*}
\abs{\nabla \alpha (x)}^2 = \abs{\p_r \alpha}^2 + \frac{1}{2r^2} \sum_{\underset{i, j=1, 2, \ldots, n}{i \neq j}} \abs{\Omega_{i  j}\alpha (x)}^2.
\end{equation*}
We also adopt the notation: $\alpha_z:= \p_z  \alpha$, $\alpha_t:=\p_t  \alpha$ and $\alpha_r:=\p_r \alpha$. 

\begin{proof}[Proof of Lemma \ref{lemma:normal_estimate}]
This result is quite similar to Lemma  \cite[Lemma 3.3]{RakeshSalo1}. Define
\[
H_T= \left\{ (y,z,t): \abs{(y,z)}\geq 1, \;  -T\leq t\leq T,\;  t\geq z \right\}.
\]
The following identities will be useful in our computations
\begin{align*}
2\alpha_t \,\square\, \alpha&= {\text{div}}_{x,t} \left(-2 \alpha_t \nabla \alpha, \, \alpha_t^2+ \abs{\nabla \alpha}^2\right),\\
2(x\cdot \nabla \alpha)\, \square \,\alpha&=  {\text{div}}_{x,t}  \left( x(\abs{\nabla\alpha}^2- \alpha_t^2)-2(x\cdot \nabla \alpha)\nabla \alpha,  2\alpha_t (x\cdot \nabla \alpha)  \right)+ n\alpha_t^2- (n-2)\abs{\nabla \alpha}^2.
\end{align*}

Thanks to the domain of dependence theorem for the wave equation $\square\, \alpha=0$ in $H_T$, see for instance \cite[Section 2.4, Theorem 6]{Evans1997}, we deduce that $\alpha$ is compactly supported for each fixed $t$. In particular, for each fixed $t$, one integral involving $\alpha$ on $\abs{x}\geq 1$ is actually on $M\geq \abs{x}\geq 1$ for a large enough $M>1$. This fact will be used several times throughout this proof. 

On the one hand, integrating the first identity over the region $H_T\cap \left\{ t\leq \tau \right\}$ for any fixed $\tau \in [-T, T]$, and combining Stoke's theorem with the third equation in \eqref{eq:alpha}, we obtain
\begin{equation}\label{id:different_normal}
\begin{aligned}
&0= \int_{\p(H_T\cap \left\{ t\leq \tau \right\})} \nu(y,z,t)\cdot  \left(-2 \alpha_t \nabla \alpha, \alpha_t^2+ \abs{\nabla \alpha}^2\right) dS\\
& =  \int_{H_T\cap \left\{ t=\tau \right\}} \left(\alpha_t^2+\abs{\nabla \alpha}^2\right)dx  +2 \int_{\Sigma_+\cap \left\{ t\leq \tau \right\}} \alpha_r \alpha_t \,dS\\
&\quad  -\frac{1}{\sqrt{2}}  \int_{H_T\cap \left\{ t=z \right\}} \left( \alpha_t^2+ \abs{\nabla \alpha}^2+ 2\alpha_t\alpha_z\right)dy dz .
\end{aligned}
\end{equation}
We have used that the unit normal vectors $\nu(y,z,t)$ are respectively equal to $(0, 0,1)$, $\frac{1}{\sqrt{2}}(0,1,-1)$ and $-\frac{1}{\abs{x}}(y,z,0)$ on the regions $H_T\cap \left\{ t=\tau \right\}$, $H_T\cap \left\{ t=z \right\}$ and $\Sigma_+\cap \left\{ t\leq \tau \right\}$. A domain of dependence argument shows that  $H_T\cap \supp \alpha$ is bounded and far away from the origin, thus $1/\abs{x}$ is well defined on that region. In particular, $\alpha_r$ is well defined on $H_T\cap \supp \alpha$.

On the other hand, note that $\alpha_t^2+ \abs{\nabla \alpha}^2+ 2\alpha_t\alpha_z= \abs{\nabla_y \alpha}^2 + (Z\alpha)^2$ and $Z\alpha=0$ on $H_T\cap \left\{ t=z\right\}$. Combining these facts with Young's inequality applied with $\epsilon>0$, we get from identity \eqref{id:different_normal} that 
\begin{equation}\label{eq:first_useful}
\begin{aligned}
 \int_{H_T\cap \left\{ t=\tau \right\}} \left(\alpha_t^2+\abs{\nabla \alpha}^2\right)dx& \leq \epsilon \int_{\Sigma_+} \alpha_r^2 \,dS + \epsilon^{-1}\int_{\Sigma_+} \alpha_t^2 \,dS \\
 &\quad  + \frac{1}{\sqrt{2}}\int_{H_T\cap \left\{ t=z \right\}} \abs{\nabla_y\,\alpha (y,z,z )}^2 dydz.
 \end{aligned}
\end{equation}
In a similar fashion, now integrating the second useful identity over $H_T$, we obtain
\begin{equation}\label{eq:second_useful}
\begin{aligned}
& \int_{\Sigma_+} \left(2 \alpha_r^2 +\alpha_t^2-\abs{\nabla \alpha}^2\right)dS= \int_{H_T}\left( (n-2)\abs{\nabla \alpha}^2 - n\alpha_t^2\right) dxdt \\
 & \qquad  \qquad  \qquad -2 \int_{\left\{\abs{x}\geq 1\right\}\cap \left\{t=T\right\}}\alpha_t (x\cdot \nabla \alpha) \, dx-\frac{1}{\sqrt{2}} \int_{H_T\cap \left\{ t=z\right\}}z (\abs{\nabla \alpha}^2 - \alpha_t^2)    \\
 &  \qquad  \qquad  \qquad = \int_{H_T}\left( (n-2)\abs{\nabla \alpha}^2 - n\alpha_t^2\right) dxdt \\
 & \qquad  \qquad  \qquad -2 \int_{\left\{\abs{x}\geq 1\right\}\cap \left\{t=T\right\}}\alpha_t (x\cdot \nabla \alpha)\, dx-\frac{1}{\sqrt{2}} \int_{H_T\cap \left\{ t=z\right\}}z   \abs{\nabla_y \alpha (y,z, z)}^2 dy dz,
 \end{aligned}
\end{equation}
where we used that $\abs{(y,z)}=1$ on $\Sigma_+$ and $\alpha_t^2=\alpha_z^2$ on $H_T\cap \left\{ t=z\right\}$.

Integrating inequality \eqref{eq:first_useful} in time from $-T$ to $T$, and using the resultanting estimate into \eqref{eq:second_useful}, we deduce
\begin{align*}
\int_{\Sigma_+} (\alpha_r^2 + \alpha_t^2)\, dS& \lesssim \abs{ \int_{\Sigma_+} \left(\abs{\nabla \alpha(x)}^2-\alpha_r^2\right)dS} + \int_{-T}^T\int_{H_T\cap \left\{t=\tau \right\}} \left(\alpha_t^2 + \abs{\nabla \alpha}^2 \right) dxd\tau \\
&\quad +\int_{H_T \cap \left\{ t=T \right\}} \left( \alpha_t^2+ \abs{\nabla \alpha}^2\right)dx+  \int_{H_T\cap \left\{ t=z\right\}} \abs{\nabla_y \alpha (y, z, z)}^2 dy dz\\
 & \leq   \sum_{j<k}   \int_{\Sigma_+}\abs{\Omega_{jk}\alpha(x)}^2\, dS + (1+ (2T+1)\epsilon^{-1})\int_{\Sigma_+}  \alpha_t^2\, dS\\
 &\quad  + (2T+1)\epsilon \int_{\Sigma_+}  \alpha_r^2\, dS+ (\sqrt{2}(T+1)+1)  \int_{H_T\cap \left\{ t=z\right\}} \abs{\nabla_y \alpha (y,z,z)}^2 dy dz.
\end{align*}
Here we have also used \eqref{eq:first_useful} with $\tau=T$ to bound the integral on $H_T\cap \bk{t=T}$. Choosing $\epsilon$ small enough, we obtain
\begin{equation}\label{eq:radial_estimate}
\int_{\Sigma_+} \alpha_r^2 \, dS \lesssim \norm{\alpha}^2_{H^1(\Sigma_+)}+  \int_{H_T\cap \left\{ t=z\right\}} \abs{\nabla_y \alpha (y,z, z)}^2 dy dz.
\end{equation}
The last term on the right can be bounded  using \eqref{eq:alpha_t=z}. Indeed 
\begin{equation*}
\begin{aligned}
\int_{H_T\cap \left\{ t=z\right\}} \abs{\nabla_y \alpha (y,z, z)}^2 dy dz& \lesssim  \int_{\abs{y}\leq 1}\int_{\sqrt{1-\abs{y}^2}\leq z \leq T}\abs{\nabla_y \alpha (y, z, z))}^2 dzdy\\
&  = \int_{\abs{y}\leq 1}\int_{\sqrt{1-\abs{y}^2}\leq z \leq T}\abs{\nabla_y \beta (y)}^2 dzdy \\
&\leq  T \varepsilon^{-1} \int_{\abs{y}\leq 1} \frac{1}{\sqrt{1-\abs{y}^2}}\abs{ {\sqrt{1-\abs{y}^2}} \nabla_y \beta (y)}^2 dy,
\end{aligned}
\end{equation*}
where we used that $\beta$ vanishes on $\abs{y}\geq 1-\varepsilon$. Since $\beta$ is independent of $z$, and the tangential part of the gradient to the unit sphere is $z\nabla_y- y\p_z$, we have
\[
\int_{H_T\cap \left\{ t=z\right\}} \abs{\nabla_y \alpha (y,z, z)}^2 dy dz \lesssim \varepsilon^{-1}  \int_{\abs{y}\leq 1}  \abs{ (z \nabla_y - y\nabla_z) \alpha (y,z,z)}^2 dS \leq  \varepsilon^{-1} \norm{\alpha}^2_{H^1(\Sigma_+ \cap \left\{t=z \right\})}.
\]
This estimate and \eqref{eq:radial_estimate} implies
\[
\int_{\Sigma_+} \alpha_r^2 \, dS \lesssim   \varepsilon^{-1} \left(\norm{\alpha}^2_{H^1(\Sigma_+)}+  \norm{\alpha}^2_{H^1(\Sigma_+\cap \Gamma)}\right).
\]
Let $\chi \in C^{1}(\overline{Q})$. The chain rule shows that it also holds for $\chi \alpha$ in place of $\alpha$, where the implicit constant is proportional to $\norm{\chi}_{C^1(\overline{Q})}$. This finishes the proof. 
\end{proof}

\begin{proof}[Proof of Lemma \ref{lemma:energy_T}] Following the proof of \cite[Lemma 3.4]{RakeshSalo1}, we shall prove the estimate
 \[
 \int_B (\abs{\nabla_{x,t}\alpha}^2 +\abs{\alpha}^2)(x, \tau)dx \lesssim \norm{\alpha}^2_{H^1(\Gamma)} + \norm{(\square+ 2\W\cdot \nabla + V)\alpha}^2_{L^2(Q_+)} + \norm{\alpha}^2_{H^1(\Sigma_+)}  + \norm{\p_\nu \alpha}^2_{L^2(\Sigma_+)}  
 \]
 for every $\tau\in (-1,T]$. Although the desired estimate is just the case $\tau=T$, our proof involves Gronwall's inequality, so we need to consider the left-hand side information when $\tau$ is also far way from $T$. Due to this, we split the energy level into two cases depending on  whether or not  the intersection of $t=\tau$ with $t=z$ is inside $Q_+$. When $\tau \in (-1, 1]$, define
  \begin{align*}
 E(\tau)&:= \int_{B\cap \left\{z\leq \tau \right\}} (\alpha^2+ \alpha_t^2+ \abs{\nabla \alpha}^2)(y,z, \tau)dz dy,\\
 J(\tau)&:=  \int_{B\cap \left\{z\leq \tau \right\}}  (\alpha^2+ (Z\alpha)^2 + \abs{\nabla_y \alpha}^2) (y,z,z)dzdy,
 \end{align*}
 and when $\tau\in(1, T]$, define
 \begin{align*}
 E(\tau)&:= \int_{B} (\alpha^2+ \alpha_t^2+ \abs{\nabla \alpha}^2)(y,z, \tau)dz dy,\\
 J&:= J(1).
 \end{align*}
 A straightforward computation shows
 \[
 2\alpha_t (\square \, \alpha + \alpha)= \text{div}_{x,t}(-2\alpha_t \nabla \alpha, \alpha_t^2+\abs{\nabla \alpha}^2+ \alpha^2).
 \]
For any $\tau\in (-1, T]$, integrating this identity over the region $Q_+\cap \left\{ t\leq \tau\right\}$, and using Stoke's theorem, we obtain
\begin{align*}
2\int_{Q_+\cap \left\{ t\leq \tau\right\}} \alpha_t (\square \, \alpha + \alpha)&=  \int_{\p(Q_+\cap \left\{ t\leq \tau \right\})} \nu(y,z,t)\cdot  \left(-2 \alpha_t \nabla \alpha,\alpha_t^2+\abs{\nabla \alpha}^2+ \alpha^2 \right) dS\\
& =  -\int_{Q_+\cap \left\{ t=\tau \right\}} \left(\alpha_t^2+\abs{\nabla \alpha}^2+\alpha^2\right)dx -2 \int_{\Sigma_+\cap \left\{ t\leq \tau \right\}} \alpha_r \alpha_t \,dS\\
&\quad  +\frac{1}{\sqrt{2}}  \int_{Q_+\cap \left\{ t=z \right\}} \left( \alpha_t^2+ \abs{\nabla \alpha}^2+ 2\alpha_t\alpha_z + \alpha^2\right)dy dz.
\end{align*}
Now using Young's inequality with $\epsilon=1$, we get for all $\tau\in (-1, T]$
\begin{align*}
E(\tau)& \lesssim J + \int_{\Sigma_+} (\alpha_t^2 + \alpha_r ^2)dS + \int_{Q_+\cap \left\{ t\leq \tau\right\}} 2|\alpha_t \left((\square +2\W\cdot \nabla +V ) \alpha + (1-V-2\W\cdot \nabla) \alpha\right)|\\
& \lesssim J + \int_{\Sigma_+}( \alpha_t^2 + \alpha_r ^2)dS + \int_{Q_+}\abs{(\square + 2\W\cdot \nabla + V)\alpha}^2 + \int_{Q_+} \alpha^2+ \alpha_t^2 + \abs{\nabla \alpha}^2\\
& \leq J + \int_{\Sigma_+}( \alpha^2 + \abs{\nabla\alpha}^2+\alpha_t^2+\alpha_r^2)dS  + \int_{Q_+}\abs{(\square + 2\W\cdot \nabla + V)\alpha}^2 + \int_0^\tau E(t)dt.
\end{align*}
Note that the implicit constant depends on $\norm{V}_{L^\infty}$ and $\norm{\W}_{L^\infty}$. It follows from the first line in above estimate. Applying Gronwall's inequality, we have for all $\tau\in (-1, T]$
\[
E(\tau) \lesssim J +  \int_{Q_+}\abs{(\square + 2\W\cdot \nabla + V)\alpha}^2 +  \int_{\Sigma_+}( \alpha^2 + \abs{\nabla\alpha}^2+\alpha_t^2+\alpha_r^2)dS.
\]
The proof is now complete by taking $\tau=T$. 
\end{proof}

\begin{proof}[Proof of Lemma \ref{lemma:energy_estimate_gamma}]
As in the proof of \cite[Lemma 3.5]{RakeshSalo1}, we first set $v=e^{\sigma \psi}\alpha $. For any $\tau\in [1,T]$, define the energy corresponding to $t=\tau$ and $t=z$ as follows
 \begin{align*}
 E(\tau)&:= \int_{B} (\sigma^2 v^2+ v_t^2+ \abs{\nabla v}^2)(y,z, \tau)dz dy,\\
 J&:= \int_B (\sigma^2 v^2+ (Zv)^2 + \abs{\nabla_y v}^2) (y,z,z)dzdy.
 \end{align*}
For any $\tau\in [1, T]$, integrating the identity
 \[
 2 v_t (\square \, v + \sigma^2 v)= \text{div}_{x,t}(-2v_t \nabla v, v_t^2+\abs{\nabla v}^2+ \sigma^2 v^2).
 \]
over the region $Q_+\cap \left\{ t\leq \tau\right\}$, and using Stoke's theorem, we obtain

\begin{align*}
2\int_{Q_+\cap \left\{ t\leq \tau\right\}} v_t (\square \, v + \sigma^2 v)&=  \int_{\p(Q_+\cap \left\{ t\leq \tau \right\})} \nu(y,z,t)\cdot  \left(-2v_t \nabla v, v_t^2+\abs{\nabla v}^2+ \sigma^2 v^2 \right) dS\\
& = - \int_{Q_+\cap \left\{ t=\tau \right\}} \left(v_t^2+\abs{\nabla v}^2+\sigma^2 v^2\right)dx -2 \int_{\Sigma_+\cap \left\{ t\leq \tau \right\}} v_r v_t \,dS\\
&\quad  +\frac{1}{\sqrt{2}}  \int_{Q_+\cap \left\{ t=z \right\}} \left( v_t^2+ \abs{\nabla v}^2+ 2v_tv_z + \sigma^2 v^2\right)dy dz.
\end{align*}
Using Young's inequality with $2\sigma v_t(\sigma v)\leq \sigma( \sigma^2 v^2 + v_t^2 )$, we get
\begin{align*}
J& \lesssim E(\tau)+ 2\int_{\Sigma_+ \cap \left\{t\leq \tau \right\}} \abs{v_t  v_r} dS + \int_{Q_+\cap \left\{ t\leq \tau\right\}} 2|v_t \left((\square +2\W\cdot \nabla +V ) v + (\sigma^2-2\W\cdot \nabla-V) v\right)|\\
& \lesssim E(\tau) + 2\int_{\Sigma_+} \abs{v_t  v_r} dS  + \int_{Q_+}2\abs{v_t(\square + 2\W\cdot \nabla + V)v} + \sigma \int_{Q_+} v_t^2 + \abs{\nabla v }^2+ \sigma^2 v^2,
\end{align*}
whenever $\sigma^2\geq 1+ \norm{V}_{L^\infty}+2\norm{\W}_{L^\infty}$. Since $J$ is independent of $\tau$, integrating this estimate in $\tau$ from $1$ to $T$, we have
\begin{equation*}
J \lesssim  \int_{\Sigma_+} \abs{v_t  v_r} dS  + \int_{Q_+}\abs{v_t(\square + 2\W\cdot \nabla + V)v} + \sigma \int_{Q_+} v_t^2 + \abs{\nabla v }^2+ \sigma^2 v^2.
\end{equation*}
The right-hand side can be bounded by integral terms involving weighted versions of $\alpha$, namely, terms given by the right-hand side of \eqref{est:top}. This can be seen by using the following inequalities:
\begin{align*}
&\abs{v}\leq e^{\sigma \psi} \abs{\alpha}, \quad \abs{v_t} + \abs{\nabla v}\leq e^{\sigma\psi} (\abs{\alpha_t} + \abs{\nabla \alpha}+ \sigma \abs{\alpha}),\\
& \abs{(\square + 2\W\cdot \nabla + V) v}\lesssim  e^{\sigma \psi} (\abs{(\square + 2\W\cdot \nabla + V) \alpha} + \sigma (\abs{\alpha_t}+ \abs{\nabla \alpha})+ \sigma^2 \abs{\alpha}),\\
& 2\sigma \abs{\alpha_t} \abs{\alpha} \leq \alpha_t^2 + \sigma^2 \alpha^2, \quad 2\sigma \abs{\nabla\alpha} \abs{\alpha} \leq \abs{\nabla \alpha}^2 + \sigma^2 \alpha^2.
\end{align*}
Finally, we complete the proof by using the estimate below
\[
\int_{B} e^{2\sigma \psi} (\sigma^2 \alpha^2+ (Z\alpha)^2 + \abs{\nabla_y \alpha}^2) (y,z,z) \, dzdy\lesssim  J.
\]
\end{proof}

\bibliographystyle{alpha}

\end{document}